\documentclass[11pt]{amsart}
\usepackage{amsmath}
\usepackage{amssymb}
\usepackage{amscd}
\usepackage{amsthm}
\usepackage{graphicx}
\usepackage{fullpage}
\newtheorem{theorem}{Theorem}[section]
\newtheorem{proposition}[theorem]{Proposition}
\newtheorem{lemma}[theorem]{Lemma}
\newtheorem{claim}[theorem]{Claim}
\newtheorem{corollary}[theorem]{Corollary}

\newtheorem{assumption}[theorem]{Assumption}
\newtheorem{D}[theorem]{Definition}
\newenvironment{definition}{\begin{D} \rm }{\end{D}}
\newtheorem{R}[theorem]{Remark}
\newenvironment{remark}{\begin{R}\rm }{\end{R}}
\newtheorem{E}[theorem]{Example}
\newenvironment{example}{\begin{E}\rm }{\end{E}}

\def\Zee{\mathbb{Z}}

\def\Cee{\mathbb{C}}
\def\Pee{\mathbb{P}}

\def\Ext{\operatorname{Ext}}
\def\Sym{\operatorname{Sym}}

\def\Pic{\operatorname{Pic}}

\def\Spec{\operatorname{Spec}}

\def\scrO{\mathcal{O}}
\def\spcheck{^{\vee}}
\def\hU{\widetilde{U}}
\def\hY{\widetilde{Y}}
\def\hhY{\widetilde{\widetilde{Y}}}

\title{Deformations of I-surfaces with elliptic singularities}

\begin{document}
\author[R. Friedman]{Robert Friedman}
\address{Columbia University, Department of Mathematics, New York, NY 10027}
\email{rf@math.columbia.edu}
\author[P. Griffiths]{Phillip   Griffiths}
\address{Institute for Advanced Study, Princeton, NJ  08540 and IMSA, The University of Miami, Coral Gables, FL 33146}
\email{pg@ias.edu}

\begin{abstract} An I-surface $S$ is an algebraic  surface   of general type with $K_S^2 = 1$ and $p_g(S) = 2$.  Recent research has centered on trying  to give an explicit description of  the KSBA compactification of the moduli space of these surfaces. The possible normal Gorenstein examples have been enumerated by work of Franciosi-Pardini-Rollenske. The goal of this paper is to give a more precise description of such surfaces in case their singularities are simple elliptic and/or cusp singularities, and to work out their deformation theory. In particular, under some mild general position assumptions, we show that deformations of the surfaces in question are versal for deformations of the singular points, with two exceptions where the discrepancy is analyzed in detail. 
\end{abstract}

\bibliographystyle{amsalpha}
\maketitle

\section*{Introduction}

A fundamental problem in the study of algebraic surfaces is to find a good description of their compactified moduli spaces. For $K3$ surfaces, this problem has a long history, dating back to early work of Shah \cite{Shah1}, \cite{Shah2}, and continuing through the very recent work of Alexeev-Engel \cite{AlexeevEngel}. Aside from its intrinsic interest, this study has applications to the asymptotic  behavior of the period map and to various related issues such as Torelli-type questions. A common theme throughout has been the interplay between certain elliptic singularities, notably simple elliptic and cusp singularities, and normal crossing models with trivial dualizing sheaf. In the case of simple elliptic singularities, there is a good understanding of how these two models are related (cf.\ for example M\'erindol \cite{Mer} and Grojnowski  and Shepherd-Barron \cite{GSB}), but the relationship in the cusp case is still a mystery. Controlling the deformation theory of the singular models is also important, in order to describe the local structure of the compactified moduli space,  to understand how the various boundary strata   are related, and to   analyze the limiting behavior of   the period map. In general, however, the deformation theory of the elliptic singularity model should contain more information than that of the normal crossing model. 

For surfaces of general type there is a very general compactification of the moduli space, the KSBA compactification. However, at present it is not practical to describe the myriad possibilities for boundary components in general. Thus, attention has focused in recent years on trying to enumerate all of the possible boundary strata for very simple examples of surfaces $S$ with ``small" canonical class $K_S$; here small could refer either to the size of $K_S^2$ or to that of $\dim H^0(S; K_S)$ or both. An appealing class for surfaces for which both invariants are  small are the \textsl{I-surfaces}:

\begin{definition} An \textsl{I-surface} $S$ is a minimal surface of general type such that $K_S^2=1$ and $\dim H^0(S; K_S) =2$.  In this case, necessarily $H^1(S;\scrO_S) =0$. 
\end{definition} 

I-surfaces were considered classically by Castelnuovo-Enriques. They provided the first examples of surfaces of general type where the $4$-canonical morphism $\varphi_{4K}$ fails to be birational, thus showing that Bombieri's result that $\varphi_{5K}$ is birational is sharp. The modern study of these surfaces goes back at least to Kodaira (unpublished) and Horikawa \cite{Horikawa0}, who showed in particular that $K_S$ has a single base point, necessarily simple, and that $S$ is the minimal resolution of a weighted hypersurface of degree $10$ with at worst rational double points contained in the smooth locus of  the weighted projective space $\Pee(1,1,2,5)$. The analysis was extended to the Gorenstein case by Franciosi-Pardini-Rollenske, who showed in particular \cite{FPR2}:

\begin{theorem}\label{thm02}  Let $Y$ be a normal Gorenstein surface with at worst rational double point,  simple elliptic and cusp singularities such that $\omega_Y$ is ample, $\omega_Y^2=1$, and $\dim H^0(Y; \omega_Y) = 2$. Then $Y$ is isomorphic to a weighted hypersurface of degree $10$ contained in the smooth locus of  the weighted projective space $\Pee(1,1,2,5)$. \qed
\end{theorem}

Let $Y$ be a surface as in the statement of Theorem~\ref{thm02} and let $\pi\colon \hY\to Y$ be its  minimal resolution. Let $p_1, \dots, p_k\in Y $ be the points where $Y$ has an elliptic singularity, and let $D_i$ be the reduced divisor $\pi^{-1}(p_i)$. For each $i$, let $m_i =-D_i^2$. In \cite{FPR1}, \cite{FPR2},  Franciosi-Pardini-Rollenske show the following:

\begin{theorem}\label{thm03} In the above notation, the possibilities for $Y$ and the $m_i$ are as follows:
\begin{enumerate}
\item[\rm(i)] {\rm (Stratum $\mathfrak{N}_{\emptyset}$):} $\hY$ is a minimal surface of general type and $Y$ is its canonical model.
\item[\rm(ii)] {\rm (Stratum $\mathfrak{N}_1$):} $\hY$ is a minimal surface with $\kappa(\hY) =1$ and $k = m_1 =1$.  In this case, $\hY$ is an elliptic surface over $\Pee^1$ with $p_g(\hY) =1$ and   a single multiple fiber, necessarily of multiplicty $2$. The divisor $D = D_1$ consists of a bisection of the elliptic fibration together with some fiber components if $D$ is reducible. 
\item[\rm(iii)] {\rm (Stratum $\mathfrak{N}_2$):} $\hY$ is the blowup of a $K3$ surface at one point, and $k = 1$, $m_1 =2$.
\item[\rm(iv)] {\rm (Stratum $\mathfrak{N}_{1,1}^E$):} $\hY$ is  the blowup of an $Enriques$ surface at one point, and $k = 2$, $m_1 =m_2 =1$.
\item[\rm(v)] {\rm (Strata $\mathfrak{N}_{2,2}$, $\mathfrak{N}_{2,1}$, $\mathfrak{N}_{1,1}^R$):} $\hY$ is  a rational surface, $k = 2$, and the possibilities for $(m_1, m_2)$ are $(2,2)$, $(2,1)$, or $(1,1)$ as given in the subscripts.
\item[\rm(vi)] {\rm (Strata $\mathfrak{N}_{2,1,1}$, $\mathfrak{N}_{1,1,1}$):} $\hY$ is  the blowup of an elliptic ruled surface, $k = 3$, and the possibilities for $(m_1, m_2, m_3)$ are $(2,1,1)$    or $(1,1,1)$ as given in the subscripts. \qed
\end{enumerate}
\end{theorem}

There is also a description in \cite{FPR2} of  some of the incidence relations among the various strata.

The goal of this paper is to study the corresponding deformation theory of I-surfaces with the most basic elliptic singularities, namely simple elliptic or cusp singularities. The ideal result would be one along the following lines (a more precise discussion is given in \S1):

\medskip
\noindent \textbf{Ideal Theorem:}  \textit{For $Y$ as above and $p_1, \dots, p_k$ either simple  elliptic or cusp singularities, the deformations of $Y$ are versal for the local deformations of the singular points except in Case} (vi).  
\medskip 

Here, versal means roughly that, given any collection of    local deformations of the germs of the singular points $p_1, \dots, p_k$, there is a global deformation of the projective surface $Y$ which realizes these deformations. In particular, the singular points can be smoothed or deformed independently. This also explains why Case (vi) is exceptional: in this case, only simple elliptic singularities can arise and the exceptional divisors must all be isogenous. Thus, any deformation of $Y$ which is equisingular at the points $p_1, p_2, p_3$ has to preserve the condition that the exceptional divisors are isogenous. In fact, this is the only constraint on the deformation theory as we shall see. 

The deformations of a Gorenstein I-surface $Y$ are unobstructed. More precisely, it follows easily from the description of $Y$ as a degree $10$ surface in $\Pee(1,1,2,5)$ that the obstruction space $\mathbb{T}^2_Y$ to the deformation space of $Y$ is $0$ (Theorem~\ref{defunobstr}). Given this, the ``Ideal Theorem" would reduce  the study of the discriminant locus in the moduli space of an I-surface with simple elliptic or cusp singularities to the corresponding discriminant locus for the local deformation space of the singular points. For example, in case $Y$ has just one singular point which  is  simple elliptic, there is a neighborhood of $Y$ in the moduli space where the locus where the  surface continues to have a simple elliptic singularity  is a smooth submanifold whose codimension can be read off from the multiplicity of the simple elliptic singularity. Blowing up this submanifold, and   taking a cyclic branched cover along the exceptional divisor, leads to a normal crossing model for the degenerations of the I-surface. Given a theory of semistable models for cusp singularities, a similar picture should hold in the cusp case. 

As stated, the ``Ideal Theorem" is not quite the most ideal version that one could hope for: it would be natural to show that the deformations of $Y$ are also versal for the rational double points, or at least that these can all be smoothed in a family which is equisingular for the elliptic singularities.  However, in order to keep this paper at a reasonable length, we shall ignore this question and in fact make the running assumption that $Y$   has elliptic singularities, but not rational double points (Assumption~\ref{assumption27} below). Even so, we are only able to prove (most of) the ``Ideal Theorem" in a case-by-case fashion and under some mild general position assumptions. More  precise statements are contained in the following: For the case $\kappa(\hY) = 1$, Corollary~\ref{kappa1cor1} and Corollary~\ref{kappa1cor2}. For the case where $\hY$ is the blowup of a $K3$ surface, Theorem~\ref{K3irred} and Theorem~\ref{K3red}.  For the case where $\hY$ is the blowup of an Enriques  surface, Theorem~\ref{Enriquesirred}, Theorem~\ref{Enriquesred}, and Theorem~\ref{Enriquesmixed}. For the case where $\hY$ is a rational surface, Proposition~\ref{22irred} and Corollary~\ref{22cusp} if $(m_1, m_2) = (2,2)$,  Proposition~\ref{21irred} and Proposition~\ref{21red} if   $(m_1, m_2) = (2,1)$, and Proposition~\ref{11irred}  if $(m_1, m_2) = (1,1)$  and $D_1$, $D_2$ are irreducible.  (We do not treat the case where $\hY$ is rational, $(m_1, m_2) = (1,1)$ and either $D_1$ or $D_2$ is reducible.) For the case where $\hY$ is the blowup of an elliptic ruled surface, Theorem~\ref{ellruled}. These results are then used to refine and clarify the discussion of the adjacencies between the various strata given in \cite[Figure 4]{FPR2}, i.e.\ when the closure of one stratum has nonempty intersection with another. For example, using the exotic adjacencies between cusp singularities of degree $2$ and cusp or simple elliptic singularities of degree one due to Brieskorn-Wahl, we show that the closure of $\mathfrak{N}_1$ meets $\mathfrak{N}_2$ (Theorem~\ref{K3red} and Corollary~\ref{K3toell}), the closure of $\mathfrak{N}_{2,1}$ meets $\mathfrak{N}_{2,2}$ (Remark~\ref{exotic22}), and the closure of $\mathfrak{N}_{1,1}^R$ meets $\mathfrak{N}_{2,1}$ (Proposition~\ref{21toR}). In addition, $\mathfrak{N}_{1,1,1}$ is contained in the closure both of $\mathfrak{N}_{1,1}^R$ and $\mathfrak{N}_{1,1}^E$ (Theorem~\ref{ellruled}(v)). 

It seems  likely that special cases of the ``Ideal Theorem," especially when $Y$ has just one singular point which is not overly complicated, should also follow from Theorem~\ref{thm02} by explicit calculations with the equations defining a weighted hypersurface of degree $10$ in $\Pee(1,1,2,5)$. 

The  contents of this paper are as follows. In Section~\ref{section1}, we collect various facts about simple elliptic and cusp singularities and their deformations. In Section~\ref{section2}, we give a quick review of the arguments of \cite{FPR1} used to prove Theorem~\ref{thm03}.  Then we give a more detailed discussion in case $\kappa(\hY) = -\infty$, i.e.\ the cases (v), (vi) of  Theorem~\ref{thm03}. The main point is the following, which generalizes the picture for $\kappa(\hY) \ge 0$: The multiplicities $m_i = -D_i^2$ must be either $1$ or $2$. If $m_i=2$, then $\hY$ is the one point blowup of a surface $X$ such that $-K_X$ is effective. If  $m_i=1$, then there exists a   not necessarily relatively minimal elliptic fibration $\hY\to \Pee^1$, with at least one multiple fiber, such that all multiple fibers have multiplicity $2$ and there is a bisection of the fibration. 

We turn then to the various cases. Section~\ref{section4} deals with the case $\kappa(\hY) = 1$.  The main technical tool is a description via double covers of relatively minimal elliptic surfaces with a bisection and a multiple fiber of multiplicity  2 which essentially goes back to Horikawa \cite{Horikawa}.   Section~\ref{section5} deals with the case $\kappa(\hY) = 0$, and Section~\ref{section6} with the case where $\hY$ is a rational surface. After a preliminary Section~\ref{section7} on elliptic ruled sufaces which are also elliptic surfaces with two or three multiple fibers of multiplicity $2$, Section~\ref{section8} turns to the case where $\hY$ is the blowup of an elliptic ruled surface. In all cases, the main point is typically to prove some kind of vanishing theorem, through a variety of somewhat \emph{ad hoc} techniques. This is in contrast to the case where $\omega_Y$ is trivial or negative, where the corresponding results follow easily from standard methods. It is natural to ask if there is a more unified approach in the I-surface case. 

\subsection*{Acknowledgements} The second author would like to thank Mark Green for the many discussions and correspondences about I-surfaces, especially regarding their Hodge theory.

\subsection*{Notation and conventions} We will generally identify line bundles with the corresponding divisor classes, in particular for the canonical bundle of a surface. Thus we usually   use additive notation for the tensor product of two line bundles where we mean the sum  of the corresponding divisor classes. For two divisors $A$, $B$ on a surface $S$, the notation $A\equiv B$ means: $A$ is numerically equivalent to $B$.

\section{Deformations of  simple elliptic and cusp singularities}\label{section1}

\subsection{Preliminary results on simple elliptic and cusp singularities}

Throughout this section, $U$ is the germ of an isolated Gorenstein surface singularity $p$ and $\pi \colon \hU \to U$ is its minimal resolution. We will also use $U$ to denote a contractible Stein representative of the germ.  Let $D = \bigcup_iE_i$ be the exceptional set $\pi^{-1}(p)$.  The germ $U$  has a \textsl{minimally elliptic singularity} if $R^1\pi_*\scrO_{\hU}$ has length one. Some special minimally elliptic singularities are given as follows: 

\begin{definition} The germ $(U,p)$ is a \textsl{simple elliptic singularity} if $D$ is a smooth elliptic curve.  The germ $(U,p)$ is a \textsl{cusp singularity} if $D=\bigcup_{i=1}^rE_i$ is either an irreducible nodal curve or a cycle of smooth rational curves, i.e.\ $E_i\cdot E_{i+1} = 1$ and $E_r \cdot E_1 =1$, with necessarily $E_i^2 \le -2$ for all $i$ and $E_i^2\le -3$ for at least one $i$. In this case the integer $r$ is the \textsl{length} of the singularity (or of $D$), where $r=1$ if $D$ is irreducible.  
The germ $(U,p)$ is a \textsl{Dolgachev or triangle singularity} if $D=\bigcup_{i=1}^rE_i$ with $r \le 3$, where $D$ is a cuspidal rational curve if $r=1$, two smooth rational curves with a tacnode singularity if $r=2$, and three smooth rational curves meeting at a point if $r=3$.  We will also inaccurately refer to $D$ as a simple elliptic, cusp or triangle singularity.
In all cases the integer $m=-D^2$ is the \textsl{multiplicity} of the  singularity (or of $D$). Here, in the cusp case, with $D=\bigcup_{i=1}^rE_i$ and $ E_i^2=-e_i$, $m= \sum_{i=1}^r(e_i -2)$ if $r\ge 2$. A similar result holds in the triangle case. (Note however that $m$ is the multiplicity in the usual sense only when $m\ge 2$.) We will refer to the sequence $(-e_1, \dots, -e_r)$ as the \textsl{type} of the cusp or triangle singularity; it is well-defined up to cyclic permutation and reflection. 

A simple elliptic singularity is determined  up to analytic isomorphism by the multiplicity $m$ and the $j$-invariant (i.e.\ the isomorphism class) of the elliptic curve $D$.  A cusp singularity is determined  up to analytic isomorphism by its type, i.e.\ by the sequence $(-e_1, \dots, -e_r)$.  This is usually phrased as saying that cusp singularities are \textsl{taut}. A similar result is almost true for triangle singularities: there are two analytic types, depending on whether the singularity does or does not have a $\Cee^*$-action.  
\end{definition} 

\begin{remark} The above singularities are exactly the minimally elliptic singularities with reduced fundamental cycle, which is equivalent in this case to: $K_{\hU}\cong \scrO_{\hU}(-D)$. By the Grauert-Riemenschneider theorem, $R^i\pi_*\scrO_{\hU}(-D) = 0$ for $i > 0$. For a simple elliptic or cusp singularity, this says that $(U,p)$ is a Du Bois singularity \cite[(3.8)]{Steenbrink81}. However, the analogous statement does not hold for a triangle singularity, because the minimal resolution is not a log resolution, i.e.\ the exceptional set $D$ in the minimal resolution has worse than nodal singularities. 

While we will only consider the case of simple elliptic and cusp singularities in this paper, many of the classification results in \S\ref{section2} carry over to the case where some of the singularities are allowed to be triangle singularities. 
\end{remark}

We shall be concerned here with singularities of multiplicity $1$ or $2$. An immediate calculation gives:

\begin{lemma} In the above notation, if $m=1$ or $2$ and $D$ is not irreducible, then the possibilities for the sequence $(-e_1, \dots, -e_r)$, up to cyclic permutation, are as follows:
\begin{enumerate}
\item[\rm(i)] $m=1$: $(-e_1, \dots, -e_r) = (-3, -2, \dots, -2)$.
\item[\rm(ii)]  $m=2$: Either $(-e_1, \dots, -e_r) = (-4, -2, \dots, -2)$ or 
$$(-e_1, \dots, -e_r) =(-3, \underbrace{-2, \cdots, -2}_{a}, -3, \underbrace{-2, \cdots, -2}_{b}).  \qed$$
\end{enumerate}
\end{lemma}

\subsection{Adjacencies of some elliptic singularities}

\begin{definition} An \textsl{adjacency} between two (germs of) isolated singularities $U_1$ and $U_2$ is a deformation (a flat family) of germs $f\colon (\mathcal{U},0) \to (\Delta,0)$, where $\Delta$ is the unit disk, $f^{-1}(0) \cong (U_1,0)$  and $f^{-1}(t) \cong (U_2,0)$ for all $t\neq 0$. In particular,   the restriction of $f$ to the critical locus is proper. Typically, however,  $U_2$ can  be a disconnected germ. In case $U_1$ is a local complete intersection, adjacencies are transitive: if $U_1$ is adjacent to $U_2$ and $U_2$ is adjacent to $U_3$, then $U_1$ is adjacent to $U_3$. In this case, we say that the adjacency from $U_1$ to $U_3$ \textsl{is a consequence of} the adjacency from $U_1$ to $U_2$ and will not always note it explicitly. 
\end{definition}

Adjacencies for all simple elliptic and cusp singularities of multiplicity at most $2$ or more generally which are local complete intersections have been classified by the work of Karras \cite{Karras}, Brieskorn \cite{Brieskorn}, Wahl \cite{Wahl} and Looijenga \cite{Looijenga}: 

\begin{theorem}\label{cuspadj} 
 \begin{enumerate}
\item[\rm(i)] Suppose that $U$ is a simple elliptic singularity of multiplicity $m =1,2$. Then the possible adjacent singularities $U$ are either simple elliptic singularities of multiplicity $m$, rational double points, or smooth germs.
\item[\rm(ii)] The cusp singularity of type $(-3, \underbrace{-2, \dots, -2}_{r-1})$ is adjacent  to the cusp singularity of type $(-3, \underbrace{-2, \dots, -2}_{s-1})$ with $s\le r$,    to a simple elliptic singularity or length one cusp singularity of multiplicity $1$, or to a rational double point or smooth germ.
\item[\rm(iii)] The adjacencies of the  cusp singularity of type $(-4, \underbrace{-2, \dots, -2}_{r-1})$ are a consequence of the following: 
\begin{enumerate} 
\item[\rm(a)] To the cusp singularity of type $(-4, \underbrace{-2, \dots, -2}_{s-1})$ with $s\le r$.
\item[\rm(b)] If $r\ge 4$,    to  the cusp singularity of type $(-3, \underbrace{-2, \dots, -2}_{r-3})$. For $r=3$,   to a simple elliptic singularity or length one cusp singularity of multiplicity $1$, i.e.\  of type $(-1)$. For $r=2$, to a  simple elliptic singularity of multiplicity $1$.
\end{enumerate}
\item[\rm(iv)] The adjacencies of the  cusp singularity  of type $ (-3, \underbrace{-2, \cdots, -2}_{a}, -3, \underbrace{-2, \cdots, -2}_{b})$ are a consequence of the following: 
\begin{enumerate} 
\item[\rm(a)] For $a> 0$, to the  cusp singularity  of type $ (-3, \underbrace{-2, \cdots, -2}_{a-1}, -3, \underbrace{-2, \cdots, -2}_{b})$.
\item[\rm(b)] For $a=0$, the cusp   singularity  of type $ (-3, -3, \underbrace{-2, \cdots, -2}_{b})$ is adjacent to the one  of type   $ (-4,   \underbrace{-2, \cdots, -2}_{b})$ for $b \ge 1$ and the one  of type $(-3, -3)$ is adjacent to the one  of type $(-2)$ and to a simple elliptic singularity  of multiplicity $2$. All further adjacencies are consequences of those given previously. 
\end{enumerate}
\end{enumerate}
\end{theorem}

\begin{remark} The very complicated question as to which (unions of) rational double point singularities are adjacent to a cusp singularity of multiplicity $1$ or $2$ has been answered by  Looijenga \cite{Looijenga}. The much simpler   question as to which (unions of) rational double point singularities are adjacent to a simple elliptic singularity  of multiplicity $1$ or $2$ reduces to a question about the rational double point configurations in generalized del Pezzo surfaces of degree $1$ or $2$.
\end{remark}

\subsection{Deformations of elliptic singularities: the local case}
Let $\pi\colon \hU \to U$ be the minimal resolution of a simple elliptic or cusp singularity. In particular, the divisor $D$ has (local) normal crossings. 
As usual, we denote by $\Omega^1_{\hU}(\log D)$ the sheaf of differential $1$-forms with logarithmic poles along $D$. Let $T_{\hU}(-\log D) = (\Omega^1_{\hU}(\log D))\spcheck$ be the dual sheaf. Then $H^1(\hU; T_{\hU})$ (roughly) classifies first order deformations of the noncompact manifold  $\hU$ and $ H^1(\hU; T_{\hU}(-\log D))$ (roughly) classifies first order deformations of $\hU$ keeping the effective divisors $E_i$. A key technical result is the following:

\begin{lemma}\label{claim0}  \begin{enumerate} \item[\rm(i)] $\dim H^1(\hU; T_{\hU}(-\log D))  =1$ if $D$ is a simple elliptic singularity. 
\item[\rm(ii)] $H^1(\hU; T_{\hU}(-\log D)) =0$ if $D$ is a cusp singularity. 
\end{enumerate}
\end{lemma}
\begin{proof}   There is a perfect pairing 
$$\Omega^1_{\hU}(\log D) \otimes \Omega^1_{\hU}(\log D) \to \Omega^2_{\hU}(\log D) = K_{\hU} \otimes \scrO_{\hU}(D)\cong \scrO_{\hU}.$$
Thus  the dual $T_{\hU}(-\log D)$ of $\Omega^1_{\hU}(\log D)$ is identified with 
$\Omega^1_{\hU}(\log D)$. 
From the Poincar\'e residue sequence
$$0 \to  \Omega^1_{\hU} \to \Omega^1_{\hU}(\log D) \to \bigoplus_i\scrO_{E_i} \to 0$$
and the fact that   $H^2(\hU; \Omega^1_{\hU})  =0$ for dimension reasons,
there is a long exact sequence
$$\bigoplus_iH^0(\scrO_{E_i}) \to H^1(\hU; \Omega^1_{\hU}) \to H^1(\hU; \Omega^1_{\hU}(\log D)) \to \bigoplus_iH^1(\scrO_{E_i}) \to 0.$$
 Note that $\bigoplus_iH^1(\scrO_{E_i}) =0$ if $D$ is a cusp singularity whereas $\bigoplus_iH^1(\scrO_{E_i}) = H^1(D; \scrO_D)$ has dimension one if $D$ is simple elliptic. 
There is an exact sequence (cf.\  \cite[Lemma 2.1(iii)]{FL22d})
$$0 \to \Omega^1_{\hU}(\log D)(-D) \to \Omega^1_{\hU} \to \Omega^1_D/\tau^1_D \to 0,$$
where $\Omega^1_D/\tau^1_D$ is the K\"ahler differentials mod the torsion subcomplex. Since cusp singularities are log canonical, and hence Du Bois, it follows from a result of  Steenbrink \cite[p.\ 1369]{Steenbrink} that
  $$H^1(\hU; \Omega^1_{\hU}(\log D)(-D) ) =0$$  (this is a very special case of   the ``extra vanishing lemma"   \cite[Lemma 2.5]{FL22d}). Moreover,  $H^2(\hU; \Omega^1_{\hU}(\log D)(-D) ) =0$    for dimension reasons.  Thus 
  $$H^1(\hU; \Omega^1_{\hU}) \cong H^1(D; \Omega^1_D/\tau^1_D) \cong   \bigoplus_iH^1(E_i;\Omega^1_{E_i}) \cong H^2(D;\Cee)\cong  H^2(\hU;\Cee).$$  The map $\bigoplus_iH^0(E_i; \scrO_{E_i}) \to H^1(\hU; \Omega^1_{\hU})$ from the Poincar\'e residue sequence is the same as the fundamental class map $\bigoplus_iH^0(E_i) \to \bigoplus_iH^2(D_i)$ and hence is an isomorphism (by negative definiteness). Hence  
$$H^1(\hU; T_{\hU}(-\log D)) \cong H^1(\hU; \Omega^1_{\hU}(\log D))  \cong \bigoplus_iH^1(E_i;\scrO_{E_i})$$ 
and thus has dimension $1$ if $D$ is simple elliptic and is $0$ if $D$ is a cusp.
\end{proof}

We can see a more precise version of Lemma~\ref{claim0}(i) from a slightly different point of view. There is a commutative diagram
$$\begin{CD}
@. 0 @. 0 @. @. @.\\
@. @VVV @VVV @. @. @.\\
@. T_{\hU}(-D) @= T_{\hU}(-D) @. @. @.\\
@. @VVV @VVV @. @. @.\\
0@>>> T_{\hU}(-\log D) @>>> T_{\hU} @>>>\bigoplus_iN_{E_i/\hU} @>>> 0 @.\\
@. @VVV @VVV @VVV @. @.\\
0@>>> T^0_D @>>> T_{\hU}|D @>>>N_{D/\hU} @>>> T^1_D @>>> 0\\
@. @VVV @VVV @. @. @.\\
@. 0 @. 0 @. @. @.
\end{CD}$$
Here $T^0_D$ is the sheaf of derivations of $\scrO_D$ and $H^1(D;T^0_D)$ is the tangent space to locally trivial deformations of $D$. Thus $H^1(D;T^0_D) =0$ in the cusp case and $H^1(D;T^0_D) = H^1(D; T_D)$ has dimension $1$ in the simple elliptic case.

\begin{lemma} In the simple elliptic case, the map $H^1(\hU; T_{\hU}(-\log D)) \to H^1(D; T_D)$ is an isomorphism.
\end{lemma} 
\begin{proof} By Lemma~\ref{claim0}, both $H^1(\hU; T_{\hU}(-\log D))$ and $ H^1(D; T_D)$ have dimension $1$.  The map $H^1(\hU; T_{\hU}(-\log D)) \to H^1(D; T_D)$ is surjective since $H^2(\hU; T_{\hU}(-D)) =0$. Thus it is an isomorphism. 
\end{proof} 

Using the middle row of the above commutative diagram then gives:

\begin{corollary} \begin{enumerate} \item[\rm(i)] In the simple elliptic case, there is an exact sequence
$$0 \to H^1(D; T_D) \to H^1(\hU; T_{\hU}) = H^0(U; R^1\pi_*T_{\hU}) \to H^1(D; N_{D/\hU}) \to 0.$$
\item[\rm(ii)] In the cusp case, 
$$H^1(\hU; T_{\hU}) = H^0(U; R^1\pi_*T_{\hU}) \cong  \bigoplus_iH^1(E_i; N_{E_i/\hU}) . \qed$$
\end{enumerate}
\end{corollary} 
 
\begin{remark}\label{equiremark}  For a simple elliptic singularity, we can consider the \textsl{equisingular} deformations, i.e.\ those for which the singularity remains simple elliptic but the $j$-invariant of the resolution, i.e.\ the isomorphism type, is allowed to vary. (For a cusp singularity, an equisingular deformation is the same thing as the trivial deformation.)  We saw in the course of the proof of Lemma~\ref{claim0} that $H^1(\hU; \Omega^1_{\hU}(\log D)(-D)) = 0$. By local duality and the first line of the proof of Lemma~\ref{claim0}, 
$$0=(H^1(\hU; \Omega^1_{\hU}(\log D)(-D)) )\spcheck = H^1_D(\hU; \Omega^1_{\hU}(\log D)) = H^1_D(\hU; T_{\hU}(-\log D)) .$$ By Wahl's theory \cite[Proposition 2.5, Proposition 2.7]{WahlI}, since $H^1_D(\hU; T_{\hU}(-\log D)) = 0$, there is an inclusion $H^1(\hU; T_{\hU}(-\log D)) \to H^0(U; T^1_U)$ whose image is  identified with  the tangent space to the equisingular deformations of $U$ inside the tangent space to all deformations. 
\end{remark}

\subsection{Deformations of elliptic singularities: the global case}

We now consider the following global issue: Let $Y$ be a compact analytic surface or a projective surface with a collection of simple elliptic and cusp singularities of multiplicity at most $2$, or more generally which are local complete intersections. Then $T^2_Y =0$ and $H^1(Y; T^1_Y) =0$ since the singularities are isolated. The local to global spectral sequence for $\mathbb{T}^i_Y = \Ext^i(\Omega^1_Y, \scrO_Y)$ then reduces to the exact sequence
$$0 \to H^1(Y; T^0_Y) \to  \mathbb{T}^1_Y \to H^0(Y; T^1_Y) \to  H^2(Y; T^0_Y) \to \mathbb{T}^2_Y \to 0.$$
Thus we have:

\begin{lemma}\begin{enumerate}\label{lemma1101} \item[\rm(i)]  If the map $H^0(Y; T^1_Y) \to  H^2(Y; T^0_Y) $ is surjective, then $\mathbb{T}^2_Y=0$.
 \item[\rm(ii)] If $H^2(Y; T^0_Y) =0$, then $\mathbb{T}^2_Y=0$ and the map $\mathbb{T}^1_Y \to H^0(Y; T^1_Y)$ is surjective. \qed
\end{enumerate}
\end{lemma}

This is related to the Leray spectral sequence as follows. Since $\pi\colon \hY \to Y$ is the minimal resolution, it is equivariant and hence $T^0_Y \cong R^0\pi_*T_{\hY}$. Then the Leray spectral sequence $E^{p,q}_2 = H^p(Y; R^q\pi_*T_{\hY}) \implies H^{p+q}(\hY; T_{\hY})$ gives an exact sequence
$$0 \to H^1(Y;T^0_Y) \to  H^1(\hY; T_{\hY}) \to H^0(Y; R^1\pi_*T_{\hY}) \to    H^2(Y;T^0_Y) \to  H^2(\hY; T_{\hY})\to 0.$$

\begin{lemma}\label{lemma1111} \begin{enumerate} \item[\rm(i)]  If $H^2(\hY; T_{\hY}(-\log D)) =0$, then $H^2(\hY; T_{\hY}) =0$.
 \item[\rm(ii)] If $H^2(\hY; T_{\hY}(- D)) =0$, then $H^2(\hY; T_{\hY}(-\log D))  = H^2(\hY; T_{\hY}) =0$ and the map 
 $$H^1(\hY; T_{\hY}(-\log D)) \to H^1(D; T^0_D)$$ is surjective.
\end{enumerate}
\end{lemma}
\begin{proof} We have the  two exact sequences
\begin{gather*}
0 \to T_{\hY}(-D) \to T_{\hY}(-\log D) \to T^0_D \to 0;\\
0 \to T_{\hY}(-\log D) \to T_{\hY} \to \bigoplus_i N_{E_i/\hY} \to 0.
\end{gather*}
Then the  fact that $H^2(D;T^0_D) = \bigoplus_i H^2(E_i; N_{E_i/\hY}) =0$ immediately implies the lemma.
\end{proof}

The following is the key result linking the deformation theory of $Y$ to certain vanishing statements: 

\begin{proposition}\label{prop1121}  \begin{enumerate} \item[\rm(i)] In the simple elliptic case, suppose that $H^2(\hY; T_{\hY}(-D) ) =0$. Then the map   
$$H^1(\hY; T_{\hY}) \to H^0(Y; R^1\pi_*T_{\hY})$$ is surjective and   $H^2(Y; T^0_Y) = \mathbb{T}^2_Y = 0$. Thus $\mathbb{T}^1_Y \to H^0(Y; T^1_Y)$ is surjective. 
\item[\rm(ii)] In the cusp case, suppose that $H^2(\hY; T_{\hY}(-\log D) ) =0$. Then the map   
$$H^1(\hY; T_{\hY}) \to H^0(Y; R^1\pi_*T_{\hY})$$ is surjective and   $H^2(Y; T^0_Y) = \mathbb{T}^2_Y = 0$. Thus $\mathbb{T}^1_Y \to H^0(Y; T^1_Y)$ is surjective. 
\end{enumerate}
\end{proposition}
\begin{proof} In the simple elliptic case, there is a commutative diagram
$$\begin{CD}
@. H^1(\hY; T_{\hY}(-\log D)) @>>> H^1(\hY; T_{\hY}) @>>>  H^1(D;N_{D/\hY}) @. \\
@. @VVV @VVV @VV{=}V @. \\
0 @>>> H^1(D; T_D) @>>>  H^0(Y; R^1\pi_*T_{\hY}) @>>> H^1(D;N_{D/\hY}) @>>> 0.
\end{CD}$$
By Lemma~\ref{lemma1111}, under the assumptions of (i), $H^2(\hY; T_{\hY}(-\log D)) =0$ and hence  $H^1(\hY; T_{\hY}) \to  H^1(D;N_{D/\hY})$ is surjective. Since $H^2(\hY; T_{\hY}(-D) ) =0$, the map $H^1(\hY; T_{\hY}(-\log D)) \to H^1(D; T_D)$ is surjective. Thus $H^1(\hY; T_{\hY}) \to  H^0(Y; R^1\pi_*T_{\hY})$ is surjective. Again by Lemma~\ref{lemma1111}, $H^2(\hY; T_{\hY}) =0$.  The Leray spectral sequence then gives an exact sequenece
$$H^1(\hY; T_{\hY}) \to H^0(Y; R^1\pi_*T_{\hY}) \to    H^2(Y;T^0_Y) \to    0.$$
Since  $H^1(\hY; T_{\hY}) \to H^0(Y; R^1\pi_*T_{\hY})$ is surjective, $H^2(Y;T^0_Y)=0$.  By Lemma~\ref{lemma1101}(ii), $\mathbb{T}^2_Y = 0$ and the map $\mathbb{T}^1_Y \to H^0(Y; T^1_Y)$ is surjective.  The proof in the cusp case is similar (and simpler).
\end{proof}

\begin{example}\label{ex13}  Suppose that $D\in |-K_{\hY}|$ and $D$ is smooth. Then 
$$H^2(\hY; T_{\hY}(-D)) = H^2(\hY; T_{\hY} \otimes K_{\hY}) = H^2(\hY; \Omega^1_{\hY}).$$
It is a standard result (due in various forms  to Kulikov-Persson-Roan) that either 
  $\hY$ is rational and $D$ is a single elliptic curve, or $\hY$ is the blowup of a minimal elliptic ruled surface  and $D = D_1 + D_2$ has two connected components $D_1$ and $D_2$, both sections of the ruling. In case $\hY$ is rational,  $H^2(\hY; \Omega^1_{\hY})=0$. If $\hY$ is the blowup of a minimal elliptic ruled surface, then $\dim H^2(\hY; \Omega^1_{\hY}) = 1$.   In this case, the exact sequence $0 \to T_{\hY}(-D) \to T_{\hY}(-\log D) \to T_D \to 0$ is identified with the Poincar\'e residue  sequence $$0 \to \Omega^1_{\hY} \to \Omega^1_{\hY}(\log (D_1+ D_2)) \to \scrO_{D_1} \oplus \scrO_{D_2} \to 0.$$
The Gysin map $H^1(D_i; \scrO_{D_i}) \to H^2(\hY; \Omega^1_{\hY})$ is an isomorphism for $i=1,2$. Hence the image of $H^1(Y; T_{\hY}(-\log D) )$ in $H^1(D_1; T_{D_1}) \oplus H^1(D_2; T_{D_2})$ has dimension one and $H^2(\hY; T_{\hY}(-\log D) )=0$. Note that $H^2(\hY; T_{\hY}(-D) ) \neq 0$; roughly speaking, this represents the obstruction to deforming the two elliptic curves $D_1$ and $D_2$ independently. However, $H^2(\hY; T_{\hY}) =0$ since this is a birational invariant of $\hY$ and $\hY$ is a ruled surface.  One can then check that  $H^0(Y; T^1_Y) \to  H^2(Y; T^0_Y) $ is surjective. Thus $\mathbb{T}^2_Y =0$ in this case as well. We will discuss this case in more detail below.
\end{example}

\subsection{Some variations of the above} In some of our applications, Proposition~\ref{prop1121} will not hold, and so we will need some variations on it. For the first variation, assume that $Y$ has simple elliptic singularities at the points $p_1, \dots, p_k$ and is otherwise smooth.  We begin by fixing some notation:

\begin{definition}\label{piiprime}
As always, let $\pi\colon \hY\to Y$ be the minimal resolution and let $D_i = \pi^{-1}(p_i)$.  Set $D_i' = \sum_{j\neq i}D_j$. We can form partial resolutions in two different ways: first, let $\pi_i \colon Y_i \to Y$ be a resolution at the singular point $p_i$ and an isomorphism elsewhere.  The exceptional set of $\pi_i$ is $D_i$.  Then there is an  induced morphism $\psi_i\colon \hY \to Y_i$, with exceptional set $D_i'$.  There is also  the morphism $\pi_i'\colon Y_i' \to Y$, which is a resolution at the points $p_j, j\neq i$ and is an isomorphism over $p_i$. Thus the exceptional set of $\pi_i'$ is $D_i'$. Of course, if $k=2$, then $\pi_1\colon Y_1 \to Y$ is identified with $\pi_2'\colon Y_2' \to Y$.  Let  $\psi_i'\colon\hY \to Y_i'$ be the corresponding induced morphism, with exceptional set $D_i$. 
\end{definition} 

We shall look at $\pi_i \colon Y_i \to Y$  more systematically in \S\ref{section2}. At present, we are concerned with the morphism $\pi_i'\colon Y_i' \to Y$.

\begin{definition} Let $\Omega^1_{Y_i'}(\log D_i')$  be the sheaf equal to $\Omega^1_{Y_i'}$ on the singular surface $Y_i'-D_i'$ and to  $\Omega^1_{Y_i'}(\log D_i')$ on the smooth surface $Y_i'-\{p_i\}$, glued together in the obvious way. Set $T^i_{Y_i'}(-\log D_i') =\mathit{Ext}^i(\Omega^1_{Y_i'}(\log D_i'), \scrO_{Y_i'})$ and $\mathbb{T}^i_{Y_i'}(-\log D_i') =\Ext^i(\Omega^1_{Y_i'}(\log D_i'), \scrO_{Y_i'})$.
\end{definition}

\begin{theorem} Let $\mathbf{Def}_{Y_i', D_i'}$ be the functor of deformations of $Y_i'$ keeping the effective Cartier divisor  $D_i'$. Then the tangent space to $\mathbf{Def}_{Y_i', D_i'}$ is  $\mathbb{T}^1_{Y_i'}(-\log D_i') $ and $\mathbb{T}^2_{Y_i'}(-\log D_i')$ is an   obstruction space for $\mathbf{Def}_{Y_i', D_i'}$.
\end{theorem}
\begin{proof} This is a minor variation on the proof for smooth $Y$ (cf.\ e.g.\ Sernesi \cite[3.4.17]{Sernesi}). 
\end{proof}

\begin{corollary}\label{cor16}  In the above notation, suppose that $H^2(Y_i'; T^0_{Y_i'}(-\log D_i')) = 0$ and that $p_i$ is a local complete intersection. Then there exists a deformation of $Y$ which smooths the singular point $p_i$ and for which the remaining points $p_j$, $j \neq i$, remain simple elliptic singularities.
\end{corollary}
\begin{proof} In any case, identifying $H^0(Y_i'; T^1_{Y_i'})$ with  $H^0(Y; T^1_{Y,p_i})$ and using the fact that $H^0(Y; T^2_{Y,p_i}) =0$ by our assumption,  the local to global spectral sequence for  $\mathbb{T}^i_{Y_i'}(-\log D_i')$ gives an exact sequence
$$ \mathbb{T}^1_{Y_i'}(-\log D_i')  \to   H^0(Y; T^1_{Y,p_i}) \to H^2(Y_i'; T^0_{Y_i'}(-\log D_i')) \to \mathbb{T}^2_{Y_i'}(-\log D_i') \to 0.$$
 Under the hypothesis that $H^2(Y_i'; T^0_{Y_i'}(-\log D_i')) = 0$,  $ \mathbb{T}^2_{Y_i'}(-\log D_i') =0$,  and hence  $\mathbf{Def}_{Y_i', D_i'}$ is unobstructed. Moreover,    the map $ \mathbb{T}^1_{Y_i'}(-\log D_i')  \to   H^0(Y; T^1_{Y,p_i})$ is surjective. We can then find a deformation $\mathcal{Y}_i'\to \Delta$ of $Y_i'$ which smooths the singular point $p_i$ and keeps the effective Cartier divisor  $D_i'$, in the sense that there exists a relative Cartier divisor $\mathcal{D}_i' \subseteq \mathcal{Y}_i'$ which restricts over $0$ to $D_i'$. By Wahl's theory \cite[Proposition 2.7]{WahlI}, we can blow down $\mathcal{D}_i'$ to get a flat family over $\Delta$ which is a deformation of $Y$ smoothing $p_i$ and which is equisingular (in the   sense of Remark~\ref{equiremark}) at the remaining points $p_j$ as desired. 
\end{proof}

By analogy with Proposition~\ref{prop1121}(i), we then have: 

\begin{proposition}\label{prop1171} {\rm(i)} Suppose that $H^2(\hY; T_{\hY}(-\log D_i')(-D_i)) = 0$. Then $H^2(Y_i'; T^0_{Y_i'}(-\log D_i')) = 0$ and thus the conclusions of Corollary~\ref{cor16} hold. 

\smallskip
\noindent {\rm(ii)} If the natural map
$H^1(\hY; T_{\hY}(-\log D)) \to H^1(D_i; T_{D_i})$ is surjective and $H^2(\hY; T_{\hY}(-\log D))=0$, then $H^2(\hY; T_{\hY}(-\log D_i')(-D_i)) = 0$ and   the conclusions of Corollary~\ref{cor16} hold.
\end{proposition} 
\begin{proof} (i) A minor variation on the proof of  Proposition~\ref{prop1121}(i) shows that the Leray map
$$H^1(\hY; T_{\hY}(-\log D_i')) \to H^0(Y_i'; R^1(\psi_i')_*T_{\hY}(-\log D_i'))$$ is surjective and that $H^2(\hY; T_{\hY}(-\log D_i')) =0$. Thus 
$$H^2(Y_i'; T^0_{Y_i'}(-\log D_i')) = H^2(Y_i'; R^0(\psi_i')_*T_{\hY}(-\log D_i'))  =0.$$

\smallskip
\noindent (ii) From the exact sequence
$$0 \to T_{\hY}(-\log D_i')(-D_i) \to T_{\hY}(-\log D) \to T_{D_i} \to 0,$$
we get a long exact sequence
$$H^1(\hY; T_{\hY}(-\log D)) \to H^1(D_i; T_{D_i}) \to H^2(\hY; T_{\hY}(-\log D_i')(-D_i))\to H^2(\hY; T_{\hY}(-\log D)).$$
Thus, if the first map is surjective and the last group is $0$, then $H^2(\hY; T_{\hY}(-\log D_i')(-D_i)) = 0$. 
\end{proof}

\begin{example} Returning to Example~\ref{ex13}, with $\hY$ the blowup of a minimal elliptic ruled surface, $D_1$ and $D_2$ two disjoint sections with negative self-intersection,  and $K_{\hY} = \scrO_{\hY}(-D_1 -D_2)$, 
$$(T_{\hY}(-\log D_1)(-D_2))\spcheck \otimes K_{\hY} = \Omega^1_{\hY}(\log D_1) \otimes \scrO_{\hY}(D_2-D_1 - D_2) = \Omega^1_{\hY}(\log D_1)(-D_1).$$
Thus by Serre duality $H^2(\hY; T_{\hY}(-\log D_1)(-D_2))$ is dual to $H^0(\hY; \Omega^1_{\hY}(\log D_1)(-D_1))$. There is an  exact sequence
$$0 \to \Omega^1_{\hY}(\log D_1)(-D_1)  \to  \Omega^1_{\hY} \to \Omega^1_{D_1} \to 0.$$
Thus $H^0(\hY; \Omega^1_{\hY}(\log D_1)(-D_1)) =0$ since $H^0(\hY; \Omega^1_{\hY}) \to H^0(D_1; \Omega^1_{D_1})$ is an isomorphism. It follows that, if the singularity at $p_2$ is a local complete intersection, then  we can smooth the point corresponding to $p_2$ while keeping a simple elliptic singularity in the deformation corresponding to $p_1$. Such a deformation will have a minimal resolution which is a rational surface.
\end{example} 

The second variation is the following:

\begin{proposition}\label{prop1122}   Assume   that  $Y$ has a simple elliptic singularity at   $p_1$, a local complete intersection  cusp singularity at  $p_2$,  and is otherwise smooth. As before, let $\pi\colon \hY\to Y$ be the minimal resolution and let $D_i = \pi^{-1}(p_i)$. Suppose that $H^2(\hY; T_{\hY}(-\log (D_1+ D_2)) )= 0$. Then 
the map $\mathbb{T}^1_{Y_1}(-\log D_1)\to H^0(Y_1; T^1_{Y_1,p_2})$ is surjective. In particular, 
there exists a deformation of $Y$ which is equisingular at $p_1$ and deforms $p_2$ to a simple elliptic singularity with the same multiplicity as $p_2$.
\end{proposition}
\begin{proof} As in Definition~\ref{piiprime}, let  $\pi_1\colon Y_1 \to Y$ be the minimal resolution at $p_1$ and let $\psi_1\colon \hY \to Y_1$ be the induced morphism. A Leray spectral sequence argument applied to the morphism $\psi_1$ and the sheaf 
$$R^0\psi_1{}_*T_{\hY}(-\log  D_1 ) = T^0_{Y_1}(-\log D_1)$$ as in the proof of Proposition~\ref{prop1121}(ii)  shows that $H^2(Y_1; T^0_{Y_1}(-\log D_1)) = 0$. Thus $\mathbb{T}^2_{Y_1}(-\log D_1) = 0$ and   the map $\mathbb{T}^1_{Y_1}(-\log D_1)\to H^0(Y_1; T^1_{Y_1,p_2})$ is surjective, proving the first statement. An adjacency of $p_2$ to a simple elliptic singularity with the same multiplicity   gives a global $1$-parameter deformation of the pair $(Y_1,D_1)$  keeping the Cartier divisor $D_1$ and deforms $p_2$ as desired. Blowing down the family of Cartier divisors corresponding to $D_1$ as in the proof of Corollary~\ref{cor16} then gives a deformation of $Y$ which is equisingular at $p_1$ and is as claimed at $p_2$. 
\end{proof}  

\subsection{Deformations of I-surfaces are unobstructed}

We prove a very general result about weighted hypersurfaces of degree $10$ in $\Pee(1,1,2,5)$. However, it is not of much use without further information of the kind discussed above. 

\begin{theorem}\label{defunobstr}  Let $Y$ be a weighted hypersurface of degree $10$ in $\Pee(1,1,2,5)$, not passing through the singular locus. Then $\mathbb{T}^2_Y = 0$, so that the deformations of $Y$ are unobstructed.
\end{theorem} 
\begin{proof}  For simplicity of notation denote $\Pee(1,1,2,5)$ by $\Pee$. Recall that $\mathbb{T}^2_Y = \Ext^2(\Omega^1_Y, \scrO_Y)$.  The conormal sequence for $Y$ is
$$0 \to I_Y/I_Y^2 \to \Omega^1_{\Pee}|Y \to \Omega^1_Y \to 0.$$
In particular the long exact Ext sequence  yields an exact sequence 
$$ H^1(Y;  N_{Y/\Pee}) \to \Ext^2(\Omega^1_Y, \scrO_Y) \to H^2(Y; T_{\Pee}|Y).$$
Here  $N_{Y/\Pee} \cong \scrO_Y(10)$. A standard argument shows that $H^1(\Pee; \scrO_{\Pee}(a)) = H^2(\Pee; \scrO_{\Pee}(a)) = 0$ for all $a\in \Zee$ and hence that   $H^1(Y;\scrO_Y(a))=0$   by using the long exact sequence associated to 
$$0 \to \scrO_{\Pee}(-Y)(a)= \scrO_{\Pee} (a-10) \to \scrO_{\Pee} (a) \to \scrO_Y(a) \to 0.$$
Thus $H^1(Y;  N_{Y/\Pee}) =0$. 
The remaining term   $H^2(Y; T_{\Pee}|Y)$  is Serre dual to $H^0(\Omega^1_{\Pee}|Y \otimes \omega_Y)$. The Euler exact sequence for weighted projective space is 
$$0 \to \Omega^1_{\Pee} \to \bigoplus_{i= 1}^4\scrO_{\Pee}(-a_i) \to \scrO_{\Pee} \to 0,$$
with the weights $a_i = 1,1,2,5$. Restricting to $Y$ and tensoring with $\omega_Y$ gives an exact sequence
$$0 \to H^0(Y;\Omega^1_{\Pee}|Y \otimes \omega_Y) \to \bigoplus_{i= 1}^4H^0(Y; \scrO_{\Pee}(1-a_i)) \to H^0(Y; \scrO_Y(1)) .$$
In our case,  $1-a_i$ is nonnegative only for the two weights $a_i = 1$,  $H^0(Y; \scrO_Y(1))$ has dimension $2$,  and the map $\bigoplus_{i= 1}^4H^0(Y; \scrO_{\Pee}(1-a_i)) \to H^0(Y; \scrO_Y(1)) $ is easily checked  to be surjective and hence an isomorphism.  Thus $H^0(Y;\Omega^1_{\Pee}|Y \otimes \omega_Y) =0$, and so finally $\mathbb{T}^2_Y =0$.
\end{proof} 
 
\section{Structure of I-surfaces with elliptic singularities}\label{section2}

\subsection{Some preliminary results}

The first result in this section is a standard result about iterated blowups: 

\begin{lemma}\label{lemma4} Let $\rho\colon S \to \overline{S}$ be the iterated blowup of a smooth surface $\overline{S}$ and let $\bigcup_iC_i$ be the fiber of $\rho$ over a point $p$. Let $K_S = \rho^*K_{\overline{S}} + \sum_ia_iC_i$ where the $a_i$ are nonnegative integers. Then $a_j >0$ for all $j$. If $C_j^2 = -1$ for some $j$ (i.e.\ $C_j$ is an exceptional curve) and  $\rho^{-1}(p)$ is reducible, then $a_j \ge 2$.
\end{lemma}
\begin{proof} The first statement is clear since every local section of $\rho^*K_{\overline{S}}$ vanishes along $C_j$ as a section of $K_S$. If $C_j ^2=-1$, then by adjunction  $C_j \cdot K_S = -1$. But $K_S = \rho^*K_{\overline{S}} + \sum_ia_iC_i$.  For $j \neq k$,   $C_j\cdot C_k\ge 0$,   and  by hypothesis $C_j\cdot C_k > 0$ for some $k$ as the fiber is connected.   Thus, for this choice of $k$, 
$$-1 = a_jC_j^2 + \sum_{i\neq j}a_i(C_j\cdot C_i) \ge  -a_j +a_k.$$
Hence $a_j \ge a_k + 1 \ge 2$. 
\end{proof}

\begin{lemma}\label{lemma1} Let $S$ be a smooth surface.
\begin{enumerate}
\item[\rm(i)] If $S$ is the blowup of a geometrically ruled surface over a base curve of genus at least $2$, then $S$ does not contain  a smooth elliptic curve or a cycle of rational curves.  
\item[\rm(ii)] If $S$ is the blowup of  a geometrically ruled surface over a base curve of genus at least $1$,   then $S$ does not contain a cycle of rational curves.  
\item[\rm(iii)] If $S$ is the blowup of a surface with $\kappa(S) =0$ or $\kappa(S) =1$ and $\chi(\scrO_S) =0$, then $S$ does not contain a cycle of rational curves.  
\end{enumerate}
\end{lemma}
\begin{proof} (i), (ii): Let $\rho\colon S \to C$ be the Albanese morphism   and let  $\overline{S}$ be  a relatively minimal model of $S$. Thus there is an induced morphism $\bar\rho\colon \overline{S} \to C$ exhibiting $\overline{S}$ as a  geometrically ruled surface. Then $S$ is an iterated blowup of $\overline{S}$ and   all fibers of $\rho$ are proper transforms  of fibers of $\bar\rho$.  Hence the reduction of every fiber of $\rho$  is a  tree of smooth rational curves. In case (i), the image under $\rho$ of of a smooth elliptic curve on $S$ would  be a point, and hence the curve would have to be contained in the blowup of a fiber, but this is impossible.  Likewise, in case (i) or (ii), a cycle of rational curves would have to be contained in the blowup of a fiber, but this is again impossible.  

\smallskip
\noindent (iii):   First assume that $\kappa(S) =1$ and  let $\overline{S}$ be the minimal model of $S$ and let $\rho\colon S \to C$ be the canonical elliptic fibration. Then, by the structure theory of elliptic surfaces,  all fibers of $\rho$ have smooth reduction. If there exists a  cycle of rational curves on $S$, then there exists a rational curve on $\overline{S}$, necessarily not contained in a fiber, and hence $g(C) =0$. Moreover, $p_g(\overline{S}) =0$ and $q(\overline{S}) = 1$. Then there exists a finite \'etale cover of $\overline{S}$ which is biholomorphic to a product of two smooth nonrational curves (e.g.\ \cite[Thm.\ VI.13]{Beauville} or \cite[Chapter 2]{FM}). This contradicts the existence of a rational curve on $\overline{S}$. The case where $\kappa(S) =0$ is similar (and simpler).
\end{proof} 

\begin{lemma}\label{lemmam22} Let $Y$ be a normal projective surface with $k$ isolated minimally elliptic singularities  and let $\pi'\colon Y'\to Y$ be a partial resolution of $Y$ at $k'$ of the singularities. Then $\chi(Y'; \scrO_{Y'}) = \chi(Y; \scrO_Y) - k'$.
\end{lemma}
\begin{proof} We have the Leray spectral sequence  $E_2^{p,q} = H^p(Y; R^q\pi'_*\scrO_{Y'}) \implies H^{p+q}(Y'; \scrO_{Y'})$. Thus, since $R^0\pi'_*\scrO_{Y'}=\scrO_Y$ and $R^1\pi'_*\scrO_{Y'}$ is a skyscraper sheaf with one-dimensional stalks at each of the $k'$ singular points resolved by $\pi'$, 
$\chi(Y'; \scrO_{Y'}) = \chi(Y; \scrO_Y) - k' $. 
\end{proof}

\begin{theorem}\label{theorem2}  Let $Y$ be a normal projective surface with $k>0$ isolated   singularities which are either simple elliptic or cusp singularities. Assume that $h^1(Y;\scrO_Y) = 0$ and set $p_g = h^2(Y;\scrO_Y)$. Then $k \le p_g+1$. If   $k = p_g+1$, then all singularities of $Y$ are simple elliptic and  a smooth minimal model for $Y$ is either an elliptic ruled surface, an elliptic surface   such that all fibers have smooth reduction, or is an abelian surface isogeneous to a product of two elliptic curves. 
\end{theorem}
\begin{proof} Let $\pi\colon \hY \to Y$ be the minimal resolution. 
If $k > p_g+1$, then 
$$\chi(\hY; \scrO_{\hY})  = \chi(Y; \scrO_Y) - k = 1+ p_g - k< 0,$$ by  Lemma~\ref{lemmam22}. Hence,  by the Castelnuovo-de Franchis theorem, $\hY$ is the blowup of a geometrically ruled surface over a base curve of genus at least $2$. This contradicts (i) of Lemma~\ref{lemma1}.

Now assume that $k = p_g+1$ and let $Y_{\text{\rm{min}}}$ be a smooth minimal model of $\hY$. Then  
$$\chi(Y_{\text{\rm{min}}}; \scrO_{Y_{\text{\rm{min}}}})  = \chi(\hY; \scrO_{\hY}) = \chi(Y; \scrO_Y) - k = 1+ p_g -k = 0.$$
Thus $c_1(Y_{\text{\rm{min}}})^2 + c_2(Y_{\text{\rm{min}}}) =12\chi(Y_{\text{\rm{min}}}; \scrO_{Y_{\text{\rm{min}}}}) =0$. In particular, since $\hY$ is not ruled over a curve of genus at least $2$, by classification $c_1(Y_{\text{\rm{min}}})^2 \ge 0$ and $c_2(Y_{\text{\rm{min}}}) \ge 0$, hence $c_1(Y_{\text{\rm{min}}})^2 =c_2(Y_{\text{\rm{min}}})= 0$. Thus, by the classification of such surfaces, $Y_{\text{\rm{min}}}$ is either an elliptic ruled surface, an elliptic surface over a curve $B$  such that all fibers have smooth reduction, or is an abelian surface, necessarily  isogeneous to a product of two elliptic curves.   In all cases, by Lemma~\ref{lemma1}, the blowup $\hY$ of $Y_{\text{\rm{min}}}$ will not contain a cycle of rational curves, so all singularities of $Y$ are simple elliptic.
\end{proof}

\begin{remark} With more effort, one can extend the proof of Lemma~\ref{lemma1} and hence of Theorem~\ref{theorem2} to cover the case where the singularities of $Y$ are just assumed to be minimally elliptic. 
\end{remark} 

\subsection{Numerology of I-surfaces} We begin by defining more carefully the surfaces considered in this paper.

\begin{definition}\label{defIsurface}  An \textsl{I-surface  with simple elliptic or cusp singularities} is a normal Gorenstein surface $Y$ whose singularities are simple elliptic, cusp,  or RDP singularities,  such that $\omega_Y$ is ample, $\omega_Y^2 =1$, $h^1(Y;\scrO_Y) =0$, $h^2(Y;\scrO_Y) = h^0(Y;\omega_Y) = 2$ and thus $\chi(\scrO_Y) = 3$. (In fact, the condition $h^1(Y;\scrO_Y) =0$ follows from the other two.) Let $\pi \colon \hY \to Y$ be a minimal resolution of singularities, so in particular no exceptional curve on $\hY$ is contained in a fiber of $\pi$. Let $p_1, \dots , p_k$ be the simple elliptic or cusp singularities, let $D_i=\pi^{-1}(p_i)$, and set $D = \sum_{i=1}^kD_i$. Thus   each $D_i$ is  a smooth elliptic curve, an irreducible  nodal rational  curve, or or a cycle of smooth rational curves, and $D$ is the exceptional set of $\pi$ $\iff$ there are no RDP singularities.    Set $m_i = -D_i^2$ and $m = -D^2 = \sum_i m_i$, with $m_i > 0$ for every $i$. In particular, $k\le m$.      By Theorem~\ref{theorem2}, $k \le 3$.
Let $L$ be the nef and big line bundle  $\pi^*\omega_Y$.
\end{definition}

\begin{assumption}\label{assumption27}  As noted in the introduction, to keep the number of cases manageable we will assume for the rest of this paper that  $Y$ has no RDP singularities. Thus, in the notation of Definition~\ref{defIsurface}, $L\cdot C \geq 0$ for every irreducible curve $C$ on $\hY$, and $L\cdot C =0$ $\iff$ $C$ is a component of $D_i$ for some $i$.   In particular, $L\cdot C > 0$ for every exceptional curve $C$ on $\hY$. 
\end{assumption}

\begin{remark} If  there are some  RDP singularities on $Y$,   then it is still true that $L\cdot C > 0$ for every exceptional curve $C$ on $\hY$. Under this assumption,  the arguments in this section for the classification of the possible $Y$ will carry over with minor modifications. 
\end{remark}

We have $K_{\hY} = \pi^*\omega_Y \otimes \scrO_{\hY}(-D)$, i.e.\ $L = K_{\hY}   + D$. Since $L\cdot D_i =0$ for every $i$,  
$$K_{\hY} \cdot D_i + D_i^2 =0.$$
More precisely, by adjunction,  $(K_{\hY}\otimes \scrO_{\hY}(D_i))|D_i  = \omega_{D_i} = \scrO_{D_i}$. 
Hence $K_{\hY} \cdot D_i = m_i> 0$. Also,
\begin{align*}
L^2 &= L \cdot (K_{\hY} + D) = L \cdot  K_{\hY} =1;\\
K_{\hY} ^2 &=(L -D) \cdot  K_{\hY}=L \cdot  K_{\hY} -m =1-m. 
\end{align*}

We can also consider partial resolutions:

\begin{definition}\label{defMi}  As in Definition~\ref{piiprime}, let  $\pi_i\colon Y_i \to Y$ be the minimal resolution of the singular point  $p_i$ and an isomorphism from $Y_i -D_i$ to $Y-p_i$. Note that, in case $k=1$, $Y_i = \hY$.   Denote by  $\psi_i \colon \hY \to Y_i$ the   induced morphism.   Finally set  $M_i =\psi_i^*\omega_{Y_i}$. 
\end{definition} 
  
 \begin{lemma}\label{Milist} In the above notation,  $\chi(\scrO_{Y_i}) = 2$. Moreover,
 \begin{enumerate}
 \item[\rm(i)] $M_i = L-D_i = K_{\hY} +D_i'$, where $D_i' = \sum_{j\neq i}D_j$. Here, if $k=1$, then  $D_i'=D_1'=0$ and $M_1 = K_{\hY}$. 
  \item[\rm(ii)] $M_i$ is effective and $L\cdot M_i =1$. In particular, $M_i$ is not numerically equivalent to $0$. 
   \item[\rm(iii)] For $i\neq j$, $M_i\cdot M_j =1$. In case $i=j$, $\omega_{Y_i}^2 = M_i^2 = 1 -m_i$. (Here $\omega_{Y_i}^2$ means the self-intersection of the Cartier divisor on $Y_i$.) 
   \item[\rm(iv)]  $M_i \cdot D_i'=0$, $M_i\cdot D_i = m_i$, and $M_i \cdot  K_{\hY} = 1-m_i$ .
    \item[\rm(v)] If $k> 1$ and either $h^1(\hY; \scrO_{\hY})=0$ or $\hY$ is a blown up elliptic ruled surface, then 
    $$\dim H^0(\hY; \scrO_{\hY}(M_i)) = 1.$$
 \end{enumerate}
 \end{lemma} 
 \begin{proof} The first statement follows from Lemma~\ref{lemmam22}. Then (i) is clear since $\hY$ is a resolution of $Y_i$ at the points $p_j, j\neq i$. Since $\chi(\scrO_{Y_i}) = 2$, $h^2(Y_i;\scrO_{Y_i}) \neq 0$ and hence $h^0(Y_i; \omega_{Y_i}) \neq 0$.  
Thus $M_i$ is effective. Since $\omega_{Y_i}$ is trivial in a neighborhood of $D_i'$, $M_i \cdot D_i'=0$. Moreover, $L\cdot M_i =  L\cdot (L-D_i) = L^2 = 1$. We have 
$$M_i\cdot M_j = (L-D_i)\cdot (L-D_j) =  L^2 + (D_i \cdot D_j) = \begin{cases} 1, &\text{if $i\neq j$;}\\
1-m_i, &\text{if $i=j$.}
\end{cases}$$
Then $M_i \cdot D_i = (L- D_i) \cdot  D_i = -D_i^2 = m_i$. Likewise  $M_i \cdot  K_{\hY} = M_i \cdot(M_i - D_i')  = M_i^2 = 1-m_i$.  

To see (v), note that $\dim H^0(\hY; \scrO_{\hY}(M_i)) = h^2(Y_i; \scrO_{Y_i})$. Since $\chi(\scrO_{Y_i}) = 2$, it suffices to prove that $h^1(\scrO_{Y_i}) =0$. The Leray spectral sequence gives an exact sequence
$$0 \to H^1(Y_i; \scrO_{Y_i}) \to H^1(\hY; \scrO_{\hY}) \to H^0(R^1\psi_i{}_*\scrO_{\hY}).$$
Thus, if $H^1(\hY; \scrO_{\hY}) =0$, then $H^1(Y_i; \scrO_{Y_i})=0$. If $\hY$ is the blowup of an elliptic ruled surface with base an elliptic curve $B$,  then all of the singularities on $Y_i$ are simple elliptic, and $H^0(R^1\psi_i{}_*\scrO_{\hY})\cong \bigoplus_{j\neq i}H^1(\scrO_{D_j })$. Choosing a $j\neq i$, the map $H^1(\hY; \scrO_{\hY}) \to H^1(\scrO_{D_j})$ is an isomorphism since $D_j \to B$ is an isogeny. Thus $H^1(\hY; \scrO_{\hY}) \to H^0(R^1\psi_i{}_*\scrO_{\hY})$ is injective so that $H^1(Y_i; \scrO_{Y_i})=0$ in this case as well.
 \end{proof} 
 
 \begin{remark}\label{blowdown}  The passage from $Y$ or $Y_i$ to $\hY$ is reversible, in the following sense. Suppose that $\hY$ is a smooth surface and that $D= \sum_iD_i$ is an effective divisor on $\hY$, where  the $D_i$ are the connected components of $D$ and each $D_i$ is a smooth elliptic curve, an irreducible nodal rational curve or a cycle of smooth rational curves (one could also allow configurations corresponding to the triangle singularity case). By adjunction,  $K_{\hY}\otimes \scrO_{\hY}(D)|D_i \cong \omega_{D_i} \cong \scrO_{D_i}$. Suppose  in addition that  the intersection matrix of $D$ is negative definite, or equivalently that there exists a contraction $\pi \colon \hY \to Y$, where $Y$ is a normal analytic surface and $\pi(D_i) = p_i\in Y$. In particular, the singular points are simple elliptic or cusp singularities.  Thus $\hY$ is the minimal resolution of a Gorenstein surface $Y$ with elliptic singularities. There exists a contractible  Stein neighborhood $U_i$ of $p_i$ in $Y$. Let $\widetilde{U}_i = \pi^{-1}(U_i)$. A standard argument using the Grauert-Riemenschneider theorem shows that restriction of line bundles induces an isomorphism  $\Pic \widetilde{U}_i \to \Pic D_i$. Let  $L=K_{\hY}\otimes \scrO_{\hY}(D)$. By the above,   $L|\widetilde{U}_i  = K_{\hY}\otimes \scrO_{\hY}(D)|\widetilde{U}_i \cong \scrO_{\widetilde{U}_i}$. Hence $L$  is the pullback of a line bundle on $Y$, necessarily $\omega_Y$. A similar statement holds for the line bundles $M_i = K_{\hY}\otimes \scrO_{\hY}(D_i')$ and the morphism $\psi_i\colon \hY \to Y_i$. Note that $\omega_Y$ is nef and big $\iff$ $L$ is nef and big, and in this case $\omega_Y$ is ample $\iff$ the irreducible curves $C$ on $\hY$ such that $L\cdot C =0$ are exactly the  components of $D_i$ for some $i$.
 \end{remark} 
 
\subsection{The main geometric result}  We begin with the following, due to  Franciosi-Pardini-Rollenske \cite[Theorem 4.1]{FPR1}, \cite[Proposition 4.3]{FPR2}, which deals with the case where $\kappa(\hY) \geq 0$, i.e.\ $\hY$ is not rational or ruled (and does not require Assumption~\ref{assumption27}):

\begin{theorem}\label{theorem5} Suppose that $L^2 =1$ and that $\kappa(\hY) \geq 0$.  Let $Y_{\text{\rm{min}}}$ be the minimal model of $\hY$.
\begin{enumerate}
\item[\rm(i)] If $\kappa(\hY) =2$, i.e.\ $Y_{\text{\rm{min}}}$ is of general type, then $\hY = Y = Y_{\text{\rm{min}}}$ and $k=0$, i.e.\ there are no simple elliptic or cusp singularities.
\item[\rm(ii)] If $\kappa(\hY) =1$, i.e.\ $Y_{\text{\rm{min}}}$ is properly elliptic, then $\hY= Y_{\text{\rm{min}}}$,  $k = m=1$, $\hY$ is a multiplicity two logarithmic transform of an elliptic $K3$ surface at a single fiber, one component of $D=D_1$ is a bisection of the elliptic fibration, and either $D$ is irreducible with $D^2 =-1$ or $D$ consists of a bisection of square $-3$ plus smooth components of square $-2$ lying in a fixed fiber. In this case, $K_{\hY} =\scrO_{\hY}(F)$, where $F$ is the multiple fiber, $D\cdot F =1$, and $L = F+D$. 
\item[\rm(iii)] If $\kappa(\hY) =0$, then either $k=m =2$ and $Y_{\text{\rm{min}}}$ is an Enriques surface or $k=1$, $m=2$ and $Y_{\text{\rm{min}}}$ is a $K3$ surface. In this case, $\hY$ is the blowup of  $Y_{\text{\rm{min}}}$ at one point. If $C$ is the exceptional curve, then $C\cdot D =2$ if $Y_{\text{\rm{min}}}$ is a $K3$ surface and $C\cdot D_1 = C\cdot D_2 =1$ if $Y_{\text{\rm{min}}}$ is an Enriques surface. Finally, $K_{\hY} \equiv C$ and $L \equiv D+C$. 
\end{enumerate}
\end{theorem}
\begin{proof} Let $\rho\colon \hY \to Y_{\text{\rm{min}}}$ be the canonical morphism. Then  $K_{\hY} = \rho^*K_{Y_{\text{\rm{min}}}} + \sum_ia_iC_i$, where the $C_i$ are the positive dimensional fibers of $\rho$ and $C_j$ is an exceptional curve for at least one $j$.  Since $\kappa(\hY) \geq 0$,  $K_{Y_{\text{\rm{min}}}}$ is nef and $K_{Y_{\text{\rm{min}}}}^2 \ge 0$. Since $L$ is nef and big, $L\cdot \rho^*K_{Y_{\text{\rm{min}}}} \ge 0$, with equality $\iff$ $K_{Y_{\text{\rm{min}}}}\equiv 0$ by the Hodge index theorem. Hence, if $\kappa(\hY) \ge 1$, from 
$$1 = L^2 = L\cdot K_{\hY} = L\cdot \left(\rho^*K_{Y_{\text{\rm{min}}}} + \sum_ia_iC_i\right),$$
and the fact that $L\cdot C_j > 0$ for at least one $j$, the sum $\sum_ia_iC_i$ must be empty, i.e.\ $\hY = Y_{\text{\rm{min}}}$. If $\kappa(\hY) = 0$ and $\hY \neq Y_{\text{\rm{min}}}$, then $\sum_ia_i(L\cdot C_i) = 1$. By Lemma~\ref{lemma4} and the fact that $L\cdot C > 0$ for every exceptional curve $C$ on $\hY$, there is a unique connected component of the fiber of $\rho$ and it is an irreducible exceptional curve, i.e.\ $\rho\colon \hY \to Y_{\text{\rm{min}}}$ is a single blowup. 

\smallskip
\noindent \textbf{The case $\kappa(\hY) =2$:}
Then  $K_{\hY}^2 > 0$.  By the  Hodge index theorem,
$$1\le (L^2) (K_{\hY}^2) = K_{\hY}^2 \leq (L\cdot K_{\hY})^2 = L^2 = 1.$$
Since equality holds, $L$ and $K_{\hY}$ are numerically equivalent, hence $D \equiv 0$ and $k=0$. 

\smallskip
\noindent \textbf{The case $\kappa(\hY) =1$:}   Since $\hY$ is minimal,  $K_{\hY}^2 = 0$ and $1 = L^2 = K_{\hY}^2 + m$, so $m= k=1$ and  $K_{\hY}\cdot D = 1$. In particular $D =D_1$ is connected.
By Lemma~\ref{lemmam22}, $\chi(\scrO_{\hY}) = 3-k=2$. Since $D ^2 < 0$ and $p_a(D) =1$, $D$ cannot be a union of fiber components (it would have to be a complete fiber). Hence  $D$ contains a multisection of the elliptic fibration. By the canonical bundle formula,
$$K_{\hY}\equiv (2 + 2g-2)f + \sum_\alpha(m_\alpha-1)F_\alpha = 2gf + \sum_\alpha(m_\alpha-1)F_\alpha,$$
where the $F_\alpha $ are the multiple fibers, of multiplicity $m_\alpha \ge 2$. Since $D\cdot F_\alpha \ge 1$ for all $\alpha$ and $K_{\hY}\cdot D = 1$, the only possibility is that $g=0$, there is a unique multiple fiber $F_\alpha$ with $m_\alpha = 2$ and  $D\cdot F_\alpha =1$. Hence $D$ contains a bisection of the elliptic fibration, with the other components equal to components of a single reducible fiber. In particular, the Jacobian elliptic surface corresponding to $\hY$ is an elliptic surface over $\Pee^1$ with a section and holomorphic Euler characteristic equal to $2$, hence is an elliptic $K3$ surface. Then $K_{\hY}\equiv F$. The remaining statements are clear or follow from adjunction:  either $D$ is an irreducible bisection, or $D$ is reducible and one component $E$ is a bisection which is smooth rational and the other components are smooth components of a fiber. In this case $-2 = K_{\hY}\cdot E +  E^2 = 1+E^2$, so that $E$ is the unique component of $D$ of self-intersection $-3$ and all other components have self-intersection $-2$.

\smallskip
\noindent \textbf{The case $\kappa(\hY) =0$:} In this case, $\hY\neq Y_{\text{\rm{min}}}$, because otherwise 
$$0 = 2p_a(D_i ) -2 = D_i^2  + K_{\hY} \cdot D_i =  D_i^2,$$
contradicting $D_i^2 < 0$. By the remarks at the beginning of the proof, $\hY$ is then the blowup of $Y_{\text{\rm{min}}}$ at one point. Thus, 
$$1 = L^2 = K_{\hY}^2 + m = -1+m, $$
so that $m=2$. Either $k=m=2$ and $D = D_1 + D_2$ with $D_i^2 =-1$ or $k=1$ and $D = D_1$ is connected with $D^2 =-2$. In the first case, $\chi(\scrO_{\hY}) = 3-2 =1$, so $\hY$ is an Enriques surface. In the second case, $\chi(\scrO_{\hY}) = 3-k =2$, so $\hY$ is a $K3$ surface.

Let $C$ be the unique exceptional curve on $\hY$. In the $K3$ case, $K_{\hY} = C$ and $C \cdot D = K_{\hY} \cdot D = -D^2=2$. In the Enriques case, $K_{\hY} \equiv  C$ and $C \cdot D_i = K_{\hY} \cdot D_i = -D_i^2=1$. In both cases,  $L = K_{\hY}+ D \equiv C+ D$. 
\end{proof}

We are thus reduced to analyzing the case $\kappa(\hY) =-\infty$.  First note the following special case of \cite[Lemma 4.5]{FPR1}:

\begin{lemma}\label{kappaminus} If $\kappa(\hY) =-\infty$, then there are two possibilities:
\begin{enumerate}
\item[\rm(i)] $\hY$ is a rational surface and $k=2$.
\item[\rm(ii)] $\hY$ is a ruled surface over an elliptic curve $B$ and $k=3$. In this case all singularities are simple elliptic.
\end{enumerate}
\end{lemma} 
\begin{proof} By Lemma~\ref{lemma1}, if $\kappa(\hY) = -\infty$, then either $\hY$ is rational and hence $\chi(\scrO_{\hY}) = 1$, or $\hY$ is the blowup of an elliptic ruled surface and hence $\chi(\scrO_{\hY}) = 0$ and all singularities are simple elliptic. By Lemma~\ref{lemmam22}, $k=2$ in the first case and $k=3$ in the second case. 
\end{proof}

 The strategy now is to see what prevents  the divisors $M_i$ introduced in Definition~\ref{defMi} from being nef.
 
 \begin{lemma}\label{Minotnef} If $C$ is an irreducible curve on $\hY$ such that $M_i \cdot C < 0$, then $C$ is an exceptional curve, $M_i\cdot C =-1$, and  $D_i' \cdot C =0$.  
 \end{lemma}
 \begin{proof} By Lemma~\ref{Milist}, $M_i$ is effective. Thus a curve $C$ such that $M_i \cdot C < 0$ must satisfy  $C^2 < 0$. If $C$ is a component of $D_i'$, then $M_i \cdot C =0$ since $M_i$ is trivial in a neighborhood of $D_i'$.  It follows  that $C$ is not a component of $D_i'$ so that $D_i' \cdot C \geq 0$.    Then
 $$0> M_i\cdot C = (K_{\hY} +D_i')\cdot C \ge K_{\hY}\cdot C.$$
 Thus  $C$ is exceptional, $K_{\hY}\cdot C = -1$, and therefore 
 $D_i' \cdot C  = 0$ and $M_i\cdot C =-1$. 
  \end{proof}
  
  The main result is then the following, which generalizes Theorem~\ref{theorem5} to the case where $\kappa(\hY) = -\infty$:  
  
  \begin{theorem}\label{geomanal}  With notation as above, suppose that $\kappa(\hY) = -\infty$ and in particular that $k > 1$. Then exactly one of the following holds:
  \begin{enumerate}
 \item[\rm(i)] $M_i$ is not nef. Then  there exists an exceptional curve $C$ such that $C \cdot D_i'=0$. In this case, if $\gamma \colon \hY \to Y_0$ is the surface obtained by blowing down $C$, then $M_i =  C$, $C\cdot D_i = 2$ and $m_i = 2$. Finally,   identifying $D_i'$ with its image in $Y_0$,   $K_{Y_0} \equiv  - D_i'$, $\hY$ is the blowup of $Y_0$ at a point not on $D_i'$, and $D_i$ is the proper transform of a divisor $\Gamma_i$ on $Y_0$ disjoint from $D_i'$ and containing $p$ as a point of multiplicity $2$ with $p_a(\Gamma_i) = 2$.
  \item[\rm(ii)] $M_i$ is nef. Then   $M_i^2=0$ and $m_i = 1$. Moreover,  $|2M_i|$ is a base point free linear system defining an elliptic fibration  $\hY \to \Pee^1$  (not necessarily relatively minimal) with at most $3$ multiple fibers, all of multiplicity $2$. More precisely, there is a single multiple fiber if $\hY$ is rational and there are $2$ or $3$ multiple fibers if $\hY$ is an elliptic ruled surface.  Moreover $M_i$ is a multiple fiber. Finally,  either $D_i$ is an irreducible bisection or one irreducible component of  $D_i$ is   a smooth bisection of the fibration, and the remaining components are curves of self-intersection $-2$ contained in a single reducible, non-multiple fiber. 
  \end{enumerate}
  \end{theorem}
 \begin{proof} Case (i):  $M_i$ is not nef.  Let $C$ be as in the statement of Lemma~\ref{Minotnef}. Contracting $C$ gives a morphism $\gamma_1\colon \hY \to Y_1$, where $Y_1$ is a smooth surface   containing a union of curves isomorphic to $D_i'$, which we will continue to denote by $D_i'$. Let $N_1= \gamma_1{}_*M_1$, so that $N_1 = \gamma_1{}_*(K_{\hY} + D_i') = K_{Y_1} + D_i'$ and $M_i = \gamma_1^*N_1 + C= \gamma_1^*K_{Y_1} + C + D_i'$.  The argument of Lemma~\ref{Minotnef} shows that, if $N_1$ is not nef, then there exists an exceptional curve $C'$ on $\overline{Y}$ such that $N_1\cdot C'< 0$ and $C' \cdot D_i'=0$. Then we can keep repeating this argument, which cannot continue indefinitely. Thus we eventually find a birational morphism $\gamma\colon \hY \to Y_0$, where $Y_0$ is smooth,  and a nef divisor $N_0$ on $Y_0$ such that $M_i =\gamma^*N_0 + \sum_ja_jC_j = \gamma^*K_{Y_0} + \sum_ja_jC_j + D_i'$, where the $C_j$ are the components of the one-dimensional fibers of the morphism $\gamma \colon \hY\to Y_0$ and $a_j > 0$. Then $L\cdot \gamma^*N_0  \ge 0$ since $L$ and $\gamma^*N_0$ are nef, and $L\cdot  a_jC_j  \ge a_j$ if $C_j$ is exceptional. Since at least one curve in every connected positive dimensional fiber of $\gamma$ is exceptional and
 $$1 = L\cdot M_i = (L\cdot \gamma^*N_0 ) + \sum_ja_j(L\cdot C_j),$$
 the only possibility via Lemma~\ref{lemma4} is that there is just one exceptional curve $C =C_j$,  $a_j = 1$, and   $L\cdot \gamma^*N_0  = 0$. Since $N_0$ is nef, $N_0^2 = (\gamma^*N_0)^2 \ge 0$. By the Hodge index theorem, $\gamma^*N_0 \equiv 0$ and hence $N_0\equiv 0$. Thus $M_i = K_{\hY} + D_i' \equiv C$. Since  $M_i$ is effective and $C$ is irreducible with $C^2< 0$,  $M_i \equiv  C$. Hence $L = M_i +D_i \equiv D_i + C$ and $K_{\hY} \equiv C-D_i'$.  Then $K_{\hY} = \gamma^*K_{Y_0} + C \equiv C-D_i'$, so that $\gamma^*K_{Y_0}   \equiv   -D_i' \equiv -\gamma^*D_i'$ where we view the divisor $D_i'$ as living on $Y_0$. Thus $K_{Y_0}   \equiv  -D_i'$ since $\gamma^*\colon  \operatorname{Num}Y_0 \to \operatorname{Num} \hY$ is injective.  Next we claim that $C \cdot D_i =2$, as follows from  the equalities
 $$1 = L^2 = L\cdot (D_i + C) = L \cdot C = (D_i +C)\cdot C = (D_i\cdot C) + C^2 = (D_i\cdot C)  -1.$$
The statement that $m_i =2$ follows from the equalities
 $$0 = L\cdot D_i = (M_i +D_i)\cdot D_i = (D_i + C)\cdot D_i = -m_i +2.$$
  If $\Gamma_i =\gamma(D_i)$, then $D_i = \gamma^*\Gamma_i - 2C$ because $D_i\cdot C =2$. This proves the remaining statements in Case (i). 
 
 \smallskip
 \noindent  Case (ii):  $M_i$ is nef. Then $M_i^2 = 1-m_i \ge 0$. Hence $m_i =1$ and $M_i^2 = 0$. Next we claim that $\dim |2M_i| \ge 1$ and in fact that $|2M_i|$ is a base point free linear system. This is a minor variation on a standard argument in classification theory (cf.\ e.g.\ \cite[Lemma 24 p.\ 295]{Friedmanbook}):  First,  $h^0(\hY; \scrO_{\hY}(nM_i)) = h^0(Y_i; \omega_{Y_i}^{\otimes n})$. Lemma~\ref{Milist} implies that $\chi(Y_i; \scrO_{Y_i}) = 2$. By an easy extension of the Riemann-Roch theorem to a projective Gorenstein surface with isolated singularities,
 $$\chi(Y_i; \omega_{Y_i}^{\otimes n}) = \frac{n^2M_i^2 -nM_i^2}{2} + \chi(Y_i; \scrO_{Y_i}) = 2.$$
For $n\ge 2$,  $h^2(Y_i; \omega_{Y_i}^{\otimes n}) =  h^0(Y_i; \omega_{Y_i}^{\otimes 1-n}) = h^0(\hY; \scrO_{\hY}((1-n)M_i)) = 0$ since $M_i$ is effective and nonzero. It follows that
 $$h^0(\hY; \scrO_{\hY}(nM_i)) = h^0(Y_i; \omega_{Y_i}^{\otimes n})\ge 2$$
 for all $n\ge 2$, in particular for $n=2$. Thus $\dim |2M_i| \ge 1$. Write $2M_i = G_f + G_m$, where $G_f$  is the fixed curve and $G_m$ the moving part of the linear system $|2M_i|$. Since $M_i$ is nef,  $M_i\cdot G_f \ge 0$ and $M_i\cdot G_m \ge 0$, and hence $M_i\cdot G_f=M_i\cdot G_m = 0$. By the Hodge index theorem, since $M_i$ is not numerically equivalent to $0$, $G_f$ and $G_m$ are rational multiples of $M_i$ and hence $G_m^2=G_f^2 =0$. Thus $|2M_i| = |G_m|$ defines a fibration $\varphi\colon \hY \to B$ over some base curve $B$. Let  $G$ be a general fiber of $\varphi$, so that  $G$ is smooth and $G^2 =0$.  Since $G\cdot M_i = G_f\cdot M_i = G_m \cdot M_i =0$, every effective curve    $\Gamma \in |M_i|$  is contained in fibers of the morphism $\varphi$ as is $G_f$.  Hence $\Gamma$ and $G_f$ are numerically equivalent  to positive rational multiples of $G$  since $M_i^2=G_f^2 =0$.  Also,  $M_i$ is a primitive class since $L\cdot M_i =1$. Since $2M_i = G_f+G_m$, either $G_f=0$ and $G_m=2M_i$ or $G_f\equiv G_m \equiv M_i$. In any case,  $G_m \equiv aG$ for some positive integer $a$ and hence $L\cdot G \leq 2$.  By Lemma~\ref{Milist}(iv),  $G\cdot K_{\hY}=0$, hence $G$ is an elliptic curve and $\varphi\colon \hY \to B$ is a not necessarily relatively minimal elliptic fibration. 
 
 Let $\varphi_0\colon Y_0\to B$ be the relatively minimal model of $\varphi$. Then $Y_0$ is a (relatively minimal) elliptic surface. By the canonical bundle formula, after  identifying $G$ with a general fiber of $\varphi_0$,  
 $$K_{Y_0} =(2g-2+ \chi)G + \sum_\alpha(m_\alpha-1)F_\alpha,$$
 where $g$ is the genus of $B$, $\chi=\chi(\scrO_{Y_0}) = \chi(\scrO_{\hY})\ge 0$, and the $F_\alpha$ are the multiple fibers of multiplicities $m_\alpha > 1$. Since $L\cdot G =m_\alpha ( L\cdot F_\alpha) \leq 2$, $m_\alpha \leq 2$ as well, and hence $m_\alpha = 2$ for every $\alpha$.  Because $\kappa(Y_0) = \kappa(\hY) = -\infty$, $ g= 0$, $B \cong\Pee^1$,   $\chi =0$ (the elliptic ruled case) or $1$ (the rational case), and 
 $$\chi -2 + \sum_\alpha \left(1- \frac{1}{m_\alpha}\right) = \chi -2 + \frac{\#(\alpha)}2< 0.$$
 In particular, in the rational case there is at most one multiple fiber.  In the elliptic ruled case, $Y_0$ is a logarithmic transform of a product elliptic surface $G\times \Pee^1$ and there are either $0$,  $2$ or $3$ multiple fibers, because  a logarithmic  transform of a product surface at a single  fiber is never algebraic (cf.\ \cite[I.6.12]{FM}).  Then we have the following formula for the canonical bundle of $\hY$: 
 $$K_{\hY} = (\chi-2)G + \sum_\alpha F_\alpha + \sum_ta_tC_t,$$
 where the $C_t$ are the exceptional curves in the blowup $\hY \to Y_0$. Since $M_i\cdot D_i' =0$, the connected components of $D_i'$ are contained in fibers of $\varphi$ and have arithmetic genus one. Hence 
 $D_i'$ is a union of proper transforms of fibers of $\varphi_0$, and so 
 $$D_i'= nG + \sum_\alpha n_\alpha F_\alpha - \sum_tb_tC_t,$$  where $n, n_\alpha, b_t$ are nonnegative integers and $n_\alpha$ is either $0$ or $1$. Clearly $n+ \sum_\alpha n_\alpha = k-1$. Since $M_i = K_{\hY} + D_i'$ and $M_i^2 =0$, $a_t = b_t$ for all $t$. Thus  
 $$M_i = (n+   \chi-2)G +  \sum_\alpha(n_\alpha +  1)F_\alpha.$$
 
 First suppose that $\hY$ is rational. Then $k=2$ and hence $n \le 1$. Since $M_i$ is effective and nonzero and $M_i = (n-1)G +  \sum_\alpha(n_\alpha +  1)F_\alpha$, there has to exist a multiple fiber. But $\hY$ is rational, so there  is exactly one multiple fiber $F=F_\alpha$, necessarily of multiplicity $2$. Since $n+n_\alpha = 1$, the case $n=0$, $n_\alpha =1$ is not possible since then $M_i = -G +2F_\alpha =0$. Hence $n=1$, $n_\alpha =0$ and $M_i = F_\alpha =F$.

Now suppose that $\hY$ is elliptic ruled.  Then $k=3$ and $n+ \sum_\alpha n_\alpha  =2$. In particular, $n \le 2$. Here $M_i = (n-2)G +  \sum_\alpha(n_\alpha +  1)F_\alpha$. As in the rational case, there must exist a multiple fiber, and hence there are $2$ or $3$ multiple fibers. Note that $n=2$ is impossible, since then $M_i =    \sum_\alpha(n_\alpha +  1)F_\alpha$ is a sum of at least two multiple fibers, hence is not primitive.  There are two cases: 

\smallskip
\noindent Case (1): $n =1$, $n_\alpha =1$ for exactly one $\alpha$, and $n_\beta =0$ for $\beta \neq \alpha$. In this case, 
$$M_i = -G + 2F_\alpha + \sum_{\beta \neq \alpha}F_\beta = \sum_{\beta \neq \alpha}F_\beta.$$
Since $M_i$ is primitive, there are exactly two multiple fibers $F_\alpha, F_\beta$, and $M_i =  F_\beta$. 

\smallskip
\noindent Case (2):  $n=0$, $n_{\alpha_1} =n_{\alpha_2} =1$ for two multiple fibers  $F_{\alpha_1}$, $F_{\alpha_2}$, and $n_\beta =0$ for a potential third multiple fiber $F_\beta  \neq F_{\alpha_i}$. In this case, 
$$M_i = -2G + 2F_{\alpha_1} + 2F_{\alpha_1} + \sum_{\beta \neq \alpha_1, \alpha_2}F_\beta =  \sum_{\beta \neq \alpha_1, \alpha_2}F_\beta.$$
Using again the fact that $M_i$ is primitive and effective, there are exactly $3$ multiple fibers $F_{\alpha_1}, F_{\alpha_2}, F_\beta$, and $M_i = F_\beta$.

In all cases, $G = 2M_i$, hence $G_f =0$ and $G_m = G$. Note that $M_i$ is effective and  $M_i = F$ for a multiple fiber $F$. In particular, the elliptic fibration cannot have a section. Since $M_i \cdot D_i =1$, $D_i$ is a bisection of the fibration. More precisely, since $M_i$ is nef, exactly one component of $D_i$ is a bisection (because there are no sections of the fibration) and the remaining components are fiber components. In particular, if $D_i$ is reducible, then $\hY$ is a rational surface. Since $D_i$ is then both a bisection and a cycle of rational curves, it is clear that the remaining components must lie in a single reduced fiber.
 \end{proof}
 
 If $\hY$ is a rational surface, the possibilities for the multiplicities $(m_1, m_2)$ up to order are $(2,2)$, $(2,1)$, or $(1,1)$, and we will show in \S\ref{section6} that all such possibilities occur. For the case where $\hY$ is a blown up elliptic ruled surface, the situation is quite different: 
 
 \begin{proposition}\label{211case} Suppose that $\hY$ is a blown up elliptic ruled surface and that $\rho\colon \hY \to B$ is the Albanese map. Then the possibilities for $(m_1, m_2, m_3)$ up to order are either $(2,1,1)$ or $(1,1,1)$. More precisely,  if  $m_1 = 2$, then $m_2 = m_3 = 1$ and $D_2$, $D_3$ are sections of $\rho$.
 \end{proposition} 
 \begin{proof} By Theorem~\ref{geomanal}(i), if $m_1 = 2$, then there exists an exceptional curve $C$ such that $M_1 = C$ and $D_1\cdot C = 2$. If  $f$ is a general fiber of the  Albanese map $\rho \colon \hY \to B$, then $f \cong \Pee^1$, $f^2 =0$, and $K_{\hY} \cdot f = -2$. Also,   $D_i$  is not contained in a fiber of $\rho$ and hence  $f\cdot D_i \ge 1$. Since $C$ is contained in a (reducible) fiber, $f\cdot C = 0$.  From
$$0 = f\cdot C = f\cdot (K_{\hY} +D_2 +D_3) = -2 + (f\cdot D_2) + (f\cdot D_3),$$
it follows that $(f\cdot D_2) = (f\cdot D_3) = 1$, i.e.\ $D_2$ and $D_3$ are sections. If $m_i =2$ for $i=2$ or $3$, the above would then imply that $D_1\cdot f =1$.  But $D_1\cdot C = 2$, and hence $D_1\cdot f \ge 2$, a contradiction. Thus $m_2 = m_3 =1$.  
\end{proof} 

We will show in \S\ref{section8} that both  possibilities $(2,1,1)$ and  $(1,1,1)$ occur.

\section{The case $\kappa(\hY) =1$}\label{section4}

\subsection{The cotangent bundle}\label{cotangbundless} We begin with a fairly general result:

\begin{theorem}\label{tangbundell}  Suppose that  $\pi\colon X \to \Pee^1$ is an elliptic surface with $p_g(X) =p_g$ and that the only non-reduced fibers of $\pi$ are multiple fibers $F_i$ with multiplicities $m_i$.  Then there is an exact sequence
$$0 \to \scrO_X(-2f+ \sum_i(m_i-1)F_i) \to \Omega^1_X \to I_Z\otimes \scrO_X((p_g+1)f)\to 0,$$
where $Z$ is a zero-dimensional subscheme of length $\ell(Z) = 12(1+p_g)$ supported at the singular points of the reductions of the fibers, and at a point $x$ of a   fiber whose reduction has  a normal crossing singularity,  $I_{Z,x} =\mathfrak{m}_x$ is the maximal ideal at $x$. 
\end{theorem}
\begin{proof} The natural homomorphism $\pi^*\Omega^1_{\Pee^1} \to \Omega^1_X$ vanishes along the multiple fibers, to order exactly $m_i-1$ along $F_i$. Thus there is an induced homomorphism $$\pi^*\Omega^1_{\Pee^1}(\sum_i(m_i-1)F_i)=\scrO_X(-2f+ \sum_i(m_i-1)F_i)\to \Omega^1_X$$ whose cokernel is torsion free and hence of the form $I_Z\otimes \lambda$ for a line bundle $\lambda$. By taking determinants and using the the canonical bundle formula,
\begin{align*}
\lambda &= K_X \otimes (\pi^*\Omega^1_{\Pee^1}(\sum_i(m_i-1)F_i))^{-1}  \\
&=  \scrO_X((p_g-1)f +\sum_i(m_i-1)F_i) \otimes \scrO_X(2f -\sum_i(m_i-1)F_i)= \scrO_X((p_g+1)f). 
\end{align*}
It follows that
$$c_2(X) = c_2(T_X) = c_2(\Omega^1_X) = \ell(Z),$$
and hence that $\ell(Z) = 12(1+p_g)$. A local calculation then establishes that $I_{Z,x} =\mathfrak{m}_x$ in case the reduction of the fiber containing $x$ is nodal.
\end{proof} 

We will need a somewhat specialized variant of Theorem~\ref{tangbundell} for logarithmic $1$-forms:

\begin{theorem}\label{tangbundelllog} Suppose that $\pi\colon X \to \Pee^1$ is an elliptic surface such that all non-multiple fibers are nodal and  there is exactly one  multiple fiber  $F$ of $\pi$,  of multiplicity $2$.  Let $D = E_1+ \cdots + E_r$ be a reduced normal crossing cycle of rational curves, where $E_1$ is a smooth bisection and, for $i> 1$,  the $E_i$  are components of a non-multiple  I${}_a$ fiber $f_t$, with $E_1\cdot E_2 = E_1 \cdot E_r \neq 0$ and $E_i\cdot E_{i+1} = 1$, $2\le i \le r-1$. Then there is an exact sequence
$$ 0 \to \scrO_X(-2f+ F + \sum_{i>1}E_i) \to \Omega^1_X(\log D) \to I_W\otimes \scrO_X((p_g+1)f+E_1)\to 0,$$
where $W$ is a zero-dimensional subscheme of length $\ell(W) = 12(1+p_g) -r +1$. Moreover, $x\in \operatorname{Supp}W$ $\iff$ either $x$ is a singular point of a fiber not equal to $f_t$, $x$ is a singular point of  $f_t-D$, or $x\in E_1-F$ and $x$ is a ramification point for the induced double cover morphism $E_1\to \Pee^1$, necessarily not on $f_t$. Finally,  
\begin{enumerate}
\item[\rm(i)] If $x$ is a point of a  fiber different from $f_t$ whose reduction has  a normal crossing singularity at $x$ and $x\notin E_1$,  then $I_{W,x} =\mathfrak{m}_x$ is the maximal ideal at $x$. 
\item[\rm(ii)] If $x\in E_1-F$ is a ramification point for the induced double cover morphism $E_1\to \Pee^1$ and $x$ is not a singular point of the corresponding fiber, then $I_{W,x} =\mathfrak{m}_x$. 
\item[\rm(iii)] If $x\in E_1-F$ is a ramification point for the induced double cover morphism $E_1\to \Pee^1$ and $x$ is  a singular point of the corresponding fiber, then $\ell(I_{W,x} ) = 2$. 
\item[\rm(iv)] If $x\in f_t- D$ is a singular point of $f_t$, then $I_{W,x} =\mathfrak{m}_x$. 
\end{enumerate}
\end{theorem} 
\begin{proof} Note that $\det \Omega^1_X(\log D) = K_X \otimes \scrO_X(D) = \scrO_X((p_g-1)f +F + D)$. As in the proof of Theorem~\ref{tangbundell}, the natural map $\pi^*\Omega^1_{\Pee^1} \to \Omega^1_X(\log D)$ vanishes along $F$ and has a torsion free cokernel away from $F+D$. If $x\notin D$, then the analysis of the proof of Theorem~\ref{tangbundell} shows that the cokernel of $\pi^*\Omega^1_{\Pee^1}(F) \to \Omega^1_X(\log D)$ is locally free at $x$ if $x$ is not a singular point of a fiber and is locally isomorphic to $\mathfrak{m}_x$ if $x$ is a singular point of a fiber.  

Next suppose that  $x\in E_1-f_t$.  If the reduced fiber containing $x$ meets $E_1$ transversally (including the case where $x\in E_1\cap F$), then the cokernel of $\pi^*\Omega^1_{\Pee^1}(F) \to \Omega^1_X(\log D)$ is locally free at $x$. Otherwise, $x$ is a branch point of the morphism $E_1\to \Pee^1$ and a local calculation shows that $\ell(I_{W,x} ) =1$ if $x$ is a smooth point of the fiber and $\ell(I_{W,x} ) = 2$ otherwise.

Thus we may assume that  $x\in f_t\cap D = \sum_{i>1}E_i$. If $x\in E_i-E_1$ is a smooth point of $f_t$, then, locally around $x$, there exist local coordinates $z_1, z_2$ such that  $E_i$ is defined by $z_1 = 0$ and the morphism $X\to \Pee^1$ is given by $z_1$. A local basis for $\Omega^1_X(\log D)$ is $\displaystyle\frac{dz_1}{z_1}, dz_2$ and the image of $\pi^*\Omega^1_{\Pee^1}$ is the span of $dz_1$. Thus the cokernel of the induced morphism $\pi^*\Omega^1_{\Pee^1}(E_i ) \to \Omega^1_X(\log D)$ is locally free at $x$, i.e.\ $x$ is not in the support of the scheme $W$. Next suppose that   $x\in E_i\cap E_{i+1}$ is a singular point of the fiber $f_t$. Then there are local coordinates $z_1, z_2$ such that the fiber is defined at $x$ by $z_1z_2$.  A local basis for $\Omega^1_X(\log D)$ is $\displaystyle\frac{dz_1}{z_1}, \frac{dz_2}{z_2}$ and the image of $\pi^*\Omega^1_{\Pee^1}$ is
$$ z_2dz_1 + z_1dz_2 = (z_1z_2) \left(\frac{dz_1}{z_1}+ \frac{dz_2}{z_2}\right).$$
This vanishes to order one along $E_i + E_{i+1}$ and  the cokernel of the induced morphism $$\pi^*\Omega^1_{\Pee^1}(E_i + E_{i+1}) \to \Omega^1_X(\log D)$$ is locally free at $x$, i.e.\ $x$ is not in the support of the scheme $W$. 
If $x\in E_1\cap E_2$ or $E_r \cap E_1$, then $x$ is not a singular point of $f_t$ and it is not a branch point of $E_1\to \Pee^1$. Then there are local coordinates $z_1, z_2$ such that $E_1$ is defined by $z_1 = 0$, $E_2$ by $z_2 = 0$, and the image of $\pi^*\Omega^1_{\Pee^1}$ in $\Omega^1_X(\log D)$ is spanned by $\displaystyle dz_2 = z_2\cdot \frac{dz_2}{z_2}$. Thus the cokernel of the induced homomorphism $\pi^*\Omega^1_{\Pee^1}(E_2) \to \Omega^1_X(\log D)$ is locally free at $x$ in this case as well and $x$ is not in the support of the scheme $W$.  A similar argument shows that, if $x$ is a singular point of $f_t$ of the form $E_2\cap C$ or $E_r\cap C$, where $C$ is not a component of $D$, then $x$ is likewise not in the support of the scheme $W$.

Thus, $\pi^*\Omega^1_{\Pee^1}(F) \to \Omega^1_X(\log D)$ vanishes along $E_i$, $i> 1$,  and extends to a homomorphism 
 $$\pi^*\Omega^1_{\Pee^1}(F+\sum_{i>1}E_i) =\scrO_X(-2f+ F + \sum_{i>1}E_i) \to \Omega^1_X(\log D ) $$
 with torsion free cokernel. Computing determinants then gives the exact sequence in the statement of the theorem. Finally, a straightforward argument shows that $c_2(\Omega^1_X(\log D)) = 12(1+p_g) -r$. By the Whitney product formula, 
 $$c_2(\Omega^1_X(\log D)) = 12(1+p_g) -r = \ell(W)   -4 +1 +2 = \ell(W) -1,$$
 and hence $\ell(W) = 12(1+p_g) -r +1$.
\end{proof}

\subsection{Multiple fibers of multiplicity $2$ and double covers}\label{doubcovsect}

Let $\pi \colon X\to \Pee^1$ be a relatively minimal elliptic surface with a single multiple fiber, of multiplicity $2$, with reduction $F$, together with an irreducible, not necessarily smooth   bisection $\Sigma$. There is an involution on the generic fiber $X_\eta$, defined by: if $p$ is a smooth point of a fiber $f$, then $\iota(p) = p'$, where $\scrO_X(\Sigma)|f \cong \scrO_f(p+ p')$.
Hence (by relative minimality) there is an induced involution $\iota \colon X \to X$, whose fixed locus is a $4$-section together with some isolated fixed points, in particular two isolated fixed points on $F$. Note that $\iota$ only depends on the restriction of $\Sigma$ to the generic fiber. Also, $\iota(\Sigma) =\Sigma$ and $\iota|\Sigma$ is the involution corresponding to the double cover $\Sigma \to \Pee^1$. We make the following additional assumption on $X$, which will be satisfied in \S\ref{ssect3.3}, and in particular in  the hypotheses of Lemma~\ref{2compfibers} and  Theorem~\ref{thm17}:

\begin{assumption}\label{assump1} The   multiple fiber $F$ is irreducible,   the non-multiple fibers of $\pi$ have at most two components, and, in case there are two components, $\Sigma$ has intersection number one with each component. 
\end{assumption}

Assumption~\ref{assump1} has the following consequences:  Every non-multiple fiber is reduced, and is either a nodal or a cuspidal rational curve, if irreducible, or a  I${}_2$ fiber or a tacnodal fiber, if there are two components. The involution $\iota$ has exactly two isolated fixed points and they both lie on $F$. Finally, if  the  fiber has two components, they are exchanged by $\iota$. Thus, all fibers of the induced morphism $X/\iota \to \Pee^1$ are irreducible, but $X/\iota$ has two ordinary double points at the fixed points of $\iota$ and is otherwise smooth. 

  Then as in \cite{Horikawa}, \cite{FriedSO3}, the blowup of $X$ at the two fixed points is obtained as follows: let $\mathbb{F}_N$ be a Hirzebruch surface, $N \ge 0$,  and let $\sigma_0$ be the negative section if $N> 0$ and any section of square $0$ in case $N=0$. Blow up a point $p$ in a fiber $f$, such that $p\notin \sigma_0$, and denote the exceptional curve by $d_1'$ and the proper transform $f-d_1'$ by $d_2'$. Note that we can always assume that we are in the case $p\notin \sigma_0$, because if $p\in \sigma_0$ then the proper transform of $\sigma_0$ has self-intersection $-N-1$ and the proper transform of the fiber is a disjoint exceptional curve.  Thus we could have started instead  with the elementary transform $\mathbb{F}_{N+1}$ of $\mathbb{F}_N$ and a point not in the negative section. In particular, in case $N=0$, there is always a section $\sigma_0$ of square $0$ passing through a point $p$. Hence we can in fact always assume that $N \ge 1$ in what follows. Finally blow up the intersection point of $d_1'$ and $d_2'$ and let $\widetilde{\mathbb{F}}_N$ be the resulting surface. Denote the new exceptional curve by $e$ and the proper transforms of $d_1', d_2'$ by $d_1$ and $d_2$. Then $d_1^2 = d_2^2 = -2$, $e^2 =-1$ and $f = d_1 + 2e + d_2$, where $f$ is the class of a general fiber of $\widetilde{\mathbb{F}}_N \to \Pee^1$. Identifying $\sigma_0$ with its preimage in $\widetilde{\mathbb{F}}_N$, $d_2\cdot \sigma_0 = 1$, $d_1\cdot \sigma_0 = e\cdot \sigma_0 = 0$, and $e\cdot d_1 = e\cdot d_2 =1$. Also,
$$K_{\widetilde{\mathbb{F}}_N} = -2\sigma_0 -(N+2)f +d_1 + 2e.$$

Note: the labeling of the curves $d_1, d_2$ is the opposite of that in \cite{FriedSO3}.

Choose a smooth $B_0 \in |4\sigma_0 + (2k+1)f -4e -2d_1|= |4\sigma_0 + (2k-1)f+2d_2|$ not containing $\sigma_0$ as a component and let $B = B_0 + d_1 + d_2$ as effective curves. Thus as divisor  classes
\begin{align*}
B &= B_0 + d_1 + d_2 = 4\sigma_0 +  2k f -4e -2d_1 + f+ d_1+ d_2\\
&= 4\sigma_0 +  2k f -2e + 2d_2.
\end{align*}
Note that $B_0\cdot e = 2$, $B_0 \cdot d_1  = B_0\cdot d_2 = 0$,  and that 
\begin{align*}   B\cdot f &= B\cdot e =4; \\
 B\cdot \sigma_0 &= -4N + 2k +2.
\end{align*}
Then   $(B_0-d_i)\cdot d_i = -d_i^2 >0$. Since $B_0$ is smooth, $d_1$ and $d_2$ are not components of $B_0$ and hence are disjoint from $B_0$.
Also, since   $\sigma_0$ is not a component of $B_0$,   
$$B_0 \cdot \sigma_0 = B\cdot \sigma_0 -1= -4N + 2k +1\ge 0,$$
and hence $k \ge 2N -\frac12$. Since $k$ is an integer,
$k \ge 2N$.

The divisor class $B$ is uniquely divisible by $2$, with $\frac12B = 2\sigma_0 + k f -e + d_2$. Thus
$$K_{\widetilde{\mathbb{F}}_N} + \textstyle\frac12B = (k-N-2)f + d_1 + e + d_2 = (k-N-1)f-e.$$
We can then form the double cover $\nu\colon \widetilde{X} \to \widetilde{\mathbb{F}}_N$. The inverse images $e_i$ of the $d_i$ become exceptional curves, since $\nu^*d_i = 2e_i$ and hence $4e_i^2 = 2 (-2) = -4$. Contracting these, we get an elliptic surface $X$ with a multiple fiber $F$.  Here, letting $f$ denote both a general fiber on $\mathbb{F}_N$, $\widetilde{\mathbb{F}}_N$, $X$, or $\widetilde{X}$, $\nu^*f = f$. Moreover, $f \equiv 2F = 2F' +2e_1 + 2e_2$ and $\nu^*e = F'$. 
Since $p_g(X) = h^0(\widetilde{\mathbb{F}}_N; \scrO_{\widetilde{\mathbb{F}}_N}(K_{\widetilde{\mathbb{F}}_N}+\frac12B))$,
$$p_g(X) = \begin{cases} k-N -1, &\text{if $k-N \ge 1$}; \\
0,  &\text{otherwise}.
\end{cases}$$
Note however that if $k-N =0$, i.e.\ $k=N$, then since $k \ge 2N$, we must have $k=N =0$.  So we can ignore this case in what follows since we are only interested in the case $N \ge 1$. Then  $p_g = p_g(X) = k-N -1$ and $k\ge 2N$, so $N \le p_g + 1$ and $k = p_g+1+ N$. This is in agreement with:
\begin{align*}
K_{\widetilde{X}}&= \nu^* (K_{\widetilde{\mathbb{F}}_N} + \textstyle\frac12B)  = (k-N-1)\nu^*f-F' \\
&= (k-N-2)f + 2F' +2e_1 + 2e_2 -F' = (p_g-1)f + F' +2e_1 + 2e_2 \\
&= (p_g-1)f + F + e_1+e_2.
\end{align*}

Also, 
$$B \cdot \sigma_0 = -4N + 2k +2 = 2(k-N) -2N + 2= 2p_g -2N+4.$$
In particular, if $\widetilde{\Sigma}$ is the preimage of $\sigma_0$ on the double cover $\widetilde{X}$, then $\widetilde{\Sigma}^2 = -2N$ and $p_a(\widetilde{\Sigma}) =  p_g - N +1$. Since $\sigma_0 \cdot d_2 = 1$ and $\sigma_0 \cdot d_1 =0$, $\widetilde{\Sigma}$ meets the exceptional curve $e_2$ transversally at a smooth point and is disjoint from $e_1$. Thus, setting $\Sigma$ to be the image of $\widetilde{\Sigma}$ in $X$, $\Sigma \cong \widetilde{\Sigma}$.  Hence $p_a(\Sigma) =  p_g - N +1$, and 
$$\Sigma^2 = \widetilde{\Sigma}^2 + 1 = -2N + 1.$$

It is easy to see that a divisor class $\Sigma$  is automatically effective as soon as $\Sigma^2  \ge -1$. From our point of view, we want to see when there is an effective bisection $\Sigma$ with $\Sigma^2 < -1$ as well. The above shows that, as $N$ increases by $1$, $\Sigma ^2$ drops by $2$ and $p_a(\Sigma)$ by $1$. In particular, as $N$ increases  from $1$ to $p_g + 1$, $\Sigma^2$ decreases from $-1$ to $-2p_g -1$ and $p_a(\Sigma)$ decreases from $p_g$ to $0$. Note that the case $N=1$ is the generic case. More generally, for $1\le N \le p_g+1$, the above constructs a stratum of the moduli space which has codimension one in the stratum for $N-1$.  It would be interesting to find an intrinsic characterization of the strata.

In the cases of interest, $p_g =1$ or $0$. In case $p_g=1$,   the only possibilities for $\mathbb{F}_N$ are $N =1$, $k=3$, $\Sigma^2 =-1$, $p_a(\Sigma) =1$, or $N =2$, $k=4$, $\Sigma^2 =-3$, $p_a(\Sigma) =0$ (and hence $\Sigma$ is smooth).  In case $p_g =0$, we must have $N=1$,  $k=2$, $\Sigma^2 =-1$, and $p_a(\Sigma) =0$. 

A similar  description applies if we allow more complicated non-multiple fibers. We shall just describe the following situation which will arise in \S\ref{ssect3.4} (per the statement of Theorem~\ref{thm17cusp} below):

\begin{assumption}\label{assump12} The multiple fiber $F$ is irreducible,  there is a nonmultiple fiber  $f_t\subseteq X$ which is a I${}_a$ fiber, $a\ge 3$ and the bisection $\Sigma$ meets two different components of $f_t$, necessarily transversally. Every other nonmultiple fiber    of $\pi$ has at most two components, and, in case there are two components, $\Sigma$ has intersection number one with each component.  Finally, assume that $\Sigma$ is rational, so that the map $\Sigma \to \Pee^1$ is branched at the point corresponding to $F$ and at one other point.  
\end{assumption} 

 Under this assumption, the involution $\iota$ then exchanges the two components of $f_t$ meeting $\Sigma$, as well as all of the others except possibly one or two ``end" components where it induces a degree $2$ map onto a quotient $\Pee^1$. If $\overline{X}$ is the surface obtained by contracting all of the components of $f_t$ not meeting $\Sigma$ and $\widetilde{\overline{X}}$ is the blowup at the two isolated fixed points on $F$, $\iota$ induces a finite morphism  $\nu \colon \widetilde{\overline{X}} \to \widetilde{\mathbb{F}}_N$, where $\widetilde{\mathbb{F}}_N$ is the blowup of $\mathbb{F}_N$ at the two points corresponding to the isolated fixed points of $\iota$.   Let $\overline{B}_0$ be the branch locus of $\nu$. Then $\overline{B}_0$ has two curve singularities of type $A_{k_1}$, $A_{k_2}$ lying over the point $t\in \Pee^1$, with $k_1, k_2\ge 1$ since $\Sigma$ meets two distinct components of the fiber on $X$. An embedded resolution of the singularities of $\overline{B}_0$ leads to  a surface $\widetilde{\widetilde{\mathbb{F}}}_N$ with  two chains of rational curves, of lengths $[k_1+1/2]$ and $[k_2+1/2]$ respectively. If $B_0$ is the proper transform of $\overline{B}_0$, then $\widetilde{X}$ is the double cover of $\widetilde{\widetilde{\mathbb{F}}}_N$ branched along $B_0 +   d_1 + d_2$, and the inverse images of the two chains of rational curves on  $\widetilde{\widetilde{\mathbb{F}}}_N$ are the minimal resolutions of (surface) singularities of types $A_{k_1}$ and  $A_{k_2}$. The fiber over $t$ is a I${}_a$ fiber, with $a = k_1+ k_2 + 2$. 
 
In this case, we identify  $\sigma_0$ with its the proper transform   on  $\widetilde{\widetilde{\mathbb{F}}}_N$. Then $\widetilde\Sigma =\nu^*\sigma_0$ and $\Sigma$ is the image of $\widetilde\Sigma$ in $X$.  In the application to the proof of Theorem~\ref{thm17cusp}, $\Sigma$ will be a smooth rational curve with $\Sigma^2 = -3$. Thus the degree two morphism $\Sigma \to \Pee^1$ induced by $\pi$ is branched at $\Sigma\cap F$ and at one other point. Hence   $\sigma_0 \cdot B_0 =  1$.  Moreover, $\overline{B}_0$ is a $4$-section, so that $f \cdot B_0 =  f \cdot \overline{B}_0=4$. Finally, by construction $e \cdot B_0 =  e \cdot \overline{B}_0=2$.

 \subsection{A vanishing theorem: the case where $D$ is irreducible}\label{ssect3.3} 
 
  For the rest of this section, we assume that  $\hY$ is the minimal resolution of an I-surface with $\kappa(\hY) = 1$.  In particular, $\hY$ is minimal and either there exists an irreducible  bisection $D$    with $D^2 = -1$ and $p_a(D) = 1$ or there exists a  smooth rational bisection $E_1$ with $E_1^2=-3$.  Thus we can apply the previous results to $X =\hY$.  First we note the following, recalling that   Assumption~\ref{assumption27},  that there are no rational double points on $Y$, is always in force: 
  
 \begin{lemma}\label{2compfibers} \begin{enumerate} \item[\rm(i)] If   $D$ is irreducible, then $F$ is irreducible and every fiber of the morphism $\hY \to \Pee^1$ has at most two components. In particular, every non-multiple fiber is reduced, and is either a nodal or a cuspidal rational curve, if irreducible, or an  I${}_2$ fiber or a tacnodal fiber, if there are two components.
 
 \item[\rm(ii)] If $D$ is reducible, then exactly one component of $D$ is an irreducible bisection and the remaining components are contained in a single fiber $f_t$. In this case, $F$ is irreducible and all nonmultiple fibers except for $f_t$  have at most two components and are  nodal, cuspidal, or a tacnode. Moreover, all but at most two components of $f_t$ are contained in $D$. Finally, if there are two such components of $f_t$, they meet at one point, they are exchanged by the involution $\iota$ on $\hY$, and the point where they meet is a fixed point of $\iota$.
\end {enumerate}
 \end{lemma} 
 \begin{proof} By assumption, if $\Gamma$ is an irreducible curve on $\hY$ and $\Gamma $ is not a component of  $D$, then $(F+D)\cdot \Gamma >0$. However, if $\Gamma$ is a component of a reducible fiber, then $F\cdot \Gamma =0$ so that $D \cdot \Gamma > 0$. But $D\cdot F =1$ and $D\cdot f =2$, where $f$ is any non-multiple fiber. It follows that $F$ has just one component and every non-multiple fiber not containing a component of $D$ has at most two. By  the classification of singular fibers, every non-multiple fiber with at most two components is nodal, cuspidal, or a tacnode, and in particular is reduced. If $D$ is reducible, then exactly one component is a bisection, necessarily smooth rational of square $-3$, and the remaining components therefore form a chain of $-2$-curves. In particular, since the chain is connected and has intersection number $0$ with a general fiber, it must be contained in a single fiber $f_t$. Since $D\cdot f =2$ and $D\cdot C > 0$ for every component of $f_t$ which is not a component of $D$, it is clear that there can be at most two such components. If there are two such, they form a connected component of $f_t -D$ and the one or two other components which meet them are the components meeting the bisection $E_1$. Thus they are exchanged  by $\iota$ and their intersection point is a fixed point of $\iota$.
 \end{proof}

 In this subsection, we consider the irreducible case and show: 
 
 \begin{theorem}\label{thm17}  If $D$ is irreducible and there are no cusp or tacnodal fibers, then $H^2(\hY; T_{\hY}(-D)) =0$. 
 \end{theorem}
 \begin{proof}  Since $K_{\hY}\otimes \scrO_{\hY}(D) = \scrO_{\hY}(F+ D)$, by duality it suffices to show that 
 $$H^0(\hY; \Omega^1_{\hY}\otimes K_{\hY}\otimes \scrO_{\hY}(D)) =H^0(\hY; \Omega^1_{\hY}  \otimes \scrO_{\hY}(F+ D))= 0 .$$
 By  Theorem~\ref{tangbundell}, there is an exact sequence
 $$0 \to \scrO_{\hY}(-2f+F) \to \Omega^1_{\hY} \to I_Z\otimes \scrO_{\hY}(2f) \to 0,$$
 where $Z$ is the scheme defined by the singular points of the singular, non-multiple fibers.  Hence $\ell(Z) = 24$ and  by our assumptions $Z$ consists of $24$ distinct points.  Thus it suffices to  show that
 $$H^0(\hY; \scrO_{\hY}(-2f+2F+D)) = H^0(\hY; I_Z\otimes \scrO_{\hY}(2f+F+D)) = 0.$$
 First, since $-2f+2F+D = D-f$,  $H^0(\hY; \scrO_{\hY}(-2f+2F+D)) = 0$ as $D-f$ is not effective (because $D^2< 0$). So it is enough to show that $H^0(\hY; I_Z\otimes \scrO_{\hY}(2f+F+D)) = 0$.  Referring to \S\ref{doubcovsect},  there is a blowup $\hhY$ of $\hY$ at two smooth points of $F$ and a degree $2$ morphism $\nu\colon \hhY \to \widetilde{\mathbb{F}}_1$, branched along a divisor in $\widetilde{\mathbb{F}}_1$ of the form $B_0 + d_1 + d_2$, where $B_0\in |4\sigma_0 + 5f + 2d_2|$. Here, if $e_1, e_2$ are the exceptional curves on $\hhY$ and $F'$ is the proper transform of the fiber, then $\nu^*\sigma_0 = D$, $\nu^*d_i = 2e_i$, $\nu^*e= F'$, and (using $f$ to denote a general fiber of $\hhY \to \Pee^1$ as well as $\widetilde{\mathbb{F}}_1 \to \Pee^1$),  $\nu^*f = f$. Also, we can identify $Z$ with the corresponding subscheme of $\hhY$. The total transform of $F$ on $\hhY$ is $F'+ e_1+ e_2$, and hence
 $$H^0(\hY; I_Z\otimes \scrO_{\hY}(2f+F+D))  = H^0(\hhY; I_Z\otimes \scrO_{\hhY}(2f+F' +e_1 + e_2 +D)) .$$
 By the above remarks, since $\nu^*d_i = 2e_i$, there is an inclusion
 $$\scrO_{\hhY}(2f+F' +e_1 + e_2 +D) \subseteq \nu^*\scrO_{\widetilde{\mathbb{F}}_1}(2f+e +d_1 + d_2 +\sigma_0).$$
 In what follows, we will repeatedly use the following for divisor classes on $\widetilde{\mathbb{F}}_1$:
 $$2f+e +d_1 + d_2 +\sigma_0 = \sigma_0 + 3f-e.$$
 For every line bundle $M$ on $\widetilde{\mathbb{F}}_1$, we have $\nu_*\nu^*M =M \oplus (M\otimes  \scrO_{\widetilde{\mathbb{F}}_1}(-\frac12B))$, 
 where $\frac12B $ is the divisor $2\sigma_0 + 3f -e + d_2$. Applying this to $\scrO_{\widetilde{\mathbb{F}}_1}(\sigma_0 + 3f-e)$, first note that
 $$\sigma_0 + 3f-e -(2\sigma_0 + 3f -e + d_2) = -\sigma_0    - d_2. $$
 Since this last divisor is not effective, every section of $\nu^*\scrO_{\widetilde{\mathbb{F}}_1}(\sigma_0 + 3f-e)$ is the pullback of a section $s$ of $\scrO_{\widetilde{\mathbb{F}}_1}(\sigma_0 + 3f-e)$, and we need to find sections vanishing on $\nu(Z)$.  If $x\in Z$, then $x$ is    singular point of    the    fiber of $\hY\to \Pee^1$ which contains it. Viewing $x$ as a  point on $\hhY$,  $x$  is the preimage of a point on a fiber of $\widetilde{\mathbb{F}}_1$ which is tangent to $B_0$ and the   point $\nu(x)$ is a point of tangency. In particular, $\nu(Z)$ consists of $24$ points and the section $s$ has to vanish at those points of  $B_0$. (This includes the case where $F$ is singular; since $B_0 \cdot e = 2$ and $d_1$, $d_2$ meet $e$ transversally, $F$ is singular $\iff$ $B_0$ meets $e$ doubly at a single point.) Note that 
 $$(\sigma_0 + 3f-e) - B_0 = -3\sigma_0 -2f -2d_2 -e$$
 is not effective, and so a section $s$ of $\scrO_{\widetilde{\mathbb{F}}_1}(\sigma_0 + 3f-e)$ vanishing along $B_0$ has to vanish identically. Conversely, if $s\neq 0$, then the curve $(s)$ defined by $s$ satisfies: $\#((s) \cap B_0)$ is finite and  $\geq 24$. Hence   $B_0 \cdot (\sigma_0 + 3f-e) \geq 24$. But
 \begin{gather*}
 B_0 \cdot (\sigma_0 + 3f-e) = (4\sigma_0 + 5f + 2d_2) \cdot (\sigma_0 + 3f-e )\\
 = -4 + 12+5 + 2-2 = 13 < 24. 
 \end{gather*}
 Thus, a section $\sigma$ of $\scrO_{\widetilde{\mathbb{F}}_1}(\sigma_0 + 3f-e)$ vanishing at the $24$ points of $\nu(Z)$ must be identically $0$.  Hence $H^0(\hY; I_Z\otimes \scrO_{\hY}(2f+F+D)) = 0$, so that $H^2(\hY; T_{\hY}(-D)) =0$. 
 \end{proof} 
 
 \begin{remark} The proof still works if there are a small number of cuspidal or tacnodal fibers. Concretely, let $c$ be the number of cuspidal fibers and $t$ the number of tacnodal fibers. Then the above proof works as long as $c+2t \le 10$.  It is possible that this, or a more detailed analysis,  is enough to handle the general case. 
 \end{remark} 
 
 \begin{corollary}\label{kappa1cor1} Under the assumptions of Theorem~\ref{thm17}, the deformations of $Y$ are versal for the deformations of the simple elliptic singularity or cusp singularity, i.e.\ the homomorphism $\mathbb{T}^1_Y \to H^0(Y; T^1_Y)$ is surjective. \qed
 \end{corollary}

\subsection{A vanishing theorem: the case where $D$ is reducible}\label{ssect3.4}  We prove an analogue of Theorem~\ref{thm17} in the case of a cusp singularity with $r> 1$, again under some mild general position assumptions.

\begin{theorem}\label{thm17cusp}  Suppose that  all fibers of $\hY$ are nodal. Let $D=E_1 + \cdots + E_r$ correspond  to a cusp singularity, where $E_1$ is a smooth bisection with $E_1^2=-3$ and the remaining  components are fiber components of the fiber $f_t$, with $r\le 14$. Then $H^2(\hY; T_{\hY}(-\log D)) =0$. 
 \end{theorem}
 \begin{proof} The proof is similar to that of Theorem~\ref{thm17}.  We must show that $H^0(\hY; \Omega^1_{\hY}(\log D) \otimes K_{\hY}) = 0$. Let $\nu \colon \hhY \to \widetilde{\mathbb{F}}_2$ be the double cover morphism of \S\ref{doubcovsect},  where  $\widetilde{\widetilde{\mathbb{F}}}_2$ is the blowup of $\mathbb{F}_2$ described at the end of \S\ref{doubcovsect}.  The exact sequence of Theorem~\ref{tangbundelllog} gives an exact sequence
 $$ 0 \to \scrO_{\hY}(-2f+2 F + \sum_{i>1}E_i) \to \Omega^1_{\hY}(\log D) \otimes K_{\hY} \to I_W\otimes \scrO_{\hY}(2f+F + E_1)\to 0,$$
 with $\ell(W)  = 25-r$.  Moreover, if $x\in \operatorname{Supp} W$, then either $x$ is a singular point of a fiber, not on $D$,  or $x$ is a ramification point for the morphism $E_1\to \Pee^1$ not on $F$. Since $E_1\to \Pee^1$ is a double cover with $E_1$ a  smooth rational curve, the morphism $E_1\to \Pee^1$  ramifies at the intersection $E_1\cap F$ and at one other point. If $x\in \operatorname{Supp} W$,  $x$ is a singular point of its fiber, and $x\notin D$, then    $x\in B_0$ by Lemma~\ref{2compfibers}(ii). If $x$ is the unique ramification point of $E_1\to \Pee^1$ not on $F$, then    again $x\in B_0$, and $W$ has length at most $2$ at $x$. Thus, if $\overline{W} = \nu(W)$, then  $\overline{W}\subseteq B_0$ and 
 $$\#(\overline{W}) \ge \ell(W) -1= 24-r.$$

 As $-2f+2 F + \sum_{i>1}E_i = -f+ \sum_{i>1}E_i$ is the negative of an effective divisor,
 $$H^0(\hY; \scrO_{\hY}(-2f+2 F + \sum_{i>1}E_i) ) =0.$$
We must show that $H^0(\hY;  I_W\otimes \scrO_{\hY}(2f+F + E_1)) = 0$ as well.  It suffices to show that 
 $$H^0(\hhY; I_W\otimes \scrO_{\hhY}(2f+F' +e_1 + e_2 +E_1))=0,$$
 which will follow if we show  that 
$$H^0(\hhY;   I_W\otimes \nu^*\scrO_{\widetilde{\widetilde{\mathbb{F}}}_2}(2f+e +d_1 + d_2 +\sigma_0))=0.$$
We again use the equality of divisor classes
$$2f+e +d_1 + d_2 +\sigma_0 = \sigma_0 +3f - e.$$
 Arguing as in the proof of Theorem~\ref{thm17}, 
$$H^0(\hhY; \nu^*\scrO_{\widetilde{\widetilde{\mathbb{F}}}_2}(\sigma_0 +3f - e)) = H^0(\widetilde{\widetilde{\mathbb{F}}}_2; \scrO_{\widetilde{\widetilde{\mathbb{F}}}_2}(\sigma_0 +3f - e)).$$
It suffices to show that $H^0(\widetilde{\widetilde{\mathbb{F}}}_2; I_{\overline{W}}\otimes \scrO_{\widetilde{\widetilde{\mathbb{F}}}_2}(\sigma_0 +3f - e))= 0$. Suppose by contradiction that  there is  a nonzero section $s$ of $\scrO_{\widetilde{\widetilde{\mathbb{F}}}_2}(\sigma_0 +3f - e)$ vanishing along $\overline{W}$.   As in the proof of Theorem~\ref{thm17}, such a section cannot vanish identically along $B_0$. Now we compute as before, using the computations at the end of \S\ref{doubcovsect}:
$$B_0 \cdot (\sigma_0 +3f - e) = (B_0\cdot \sigma_0) + 3(B_0\cdot f) - 2(B_0\cdot e) = 1   +12-4 = 9.$$
In particular, we must have $\#(\overline{W}) \le 9$. But we have seen that $\#(\overline{W}) \ge  24-r$, hence $r\ge 15$. Conversely, if $r\le 14$, then $H^0(\hY;  I_W\otimes \scrO_{\hY}(2f+F + E_1)) = 0$.  Thus $H^2(\hY; T_{\hY}(-\log D)) =0$.
 \end{proof}
 
 \begin{corollary}\label{kappa1cor2} Under the assumptions of Theorem~\ref{thm17cusp}, the deformations of $Y$ are versal for the deformations of the cusp singularity, i.e.\ the homomorphism $\mathbb{T}^1_Y \to H^0(Y; T^1_Y)$ is surjective. \qed
 \end{corollary}

\section{The case $\kappa(\hY) = 0$}\label{section5}

\subsection{Some examples of lattices} We begin by defining various lattices which will arise in what follows.

\begin{definition}\label{somelattices}  (i) Let $\Lambda_0(n)$ be the lattice with basis vectors $e_0, \dots, e_n$, satisfying: $e_0^2 =0$, $e_i^2=-2$ for $i>0$, $e_i\cdot e_j = 0$ for $i\neq j\pm 1 \pmod{n}$, and $e_i\cdot e_{i+1} =1$ (including the case $e_n\cdot e_0 =1$). 

\smallskip
\noindent (ii)  Let $\Lambda_1(n,m)$ be the lattice spanned by vectors $e_1, \dots e_ n, f_1, \dots, f_m, g_1, g_2$ with $e_i^2 = f_i^2 = g_i^2 = -2$, $e_1\cdot e_2 =\dots = e_{n-1}\cdot e_n =1$,  $f_1\cdot f_2 =\dots = f_{m-1}\cdot f_m =1$, $e_1 \cdot g_1 =f_1\cdot g_1 =1$, $e_n \cdot g_2 =f_m\cdot g_2 =1$,  $g_1\cdot g_2 = 1$, and all other intersections are $0$.

\smallskip
\noindent (iii) Let $\Lambda_2(n,m)$ be the lattice spanned by vectors $e_0, \dots, e_ n, f_0,  \dots, f_m$ with $e_i^2 = f_j^2=-2$,  $e_i\cdot e_j = 0$ for $i\neq j\pm 1 \pmod{n}$, and $e_i\cdot e_{i+1} =1$ (including the case $e_n\cdot e_0 =1$), similarly  $f_i\cdot f_j = 0$ for $i\neq j\pm 1 \pmod{m}$, and $f_i\cdot f_{i+1} =1$ (including the case $f_n\cdot f_0 =1$), $e_0\cdot f_0 = 1$, and $e_i\cdot f_j =0$ if $(i,j) \neq (0,0)$. 
\end{definition} 

A straightforward argument shows:

\begin{lemma}\label{somelatticeslemma} The lattices $\Lambda_0(n)$, $\Lambda_1(n,m)$, and $\Lambda_2(n,m)$ are nondegenerate of signatures $(1,n)$, $(1, n+m+1)$, and $(1, n+m+1)$ respectively. \qed
\end{lemma} 

\subsection{A lemma on double covers}

\begin{lemma}\label{doubcov} Suppose that $\nu \colon W \to Z$ is a double cover branched along $\Sigma$, where $Z$, $W$, and $\Sigma$ are all smooth. Let $\lambda$ be the line bundle defining the double cover, i.e.\ $\lambda^{\otimes 2} \cong \scrO_Z(\Sigma)$. Then there is an exact sequence
$$0 \to \Omega^1_Z \to \nu_*\Omega^1_W \to \Omega^1_Z(\log \Sigma) \otimes \lambda^{-1} \to 0.$$
\end{lemma}
\begin{proof}  Choose local holomorphic coordinates $z_1, \dots, z_n$ near a point of $\Sigma$ so that $\Sigma$ is defined by $z_n =0$ and the map $\nu$ is given by $z_i = w_1$, $i < n$, and $z_n = w_n^2$, i.e.\ $w_n =\sqrt{z_n}$. Thus $z_n$ is a local section of $\scrO_Z(-\Sigma)$ and $\sqrt{z_n}$ is a local section of $L^{-1}$. Then
 $dz_n = 2w_n \, dw_n$, i.e.\ $\displaystyle  dw_n = \frac{dz_n}{2\sqrt{z_n}}$.  Thus a basis for $\nu_*\Omega^1_W$ as an $\scrO_Z$-module is given locally by
 $$dz_1, \sqrt{z_n}dz_1, \dots, dz_{n-1}, \sqrt{z_n}dz_{n-1}, dz_n,  \frac{dz_n}{\sqrt{z_n}}.$$
 Hence the quotient of $\nu_*\Omega^1_W$ by the subspace  $\Omega^1_Z$ has as a basis
 $$\sqrt{z_n}dz_1, \dots, \frac{dz_n}{\sqrt{z_n}}= \sqrt{z_n}\cdot \frac{dz_n}{z_n}.$$
 This says invariantly that the quotient is $\Omega^1_Z(\log \Sigma) \otimes \lambda^{-1}$.
\end{proof}

\begin{remark} In fact, the natural involution on $\nu_*\Omega^1_W$ coming from the double cover involution induces a splitting
$$\nu_*\Omega^1_W \cong \Omega^1_Z  \oplus \Big(\Omega^1_Z(\log \Sigma) \otimes \lambda^{-1}\Big).$$
\end{remark}

\subsection{The case where $\hY$ is the blowup of a $K3$ surface}

Throughout this subsection, we assume  that the minimal model $Y_0$ of $\hY$ is a (smooth) $K3$ surface. By Theorem~\ref{theorem5}(iii), there exists a morphism $\rho\colon \hY\to Y_0$, where    $\rho$ is the blowup of $Y_0$ at a single point $p\in Y_0$, and the exceptional curve $C$ on $\hY$ satisfies: $C\cdot D =2$. In particular, $K_{\hY} = \scrO_{\hY}(C)$. 

 First, we consider the case where $D$ is irreducible, and is thus either an elliptic curve or a nodal rational curve. 
 
 \begin{theorem}\label{K3irred} If  $D$ is irreducible, then $H^2(\hY; T_{\hY}(-D)) = 0$.  Hence the homomorphism $\mathbb{T}^1_Y \to H^0(Y; T^1_Y)$ is surjective.
 \end{theorem}
 \begin{proof}  With notation as above, let $\overline{D}=\rho(D)$ be the image of $D$ in $Y_0$. Then $\overline{D}$ is an irreducible curve with either a node or a cusp, and $p_a(\overline{D}) = 2$. Thus $\overline{D}^2 = 2$ and the complete linear system $|\overline{D}|$ defines a morphism $\varphi\colon Y_0 \to \Pee^2$, for which $\scrO_{Y_0}(\overline{D}) =\varphi^*\scrO_{\Pee^2}(1)$. By hypothesis, $L =\scrO_{\hY}(D+C)$ induces an ample divisor on $Y$. Thus, if $\Gamma$ is an irreducible curve in $\hY$ not equal to $C$ or $D$, then either $D\cdot \Gamma > 0$ or $C\cdot \Gamma > 0$,  so for every  irreducible curve  $\Gamma_0$ on $Y_0$, 
 $$\Gamma_0\cdot \overline{D} = \Gamma_0\cdot \rho_*D = \rho^*\Gamma_0\cdot (D+2C) >0.$$
 Hence $\varphi\colon Y_0 \to \Pee^2$ is a finite degree $2$ morphism  and   its  branch locus is a smooth sextic curve $\Sigma \subseteq \Pee^2$. Thus the square root $\lambda$ of $\scrO_{\Pee^2}(\Sigma)$ is $\scrO_{\Pee^2}(3)$.

 By Serre duality, $H^2(\hY; T_{\hY}(-D))$ is dual to 
 $H^0(\hY; \Omega^1_{\hY}(D) \otimes \scrO_{\hY}(C))$. A   section of  $\Omega^1_{\hY}(D) \otimes \scrO_{\hY}(C)$ defines a section of $\Omega^1_{Y_0}(\overline{D})$ over $Y_0-\{p\}$ which then extends to a holomorphic section of $\Omega^1_{Y_0}(\overline{D})$ over $Y_0$ by Hartogs' theorem. Thus it will suffice to prove that $H^0(Y_0; \Omega^1_{Y_0}(\overline{D})) =0$. Equivalently, since $\scrO_{Y_0}(\overline{D}) =\varphi^*\scrO_{\Pee^2}(1)$, it suffices to show that  $H^0(\varphi_*\Omega^1_{Y_0}\otimes \scrO_{\Pee^2}(1)) =0$. By Lemma~\ref{doubcov}, since $\lambda^{-1} =\scrO_{\Pee^2}(-3)$, it is enough to show that
 $$H^0(\Pee^2; \Omega^1_{\Pee^2}\otimes \scrO_{\Pee^2}(1)) = H^0(\Pee^2; \Omega^1_{\Pee^2}(\log \Sigma) \otimes \scrO_{\Pee^2}(-2)) = 0.$$
The fact that $H^0(\Pee^2; \Omega^1_{\Pee^2}\otimes \scrO_{\Pee^2}(1))=0$ is well-known and follows from looking at the global sections coming from the Euler exact sequence
$$0 \to \Omega^1_{\Pee^2}\otimes \scrO_{\Pee^2}(1) \to  \left(\scrO_{\Pee^2} \right)^3 \to \scrO_{\Pee^2}(1) \to 0.$$
As for the remaining equality $H^0(\Pee^2; \Omega^1_{\Pee^2}(\log \Sigma) \otimes \scrO_{\Pee^2}(-2)) = 0$, the Poincar\'e residue  sequence gives
$$0 \to \Omega^1_{\Pee^2}(-2) \to \Omega^1_{\Pee^2}(\log \Sigma)(-2) \to \scrO_{\Sigma} (-2)\to 0.$$
On global sections, $H^0(\Pee^2; \Omega^1_{\Pee^2}(-2)) =0$ via the Euler exact sequence,  and   $H^0( \scrO_{\Sigma} (-2))=0$ because  $\scrO_{\Pee^2}(-2)|\Sigma$ is a line bundle of negative degree. Thus $H^2(\hY; T_{\hY}(-D) =0$ as claimed. 
\end{proof}

In case $D$ is reducible, we have the following which uses a mild general position assumption.

\begin{theorem}\label{K3red} Suppose that   $D$ is reducible, and assume that, if there is a component $E$ of $D$ of self-intersection $-4$, then the exceptional curve $C$ meets $E$ in $2$ distinct points. 
Then $H^2(\hY; T_{\hY}(-\log D)) = 0$, and hence the homomorphism $\mathbb{T}^1_Y \to H^0(Y; T^1_Y)$ is surjective.
 \end{theorem}
 \begin{proof} Since all components of $D$ are smooth rational, $C$ is disjoint from every component of square $-2$. If there is a component $E$ of self-intersection $-4$, then $C\cdot E =2$, and if  there are two components of square $-3$, then $C$ meets each of them transversally at one point. Thus, the hypothesis of the theorem is exactly that the curve $C+D$ is nodal on $\hY$, or equivalently that $\overline{D} = \rho(D)$ is nodal. 
 
 The group  $H^2(\hY; T_{\hY}(-\log D))$ is dual to $H^0(\hY; \Omega^1_{\hY}(\log D)\otimes \scrO_{\hY}(C))$. As in the proof of Theorem~\ref{K3irred}, it suffices to prove that 
 $H^0(Y_0; \Omega^1_{Y_0}(\log \overline{D}) ) = 0$. Thus it suffices to prove that the fundamental classes of the components of $\overline{D}$ are independent in cohomology. This follows from Lemma~\ref{somelatticeslemma}. 
 \end{proof}
 
 \begin{corollary}\label{K3toell}  Suppose that there is a component $E$ of $D$ of self-intersection $-4$ with $r\ge 2$ or two components of self-intersection $-3$ with $r\ge 3$. Then there is a deformation of $Y$ to an I-surface $Y_t$ with a cusp or simple elliptic singularity of multiplicity $1$. Hence the minimal resolution $\hY_t$ is a surface with $\kappa(\hY_t) = 1$. 
 \end{corollary} 
  \begin{proof} This follows immediately from Theorem~\ref{K3red} and Theorem~\ref{cuspadj}(iii). 
   \end{proof}
 
 Before we turn to an example showing that there are indeed examples satisfying the hypothesis of Corollary~\ref{K3toell}, we note the following standard result:
 
 \begin{lemma}\label{defratK3}  Let $X$ be a $K3$ surface and let $C_1, \dots, C_k$ be smooth rational curves on $X$ meeting transversally  whose fundamental classes are linearly independent in cohomology. Let $C = \sum_iC_i$. Then $H^2(X; T_X(-\log C)) =0$ and hence the map $H^1(X;T_X) \to \bigoplus_iH^1(C_i; N_{C_i/X})$ is surjective.
 \end{lemma} 
 \begin{proof} By Serre duality, we have to show that $H^0(X; \Omega^1_X(\log C)) = 0$, which is an easy consequence of the Poincar\'e residue sequence.
 \end{proof}
 
 \begin{example} Let $Q$ be a smooth plane quartic curve and let $L$ be a tangent line to $Q$ which is not a bitangent, so that $L$ meets $Q$ doubly at $p_0$ and transversally at $p_1, p_2$. Then there exists a conic $C$ which passes through $p_1$ and $p_2$ and meets $Q$ transversally at $6$ remaining points $p_3, \dots, p_8$. If $Y_0$ is the minimal resolution of the double cover of $\Pee^2$ branched along $Q+C$, then $X$ contains $8$ $(-2)$ curves $E_1, \dots, E_8$ corresponding to the $8$ points of $Q\cap C$. The line $L$ splits into two $(-2)$-curves $G_2, G_2$, which are disjoint from $E_3, \dots, E_8$. Moreover $G_1\cdot G_2 =1$, corresponding to the point $p_0$, $G_i\cdot E_1 = G_i\cdot E_2 = 1$, and, for $i=1,2$,  the intersection points $G_1\cap E_i$ and $G_2\cap E_i$ are distinct since $L$ is not tangent to $Q$ or $C$ at $p_i$. 
 
 Blowing up $Y_0$ at the point $G_1\cap G_2$ yields a surface $\hY$ which is the minimal resolution of an I-surface $Y$ with a cusp singularity with intersection sequence $(-3, -2, -3, -2)$ as well as an exceptional curve $C$ and $6$ ordinary double points corresponding to the curves $E_3, \dots, E_8$. Here $L = G_1+ G_2 + E_1+ E_2+C$, so that $L$ is nef and big and contracts the resolution $D = G_1+ G_2 + E_1+ E_2$ as well as the curves $E_3, \dots, E_8$.    Using Lemma~\ref{defratK3}, we can deform $Y_0$ and hence $\hY$ keeping the same configuration for the curves $E_1, E_2, G_1, G_2$ but losing the remaining curves $E_3, \dots, E_8$. The resulting I-surface $Y_t$ will have a cusp singularity with intersection sequence $(-3, -2, -3, -2)$  but no rational double point singularities. 
 \end{example}

 \begin{remark}  Let $\Lambda_1(n,m)$ be the lattice of Definition~\ref{somelattices}(ii).  If $n+m$ is not too large, there is a primitive  embedding of  $\Lambda_1(n,m)$  in the $K3$ lattice. Thus there exist $K3$ surfaces whose Picard lattice  is isomorphic to $\Lambda_1(n,m)$. The key question is whether the classes corresponding to $e_i, f_i, g_1, g_2$ are represented by irreducible curves. By standard arguments, this reduces to the following question: Is every $\alpha \in \Lambda$ with $\alpha^2=-2$ in the $W$-orbit of $e_1$, where  $W$ is the reflection group generated by the classes $e_i, f_i, g_1, g_2$? 
 \end{remark}
 
  \begin{remark}   Suppose that $E$ is irreducible of self-intersection $-4$ as above. The condition in Theorem~\ref{K3red} that the exceptional curve $C$ meets $E$ in $2$ distinct points is exactly the condition that, on the $K3$ surface $Y_0$, the irreducible curve $\overline{E}=\rho(E)$ of arithmetic genus one has a node, not a cusp. If instead $\overline{E}$ has a cusp, it is likely that a small deformation of $Y_0$, keeping the curve  $\overline{E}$ as well as the remaining components of $\overline{D}$ as effective divisors, will have the property that $\overline{E}$ is nodal. Compare \cite{CGL} for related results.
 \end{remark}
 
 \subsection{A remark on semistable smoothings}\label{sssmooth} 
 
 The method of proof above can be applied to show the existence of semistable smoothings. This picture is essentially the semistable version of one of Wahl's ``exotic" adjacencies (Theorem~\ref{cuspadj}(iii)(b)). First we recall the terminology of \cite{FriedmanSmoothings}:
 
 \begin{definition} Let $Z =Z_1\cup_DZ_2$ be the simple normal crossing surface obtained by gluing the two smooth surfaces $Z_1$ and $Z_2$ along a smooth divisor $D\subseteq Z_i$, where $D\subseteq Z_1$ is identified with $D\subseteq Z_2$ by the choice of an isomorphism $\gamma$. Then $Z$ is \textsl{$d$-semistable} if $N_{D/Z_1}\otimes \gamma^*N_{D/Z_2}\cong \scrO_D$.
 \end{definition}
 
 \begin{theorem}\label{ssWahl}  Let $X$ be a $K3$ surface with an irreducible  nodal rational curve $\overline{D}$ with $p_a(\overline{D}) =1$, so that $\overline{D}^2=0$. Let $\widetilde{X}$ be the blowup of $X$ at the singular point of $\overline{D}$ and let $D$ be the proper transform  of $\overline{D}$, so that $D^2 =-4$. Then: 
 \begin{enumerate}
 \item[\rm(i)] $H^2(\widetilde{X}; T_{\widetilde{X}}(-\log D)) = 0$.
 \item[\rm(ii)] The $d$-semistable surface $Z$ obtained by gluing $\widetilde{X}$ to $\Pee^2$ by identifying $D$ with  a smooth conic in $\Pee^2$ is smoothable, and its smoothings are elliptic surfaces with $p_g =1$ and a single multiple fiber, of multiplicity $2$. 
 \end{enumerate}
 \end{theorem}
 \begin{proof} (i) The arguments used in the proof of Theorem~\ref{K3red} show that  the dual of $H^2(\widetilde{X}; T_{\widetilde{X}}(-\log D))$ injects into $H^0(X; \Omega^1_X(\log \overline{D}))$, and this is $0$ as the fundamental class of $\overline{D}$ is nonzero in $H^1(X; \Omega^1_X)$. 
 
 \smallskip
 \noindent (ii)    Let $\nu\colon \widetilde{X} \amalg \Pee^2 \to Z$ be the normalization. By a local calculation (cf. \cite[Lemma 4.6(iv)]{FL23}, there is an exact sequence
$$0 \to T^0_Z \to \nu_*(T_{\widetilde{X}}(-\log D)  \oplus T_{\Pee^2}(-\log D) )\to i_*T_D \to 0,$$
where $i\colon D \to Z$ is the inclusion.  Then $H^1(D; T_D) =0$ since $D\cong \Pee^1$, $H^2(\widetilde{X}; T_{\widetilde{X}}(-\log D))=0$ by (i), and $H^2(\Pee^2; T_{\Pee^2}(-\log D)) =0$ by an easy direct argument. Thus $H^2(Z; T^0_Z) = 0$. By construction, $T^1_Z = \scrO_D$. Since $D\cong \Pee^1$, $H^1(Z; T^1_Z) = H^1(D; \scrO_D) =0$. Thus $\mathbb{T}_Z^2 =0$ and   $\mathbb{T}_Z^1\to H^0(Z; T^1_Z) = H^0(Z; \scrO_D) $ is surjective. In particular, there is a deformation of $Z$ over the disk whose image in $H^0(Z; T^1_Z) $ is an everywhere generating section. By   \cite[Proposition 2.5]{FriedmanSmoothings}, $Z$ is smoothable. We omit the straightforward argument that the smoothings are as claimed. 
 \end{proof} 
 
 \begin{remark}\label{remark14}  (i)    The point in Wahl's construction  is that the deformation space of the rational singularity whose minimal resolution is a smooth rational curve of self-intersection $-4$ has two different components, one corresponding to rational ruled surfaces of degree $4$ in $\Pee^5$ and the other to the Veronese surface (the ``non-Artin component"). 

\smallskip
\noindent (ii) A slightly weaker version of Theorem~\ref{ssWahl} is already present in \cite[Theorem 3.1]{Friedman83}. Here, the result is stated under  the somewhat loosely worded hypothesis that ``the number of moduli of $\widetilde{X}$ keeping the curve $D$ is $-\chi(T_{\widetilde{X}}) - 3$," i.e.\ it is exactly $3$ conditions in moduli to keep the smooth rational curve $D$ with self-intersection $-4$. For the blown up $K3$ surface $\widetilde{X}$, there are $22$ moduli: $20$ for the $K3$ and $2$ for the point to blow up. To keep the curve $D$ on the blowup, it is one condition to keep the elliptic pencil defined by the nodal curve $\overline{D}$ (which will necessarily have nodal singular fibers near $\overline{D}$ and then two conditions because the point to blow up has to be the singular point for a nearby fiber.
Of course, a more convincing way to express the above condition about the number of moduli is to note that, since $H^2(\widetilde{X}; T_{\widetilde{X}}(-\log D)) = 0$,  the map
$H^1(\widetilde{X}, T_{\widetilde{X}}) \to H^1(D; N_{D/\widetilde{X}})$
is surjective. As $\dim   H^1(D; N_{D/\widetilde{X}}) = \dim H^1(\Pee^1; \scrO_{\Pee^1}(-4)) = 3$, it follows  that $\dim H^1(\widetilde{X}, T_{\widetilde{X}}(-\log D)) = \dim H^1(\widetilde{X}, T_{\widetilde{X}})  - 3$.

\smallskip
\noindent (iii) Following the method of \cite{Friedman83}, one can explicitly construct degenerations of an elliptic surface with $\kappa =1$ to the normal crossing surface $Z$ by degenerating the branch locus.

\smallskip
\noindent (iv) The degeneration of  a $\kappa =1$  elliptic surface with a multiplicity $2$ fiber to the union of a blown up $K3$ surface and  a $\Pee^2$ glued along a conic has  appeared previously in the literature, cf.\ \cite{Usui}. 
\end{remark}

\subsection{Blown up Enriques surfaces} In this subsection, we assume that the minimal model $Y_0$ of $\hY$ is a (smooth) Enriques surface. By Theorem~\ref{theorem5}(iii), there exists a morphism $\rho\colon \hY\to Y_0$, where    $\rho$ is the blowup of $Y_0$ at a single point $p\in Y_0$, and the exceptional curve $C$ on $\hY$ satisfies: $C\cdot D =2$. In particular, $K_{\hY} \equiv \scrO_{\hY}(C)$. If $D_1$ and $D_2$ are the two connected components of $D$, then $C\cdot D_1 = C \cdot D_2 =1$. Thus, if $\overline{D}_1$ and $\overline{D}_2$ are images of $D_1$ and $D_2$ in $Y_0$, then  $\overline{D_1}$ and $\overline{D_2}$ are isomorphic to $D_1$ and $D_2$ respectively,  $\overline{D}_1$ and $\overline{D}_2$ meet transversally at $1$ point, and each $D_i$ is either a smooth elliptic curve, an irreducible nodal curve, or a cycle of smooth rational curves. It follows that $\overline{D}_1$ and $\overline{D}_2$ are the reductions of multiple fibers in two different elliptic fibrations.

As in the $K3$ case, we begin with the case where $D_1$ and $D_2$ are irreducible, in fact smooth in this case:

 \begin{theorem}\label{Enriquesirred} If  $D_1$ and $D_2$ are smooth, then $H^2(\hY; T_{\hY}(-D)) = 0$.  Hence the homomorphism $\mathbb{T}^1_Y \to H^0(Y; T^1_Y)$ is surjective. In particular, the two elliptic singularities on $Y$ can be deformed independently.
 \end{theorem}
 \begin{proof}
With $\rho\colon \hY \to Y_0$ and $\overline{D}_i$ as above, let $\overline{D} = \overline{D}_1+ \overline{D}_2$, where the $\overline{D}_i$ are smooth elliptic  curves meeting transversally. The proof of Theorem~\ref{K3irred} shows that it suffices to prove that 
$$H^0(Y_0; \Omega^1_{Y_0}\otimes K_{Y_0}\otimes  \scrO_{Y_0}(\overline{D})) = 0,$$
where  $K_{Y_0}$ is  a  $2$-torsion line bundle of order $2$. 

Let $\tau\colon Z\to Y_0$ be the $K3$ \'etale double cover of $Y_0$ corresponding to $K_{Y_0}$ and let $E_i$ be the preimage of $\overline{D}_i$ in $Z$. Thus each $E_i$ is also a  smooth elliptic curve and $E_1 \cdot E_2 =2$. Since $Z \to Y_0$ is \'etale and $\tau^*K_{Y_0}=K_Z =\scrO_Z$, the pullback of $\Omega^1_{Y_0}\otimes K_{Y_0}\otimes  \scrO_{Y_0}(\overline{D})$ is $\Omega^1_Z\otimes \scrO_ Z(E)$, where $E = E_1 + E_2$. It will suffice to prove that $H^0(Z; \Omega^1_Z\otimes \scrO_ Z(E))=0$. 

By standard results, the linear system $E$ defines a degree $2$ morphism $\nu\colon Z \to Q$, where $Q$ is a smooth quadric in $\Pee^3$ and the $E_i = \nu^{-1}(\ell_i)$, where $\ell_1$ and $\ell_2$ are the two rulings on $Q$. In other words, $\scrO_Z(E) = \nu^*\scrO_Q(1,1)$. An argument as in the proof of Theorem~\ref{K3irred} shows that $\nu$ is finite. The branch divisor $\Sigma$ of the double cover $\nu$ is a smooth $(4,4)$-divisor on $Q$, and thus $L =\scrO_Q(2,2)$. Applying Lemma~\ref{doubcov}, it  suffices to prove that 
$$H^0(Q; \Omega^1_Q  \otimes \scrO_Q(1,1)) = H^0(\Pee^2; \Omega^1_Q(\log \Sigma) \otimes \scrO_Q(-1, -1)) = 0.$$
Since $Q\cong \Pee^1\times \Pee^1$, $\Omega^1_Q \cong \scrO_Q(-2, 0) \oplus \scrO_Q(0, -2)$ and thus 
$$\Omega^1_Q  \otimes \scrO_Q(1,1) \cong \scrO_Q(-1, 1) \oplus \scrO_Q(1, -1).$$
Hence $H^0(Q; \Omega^1_Q  \otimes \scrO_Q(1,1)) = 0$. By a Poincar\'e residue argument,   since $\scrO_Q(-1, -1)|\Sigma$ is a line bundle on $\Sigma$ of negative degree, 
$$H^0(\Pee^2; \Omega^1_Q(\log \Sigma) \otimes \scrO_Q(-1, -1))=0.$$  Putting this all together,  $H^0(Z; \Omega^1_Z\otimes \scrO_ Z(E))=0$ and thus $H^2(\hY; T_{\hY}(-D) ) =0$.
\end{proof}

Next we deal with the case where both $D_1$ and $D_2$ correspond to cusp singularities:

 \begin{theorem}\label{Enriquesred} If  $D_1$ and $D_2$ are irreducible nodal curves or cycles of smooth rational curves, then $H^2(\hY; T_{\hY}(-\log D)) = 0$.  Hence the homomorphism $\mathbb{T}^1_Y \to H^0(Y; T^1_Y)$ is surjective and the two elliptic singularities on $Y$ can be deformed independently.
 \end{theorem}
 \begin{proof} Note that the divisor $\overline{D}$ is a normal crossing divisor on $Y_0$. As in the proof of Theorem~\ref{Enriquesirred}, it suffices to prove that $H^0(Y_0; \Omega^1_{Y_0}(\log \overline{D})\otimes K_{Y_0}) = 0$, and then that $H^0(Z; \Omega^1_Z(\log E)) =0$, where $E$ is the inverse image of $\overline{D}$ on the $K3$ cover $Z$. More precisely, since
 $$\tau_*\Omega^1_Z(\log E) = \tau_*\tau^*\Omega^1_{Y_0}(\log \overline{D}) = \Omega^1_{Y_0}(\log \overline{D}) \oplus (\Omega^1_{Y_0}(\log \overline{D})\otimes K_{Y_0}),$$
 the group $H^0(Y_0; \Omega^1_{Y_0}(\log \overline{D})\otimes K_{Y_0})$ is the space of anti-invariant sections of $H^0(Z; \Omega^1_Z(\log E))$. Since $H^0(Z; \Omega^1_Z)=0$, it is enough to show that the images of the anti-invariant fundamental classes are linearly independent in $H^1(Z; \Omega^1_Z) \subseteq H^2(Z; \Cee)$. Here ``anti-invariant fundamental class" means the classes of the form $\Gamma^+ -\Gamma^-$, where $\Gamma$ is a component of $E$ and $\iota(\Gamma^+) = \Gamma^-$, where $\iota \colon Z \to Z$ is the involution corresponding to the double cover.
   In particular, if $D_1 =\sum_{i=0}^n\Gamma_i^{(1)}$ and $D_2 =\sum_{i=0}^m\Gamma_i^{(2)}$, then the preimage in $Z$ of each $\Gamma_i^{(j)}$ is a union $\Gamma_i^{(j)}{}^+ + \Gamma_i^{(j)}{}^-$ of two disjoint smooth rational curves exchanged under $\iota$.  Then a calculation shows that, assuming for simplicity that $n,m\geq 2$,  the anti-invariant classes in the lattice  spanned by the components of $E$ are a lattice with basis $e_0, \dots, e_n, f_0, \dots, f_m$ and satisfying $e_i^2 = f_i^2=-4$, $e_i\cdot e_j = 0$ for $i\neq j\pm 1 \pmod{n}$, and $e_i\cdot e_{i+1} = e_n\cdot e_0 =2$, and similarly  for the $f_i$, $e_0\cdot f_0 = 2$, and $e_i\cdot f_j =0$ if $(i,j) \neq (0,0)$.  Thus, in the notation of Definition~\ref{somelattices}, they span the lattice $\Lambda_2(n,m)(2)$. Hence, by Lemma~\ref{somelatticeslemma}, the images of the anti-invariant fundamental classes are linearly independent in $H^1(Z; \Omega^1_Z)$ as claimed. 
   
   The second  statement is then a consequence of Proposition~\ref{prop1121}(ii). 
 \end{proof} 
 
A similar but simpler argument, using  Proposition~\ref{prop1122}, shows the following:

 \begin{theorem}\label{Enriquesmixed} If  $D_1$ is a smooth elliptic curve and $D_2$ is either an irreducible nodal curve  or a cycle  of smooth rational curves, then $H^2(\hY; T_{\hY}(-\log (D_1+ D_2))) = 0$.  Hence there is a deformation of $Y$  to an I-surface with   two   simple elliptic singularities. \qed
 \end{theorem}

\section{The case $\kappa(\hY) =-\infty$: the rational case}\label{section6}

Throughout  this section, we assume that $\kappa(\hY) = -\infty$ and $k=2$, or equivalently that  $\hY$ is a rational surface.

\subsection{The case $(m_1, m_2) = (2,2)$}\label{ssection51} We begin what is essentially  a converse statement to  Theorem~\ref{geomanal}(i): 

\begin{proposition}\label{genrat2}  Let $Y_0$ be a rational surface with a negative definite  anticanonical divisor $D_2$ which is either smooth or nodal. Let $\Gamma\subseteq Y_0$ be a reduced  curve with  $p_a(\Gamma) = 2$ such that $\Gamma \cap D_2 =\emptyset$,  with either a node or a cusp at   $p_0\in \Gamma$, and which is either irreducible or a cycle of curves with all singularities  nodal with the possible exception of  $p_0$. Let $\hY$ be the blow up of $Y_0$ at $p_0$, let $D_1$ be the proper transform of $\Gamma$,  let the new exceptional curve be denoted by $C$, so that $D_1 = \Gamma-2C$, and identify $D_2$ with its preimage in $\hY$.  
\begin{enumerate}
\item[\rm(i)] The divisor $L= K_{\hY} + D_1 + D_2$ is a nef and big divisor and satisfies $L^2=1$ and $L\cdot D_i =0$. 
\item[\rm(ii)] If $L$ has positive intersection with every exceptional   curve which is not a component of $D_1$ or $D_2$,  then $m_2 = -D_2^2 = 2$ if  $M_2=L+D_1$ is not nef, and $m_2 = 1$ if $M_2$ is nef. 
\item[\rm(iii)] If $L$ has positive intersection with every   curve which is not a component of $D_1$ or $D_2$,  then the contraction $Y$ of $\hY$ at $D_1$ and $D_2$ is an I-surface with  two elliptic singularities and no rational double point singularities. 
\item[\rm(iv)] Conversely, if $\hY$ is the minimal resolution of an I-surface $Y$ with two simple elliptic or cusp singularities, one of which has multiplicity $2$, and no other singularities, then $\hY$ arises as above.
\end{enumerate}
\end{proposition}
\begin{proof} If $\Gamma$ is irreducible, then $\Gamma^2 = \Gamma^2 + \Gamma \cdot K_{Y_0} = 2$, so that $\Gamma^2 = 2$ and $D_1 ^2 = -2$. A similar argument shows that, if $\Gamma$ is not irreducible but $p_0$ lies on just one component, then the component of $\Gamma$ containing $p_0$ has arithmetic genus one and hence self-intersection $0$, and all other components are smooth rational with self-intersection $-2$. In this case,  the unique component of $D_1$ meeting $C$ has self-intersection $-4$ and the remaining components have self-intersection $-2$. Similarly, if $p_0$ lies on two different components, then each of them is smooth rational and their proper transforms on $\hY$ have self-intersection $-3$. In particular, $D_1^2 =-2$ in all cases and $D_1$ is a negative definite curve corresponding to a simple elliptic or cusp singularity of multiplicity $2$. 

Since $K_{Y_0} = -D_2$, $L=K_{\hY} + D_1 + D_2 = C+ D_1$. Thus $L^2 =1$, $L|D_i \cong \scrO_{D_i}$ and $L\cdot C = 1$. Thus $L$ is nef and big. If $L$ has positive intersection with every   exceptional curve, the proof of  Theorem~\ref{geomanal} shows that $m_2 =2$ if $M_2$ is not nef, and $m_2 = 1$ if $M_2$ is nef. This proves (i) and (ii). As for (iii), since $L =\pi^*\omega_Y$, it follows that $\omega_Y$ is ample, $\omega_Y^2 =1$, and $\chi(\scrO_Y) = 3$ by Lemma~\ref{lemmam22}. Thus $Y$ is an I-surface. (Compare also Remark~\ref{blowdown}.) Finally, (iv) is  a consequence of Theorem~\ref{geomanal}(i). 
\end{proof}

For the rest of this subsection, we assume that $\hY$ is a rational surface and that $(m_1, m_2) = (2,2)$. By Proposition~\ref{genrat2}, there exists a one point blowdown of $\hY$ to a rational surface $Y_0$ with an effective anticanonical divisor of square $-2$. The following is an explicit example:  

\begin{example} Let $Q$ be a plane quartic in $\Pee^2$ meeting a smooth cubic $\overline{D}_2$  in $11$ points, one of them double, let $Y_0$ be the blow up $\Pee^2$ at the $11$ points, so the proper transform $\Gamma$ of $Q$ is a smooth curve of genus $2$. The proper transform $D_2$ of $\overline{D}_2$ then satisfies $D_2^2 = -2$. Further degenerate $Q$ so that it has a second double point not lying on $\overline{D}_2$. Then let $\hY$ be the blow up of $Y_0$ at the second double point, let $D_1$ be the proper transform of $\Gamma$ and let the new exceptional curve be $C$, so that $D_1 = \Gamma-2C$. Then $\Gamma^2 = 2$, $D_1^2  = \Gamma^2  -4 = -2$, and $C \cdot D_1 =2$. Note the symmetry between the curves $D_1$ and $D_2$ in this example:  let $\ell$ be the line in $\Pee^2$ joining the two double points of $Q$. The proper transform of $\ell$ on  $\hY$ is an exceptional curve meeting $D_1$ in two points and disjoint from $D_1$, and so can also be blown down to get the symmetric picture reversing the roles of $D_1$ and $D_2$.
\end{example} 

There is an alternate description of the surface $\hY$ in case $D_1$ and $D_2$ are irreducible:

\begin{proposition}\label{22case}  Suppose that $\hY$  is the minimal resolution of an I-surface  containing two irreducible divisors $D_1$ and $D_2$ with $D_i^2=-2$. Then there exists an exceptional curve $C_1$ on $\hY$ such that $C_1\cdot D_1 = 2$. If $\rho_1\colon \hY \to Y_0^{(1)}$ is the contraction of $C_1$, then  there exists a degree two morphism  $\nu \colon Y_0^{(1)}  \to \mathbb{F}_1$ branched along a smooth divisor in $|2\sigma_0 + 6f|$, where $\sigma_0$ is the negative section,  such that  $\rho_1(D_1) = \overline{D}_1 =\nu^*(\sigma_0+f)$  and $D_2=\nu^*\sigma_0$. 
\end{proposition}
\begin{proof} By  Theorem~\ref{geomanal}, without assuming that $D_1$ and $D_2$ are irreducible, there exist exceptional curves $C_1$ and $C_2$ on $\hY$ with $C_1\cdot D_1 = 2$, $C_1\cdot D_2 =0$, $C_2\cdot D_1=0$, $C_2\cdot D_2 =2$, and 
$$K_{\hY} = C_1 -D_2 = C_2 -D_1.$$
Thus $C_1\cdot C_2 = 1$ because
$$-1 = K_{\hY} \cdot C_1 = (C_2 - D_1)\cdot C_1 = (C_1\cdot C_2) - 2.$$
Moreover, $L\cdot C_i = 1$ and 
$$L =  K_{\hY} + D_1 + D_2 = C_1+D_1 = C_2 + D_2.$$

Let $\rho_1\colon \hY \to Y_0^{(1)}$ be the contraction of $C_1$. Then $\overline{D}_1= \rho_1(D_1)$ is a nef divisor  with $(\overline{D}_1)^2 =2$, $p_a(\overline{D}_1) = 2$, and $K_{Y_0^{(1)}} =\scrO_{Y_0^{(1)}}(-D_2)$. By symmetry, we can make the same construction by taking the contraction $\rho_2\colon \hY \to Y_0^{(2)}$ of  $C_2$, with $\overline{D}_2= \rho_2(D_2)$. The divisor $\overline{D}_i$ is nodal $\iff$ $D_i$ meets $C_i$ in two distinct points.  

The  divisor  $\overline{D}_1$ is an effective and  nef divisor with positive self-intersection disjoint from $D_2$.  It follows from standard results (e.g.\ \cite[Theorem 4.12]{Friedanticanon}) that the linear system $|\overline{D}_1|$ has no fixed components or base locus and defines a degree $2$ morphism $\varphi\colon  Y_0^{(1)} \to \Pee^2$ which contracts $D_2$.  Note that $\overline{D}_1$ is induced by the divisor $D_1+ 2C_1$ on $\hY$, i.e.\ 
$$\rho_1^*\overline{D}_1 = D_1+ 2C_1= L+C_1.$$ 

\begin{claim} The only irreducible curves that $\varphi$ contracts to a point are components of $D_2$ or components of $\overline{D}_1$.
\end{claim}
\begin{proof} If $A$ is an irreducible curve on $Y_0^{(1)}$, not a component of $\overline{D}_1$,  such that $\overline{D}_1\cdot A =0$, and $A'$ is its proper transform on $\hY$, then $A'\cdot D_1 = A'\cdot C_1 = 0$. Hence $L\cdot A' =0$, so $A'$ is a component  of $D_2$. (Recall that, as always, Assumption~\ref{assumption27} is in force.) 
\end{proof}

Returning to the proof of Proposition~\ref{22case}, suppose that $D_1$ and $D_2$ are irreducible. Let $p\in \Pee^2$ be the image $\varphi(D_2)$ and let $\mathbb{F}_1\to \Pee^2$ be the blowup of $\Pee^2$ at $p$. Then we claim that $\varphi\colon Y_0^{(1)} \to \Pee^2$ extends to a finite degree two morphism $\nu\colon Y_0^{(1)} \to \mathbb{F}_1$  branched along a smooth divisor in $|2\sigma_0 + 6f|$, where $\sigma_0$ is the negative section.  
 To see this, note that  
 $$\rho_1^*(2\overline{D}_1 - D_2) = 2L+2C_1- D_2= (C_2+D_2) + L + 2C_1 -D_2 = L + 2C_1+C_2.$$
 Thus the linear system $|2\overline{D}_1 - D_2|$ corresponds to the linear system  $|L + 2C_1+C_2|$ on $\hY$. The divisor  $L + 2C_1+C_2$ is nef and big, since:
\begin{align*}
(L + 2C_1+C_2) \cdot C_1=0;   &\qquad (L + 2C_1+C_2) \cdot C_2=1;\\
(L + 2C_1+C_2) \cdot D_1=4 ; &\qquad (L + 2C_1+C_2) \cdot D_2=2;\\
(L + 2C_1+C_2)^2 &= 6.
\end{align*}
  In particular, if $D_1$ and $D_2$ are irreducible, then 
   $C_1$ is the only irreducible curve on $\hY$ which has intersection number $0$ with $L + 2C_1+C_2$.  Since $2\overline{D}_1 - D_2 = 2\overline{D}_1 + K_{Y_0^{(1)}}$ and $2\overline{D}_1$ is nef and big, Kodaira-Mumford vanishing implies that 
   $H^1( Y_0^{(1)}; \scrO_{Y_0^{(1)}}(2\overline{D}_1 - D_2) ) = 0$. 
    Riemann-Roch then implies that $\dim |2\overline{D}_1 - D_2| =4$. Hence, if $p=\varphi(D_2)\in \Pee^2$ is the image of $D_2$, the natural inclusion $|\scrO_{\Pee^2}(2) -p|\subseteq|2\overline{D}_1 - D_2|$ is an equality.
  By  \cite[Theorem 4.12]{Friedanticanon}, $|2\overline{D}_1 - D_2|$ is a base point free linear system. Thus the induced morphism $\nu\colon Y_0^{(1)} \to \Pee^4$ has degree $2$ onto its image in $\Pee^4$, which is $\mathbb{F}_1$.  By construction, $2\overline{D}_1 - D_2 =\nu^*(\sigma_0+2f)$ and the image of $D_2$ is the negative section $\sigma_0$. Then  $D_2=\nu^*\sigma_0$ and $\overline{D}_1 = \nu^*(\sigma_0+f)$. Moreover, $K_{Y_0^{(1)}}  =\scrO_{Y_0^{(1)}}(-D_2)$, so that, if $\Sigma$ is the branch locus, then 
$$  -\sigma_0 = K_{\mathbb{F}_1} +  \frac12\Sigma = -2\sigma_0 -3f + \frac12\Sigma.$$
Thus $\Sigma \in |2\sigma_0 + 6f|$.     This completes the proof of Proposition~\ref{22case}. 
\end{proof}

\begin{proposition}\label{22irred} Suppose that $D_1$ and $D_2$ are irreducible. Then $$H^2(\hY; T_{\hY}(-D_1-D_2)) = 0.$$
Hence the map $\mathbb{T}^1_Y \to H^0(Y; T^1_Y)$ is surjective.
\end{proposition}
\begin{proof} By Serre duality, we need to show that 
$$H^0(\hY; \Omega^1_{\hY}\otimes K_{\hY}\otimes \scrO_{\hY}(D_1+ D_2)) = 0.$$
By the above, $K_{\hY}\otimes \scrO_{\hY}(D_1+ D_2) = L = C_1+ D_1$.  By the usual Hartogs argument,
$$H^0(\hY; \Omega^1_{\hY}(C_1+ D_1)) \subseteq H^0(Y_0^{(1)}; \Omega^1_{Y_0^{(1)}}(\overline{D}_1)).$$
Proposition~\ref{22case} gives a finite degree two morphism $\nu\colon Y_0^{(1)} \to \mathbb{F}_1$, with $\overline{D}_1 = \nu^*(\sigma_0+f)$. Thus, by   Lemma~\ref{doubcov}, there is an exact sequence
$$0 \to \Omega^1_{\mathbb{F}_1}(\sigma_0+f) \to \nu_*\Omega^1_{Y_0^{(1)}}(\overline{D}_1)\to \Omega^1_{\mathbb{F}_1}(\log \Sigma)(\sigma_0+f)\otimes \lambda^{-1}\to 0,$$
with $\lambda =\scrO_{\mathbb{F}_1}( \sigma_0 +3f)$.  So we must show that 
$$H^0(\mathbb{F}_1; \Omega^1_{\mathbb{F}_1}(\sigma_0+f)) = H^0(\mathbb{F}_1; \Omega^1_{\mathbb{F}_1}(\log \Sigma) (-2f)) =0.$$
The second vanishing follows from the inclusion
$$  H^0(\mathbb{F}_1; \Omega^1_{\mathbb{F}_1}(\log \Sigma) (-2f)) \subseteq  H^0(\mathbb{F}_1; \Omega^1_{\mathbb{F}_1}(\log \Sigma)),$$  
the vanishing of $H^0(\mathbb{F}_1; \Omega^1_{\mathbb{F}_1}) =0$, and the Poincar\'e residue sequence. As for the first group, the relative cotangent sequence for the morphism $g\colon  \mathbb{F}_1 \to \Pee^1$ is
$$0 \to g^*\Omega^1_{\Pee^1} \to \Omega^1_{\mathbb{F}_1}\to  \Omega^1_{\mathbb{F}_1/\Pee^1}\to 0.$$
Then  $g^*\Omega^1_{\Pee^1} =\scrO_{\mathbb{F}_1}(-2f)$ and $\Omega^1_{\mathbb{F}_1/\Pee^1} = \scrO_{\mathbb{F}_1}(-2\sigma_0- f)$. Thus 
$$H^0(\mathbb{F}_1; g^*\Omega^1_{\Pee^1}\otimes \scrO_{\mathbb{F}_1}(\sigma_0+f)) = H^0(\mathbb{F}_1;  \Omega^1_{\mathbb{F}_1/\Pee^1}\otimes \scrO_{\mathbb{F}_1}(\sigma_0+f)) = 0.$$
So finally $H^2(\hY; T_{\hY}(-D_1-D_2)) = 0$ as claimed.
\end{proof}

For the case of at least one cusp singularity, we have the following, with a mild general position hypothesis: 

\begin{proposition} Suppose that   either $\overline{D}_1$ or  $\overline{D}_2$  is nodal. Then $$H^2(\hY; T_{\hY}(-\log (D_1+D_2))) = 0.$$
\end{proposition}
\begin{proof} It suffices to consider the case where $\overline{D}_1$  is nodal. In this case, 
\begin{align*}
H^0(\hY; \Omega^1_{\hY}(\log (D_1+D_2))\otimes K_{\hY}) &\subseteq H^0(Y_0^{(1)};  \Omega^1_{Y_0^{(1)}}(\log (\overline{D}_1+D_2))\otimes K_{Y_0^{(1)}})\\
&\subseteq H^0(Y_0^{(1)};  \Omega^1_{Y_0^{(1)}}(\log (\overline{D}_1+D_2))).
\end{align*}
Then this last group is $0$ by the Poincar\'e  residue sequence and the  fact that, by Lemma~\ref{somelatticeslemma},  the components of $\overline{D}_1$ and $D_2$ are linearly independent in cohomology. 
\end{proof}

\begin{corollary}\label{22cusp} Suppose that   either $\overline{D}_1$ or  $\overline{D}_2$  is nodal.
\begin{enumerate}
\item[\rm(i)] If both singularities are cusp singularities, then the map $\mathbb{T}^1_Y \to H^0(Y; T^1_Y)$ is surjective. 
\item[\rm(ii)]  If one singularity is a cusp singularity and the other is simple elliptic, then there is a deformation of $Y$ to an I-surface with two simple elliptic singularities of multiplicity $2$. 
\end{enumerate}
\end{corollary}
 \begin{proof} The first statement follows from Proposition~\ref{prop1121}(ii) and the second statement from  Proposition~\ref{prop1122}.
 \end{proof}

\begin{remark}\label{exotic22}  Smoothing one of the elliptic singularities leads to an I-surface whose minimal resolution is a blown up $K3$ surface. In case one of the singularities is a cusp singularity, an exotic deformation to an elliptic singularity of multiplicity one leads to the case $(m_1, m_2) = (2,1)$, where the minimal resolution is necessarily rational. If both singularities are cusp singularities, taking exotic deformations of both  elliptic singularities to two elliptic singularities of multiplicity one leads to the case $(m_1, m_2) = (1,1)$. \emph{A priori}, the minimal resolution could be a blown up Enriques surface or a rational surface. However, we shall see below that minimal resolution must be  rational. 
\end{remark}

\subsection{The case $(m_1, m_2) = (2,1)$}

As in \S\ref{ssection51}, we begin with an explicit example:

\begin{example}\label{ex21}  Let $X$ be the blowup of $9$ points on a smooth plane cubic and let $F$ be the proper transform of the cubic. We assume that the points are chosen so that the normal bundle $N_{F/X}$ is a nontrivial $2$-torsion point on $F$, but  are general otherwise. Then it is a standard result that $|F|$ is just one point, i.e.\ $h^0(\scrO_X(F)) = 1$, but $|2F|$ is a pencil and  defines an elliptic surface $X \to \Pee^1$ where $F$ is the unique multiple fiber, of multiplicity $2$. If $E$ is an exceptional curve on $X$, then $E \cdot 1 = 1$, hence $E$ is a bisection of the fibration. Here $K_X =\scrO_X(-F)$.

Consider the linear system $|E+ F|$. Note that $E\cdot (E+ F) = 0$ and $F\cdot (E+ F) = 1$. Hence $(E+ F)^2 = 1$, so that $E+ F$ is nef and big,  and $(E+ F)\cdot K_X = -1$. By Riemann-Roch,
$$\chi(\scrO_X(E+ F)) = h^0(\scrO_X(E+ F)) = \frac12(1+1) + 1 = 2.$$
Thus $|E+ F|$ is a pencil. Since $\scrO_F(F)$ is a nontrivial line bundle of degree $0$ on $F$, $h^1(\scrO_F(F)) =0$ and then $h^1(\scrO_X(F)) =0$. It follows that the natural map $H^0(X; \scrO_X(E+ F)) \to H^0(E ; \scrO_E)$ is surjective, as is the map $H^0(X; \scrO_X(E+ F)) \to H^0(F;  \scrO_F(q))$, where $q$ is the point  of $F$ such that $\scrO_F(E+ F) = \scrO_F(q)$. In particular $|E+ F|$ has a simple base point at  $q$, so the general element $\Gamma$ is a smooth bisection of the elliptic fibration. From
$$0 = 2p_a(\Gamma)-2 = \Gamma^2 + K_X\cdot\Gamma,$$
$\Gamma$ is a smooth elliptic curve. Now let $\hY$ be the blow up of the two points of intersection of $\Gamma$ with a general fiber $G$ of the elliptic fibration defined by $|2F|$ and let $C_1, C_2$ be the new exceptional curves. Identifying a divisor on $X$ with its total transform on $\hY$, let $D_1$ be the proper transform of $G$, so that $D_1 = G-C_1-C_2$, and let $D_2$ be the proper transform of $\Gamma$, so that $D_2 = E + F - C_1 - C_2$. Thus $D_1^2=-2$, $D_2^2 = -1$ and $D_1\cap D_2 = \emptyset$. We have
\begin{align*}
M_2  = K_{\hY} + D_1  &=  -F + C_1 + C_2 +  G-C_1-C_2 = F;\\
M_1   = K_{\hY} + D_2 &= -F + C_1 + C_2 +  E + F - C_1 - C_2 = E;\\
L  = K_{\hY} + D_1 + D_2   &=  -F + C_1 + C_2 + f-C_1-C_2 + E + F - C_1 - C_2 =E + f-C_1 - C_2.
\end{align*}
Thus $K_{\hY} + D_1$ is nef and $(K_{\hY} + D_1)^2 =0$, $K_{\hY} + D_2$ is the exceptional curve $E$ with $E \cdot D_1 =  E \cdot f = 2$ and $E\cdot D_2 =0$, and $L= K_{\hY} + D_1 + D_2$ satisfies 
$$L^2 = (E + f)^2 - 2 = -1+4 -2 = 1.$$
Furthermore $L = E+ D_1 = F+D_2$ and $L\cdot E = 1$, $L\cdot F = 1$. Thus $L$ is nef and big, $L\cdot D_1 = (E+D_1)\cdot D_1 =0$, $L\cdot D_2 = (F+D_1)\cdot D_1 =0$, and   $L\cdot A> 0$ for all other curves $A$ as long as all fibers of $X\to \Pee^1$ are irreducible. 

If one of the blowup points  were on $E$, then the proper transform $E'$ of $E$ on $\hY$ would be  a curve such that $L\cdot E' =0$.  By our assumptions, this does not happen, so that we can identify $E$ with its image in $\hY$. Then blowing down $E$ produces a new rational surface $X'$ with $K_{X'} =\scrO_{X'}(-D_2)$. In particular, this gives an explicit  realization of $\hY$ corresponding to the two  possibilities in Theorem~\ref{geomanal}.
\end{example}

Next we show that all possibilities essentially arise as in Example~\ref{ex21}:

\begin{proposition}\label{21case}  Suppose that $\hY$ is the minimal resolution of an I-surface with two elliptic singular points $p_1$ and $p_2$ of multiplicities $m_1=2$, $m_2=1$. Then there exists a rational elliptic surface $X$ with a multiple fiber $F$ of multiplicity $2$, a smooth or nodal nonmultiple fiber $G$, and a bisection $\Gamma$ which is either smooth, irreducible nodal,  or a cycle of smooth rational curves, such $D_1$ and $D_2$ are the proper transforms of $G$ and $\Gamma$ respectively and one of the following holds:
\begin{enumerate}
\item[\rm(i)] $\Gamma$ meets $G$ transversally, necessarily at two smooth points $p_1$ and $p_2$ and   $\hY$ is the blowup of $X$ at $p_1$ and $p_2$;
\item[\rm(ii)] $G$ is a nodal fiber, $\Gamma$ meets $G$  at a node $p$, and $\hY$ is the blowup of $X$ at $p$ followed by an infinitely near blowup at the intersection of the exceptional curve with the proper transform of $\Gamma$; 
\item[\rm(iii)] $\Gamma$ is   irreducible and nodal, $G$ meets $\Gamma$  at the node $p$, and $\hY$ is the blowup of $X$ at $p$ followed by an infinitely near blowup at the intersection of the exceptional curve with the proper transform of $G$. 
\end{enumerate}
In both {\rm(ii)} and {\rm(iii)}, either $D_1$ or $D_2$ is reducible. 
\end{proposition} 
\begin{proof}  The linear system $|2M_2|$ defines an elliptic fibration on $\hY$ such that $D_2$ is a bisection. Let $\rho\colon \hY \to X$ be the relatively minimal model.  Then $X$ is a rational elliptic surface with a multiple fiber $F$ and $M_2=F$. Since $K_X^2=0$ and $K_{\hY}^2 = -2$, $\hY$ is the blowup of $X$ at two points $p_1$ and $p_2$, possibly infinitely near. We shall just consider the case where $p_1$ and $p_2$ are distinct points, corresponding to Case (i).  Let $C_1$ and $C_2$ be the corresponding exceptional curves. Then $K_{\hY} = -F + C_1 + C_2$ and
$$1 = L\cdot K_{\hY} = -(L\cdot M_2) + L\cdot C_1 + L \cdot C_2.$$
Since $L\cdot C_i > 0$, $L\cdot C_1 = L \cdot C_2 = 1$. Then
$$1 = (K_{\hY} +D_1 + D_2)\cdot C_i = -1 + (D_1\cdot C_i) + (D_2\cdot C_i).$$
Since $M_2\cdot C_i=0$ and $M_2 = K_{\hY} + D_1$,  $D_1\cdot C_i =  D_2\cdot C_i =1$. In particular, $D_i$ and $C_i$ intersect transversally at exactly one smooth point of $D_i$. Let $G =\rho(D_1)$ and $\Gamma = \rho(D_2)$. Then $G$ is isomorphic to $D_1$ and $\Gamma$ is isomorphic to $D_2$, so they are either smooth, irreducible nodal,  or a cycle of smooth rational curves, and $G$ and $\Gamma$ meet transversally at two points.  Moreover, $\Gamma$ is a bisection of the fibration on $X$. Since $D_1 \cdot M_2 =0$, $G$ is contained in a fiber of $X$. Since $G\cdot \Gamma =2$, $G\neq F$ and so $G$ is a nonmultiple fiber. Thus $X$ and $\hY$ satisfy the conclusions of the proposition.  
\end{proof}

Turning to the deformation theory, we begin as usual with the case where $D_1$ and $D_2$ are irreducible. 

\begin{proposition}\label{21irred} Suppose that $D_1$ and $D_2$ are irreducible and that, in the notation of Proposition~\ref{21case}, all singular nonmultiple fibers of the elliptic fibration $\pi\colon X\to \Pee^1$ are irreducible nodal curves. Then $$H^2(\hY; T_{\hY}(-D_1 - D_2)) = 0.$$ 
Hence the map $\mathbb{T}^1_Y \to H^0(Y; T^1_Y)$ is surjective.
\end{proposition} 
\begin{proof} We have to show the vanishing of 
$$H^0(\hY; \Omega^1_{\hY}(D_1 + D_2+K_{\hY})) = H^0(\hY; \Omega^1_{\hY}(G + \Gamma -2C_1 -2C_2 -F + C_1 + C_2)).$$
If $f\equiv G$ is a general fiber of $\pi$, then $G + \Gamma -F = F+\Gamma$. By the usual Hartogs argument,  it suffices to show that $H^0(X; \Omega^1_X(F + \Gamma    ))=0$, where $\Gamma$ is an irreducible bisection of $\pi$ with $\Gamma^2 = 1$. From Theorem~\ref{tangbundell}, there is an exact sequence
$$0 \to \scrO_X(-2f+F) \to  \Omega^1_X \to I_Z \otimes \scrO_X(f) \to 0,$$
where $I_Z$ is a reduced scheme supported at the nodes of the singular fibers, and hence $\ell(Z) = 12$. Noting that $-2f+ F +F + \Gamma    = \Gamma-f$, 
 we must show that
$$H^0(X; \scrO_X(  \Gamma-f)) = H^0(X;  I_Z \otimes \scrO_X( f+F +\Gamma)) =0.$$
For the first term, suppose that $\Gamma -f$ is   an effective divisor.  Since all fibers of $\pi$ were assumed irreducible, $\Gamma -f$  would necessarily be numerically equivalent to $\Gamma' + rf$ for some nonnegative rational number $r$, where $\Gamma'$ is again an irreducible bisection. But then
$$-3 = (\Gamma -f)^2 = (\Gamma' + rf) \ge (\Gamma')^2.$$
But by adjunction 
$$(\Gamma')^2 + \Gamma'\cdot K_X = (\Gamma')^2 - \Gamma'\cdot F = (\Gamma')^2 - 1 \ge -2,$$
hence $ (\Gamma')^2 \ge -1$. This is a contradiction, so that $H^0(X; \scrO_X( \Gamma-f)) = 0$.

As for the second term, we use the double cover picture \S\ref{doubcovsect} of the blowup $\widetilde{X}$ of $X$ at two points of $F$ as in the proof of Theorem~\ref{thm17}.  In this case, $\widetilde{X}$ is realized as a double cover $\nu\colon \widetilde{X} \to \widetilde{\mathbb{F}}_1$. Identifying the points of $Z\subseteq X$ with the corresponding points of $\widetilde{X}$, a  section $s$ of  $I_Z \otimes \scrO_X( f+F +\Gamma)$ defines a section of  $\scrO_{\widetilde{X}}( f+F' +e_1+ e_2+\Gamma)$ vanishing along $Z$ in the notation of \S\ref{doubcovsect}. By the arguments of Theorem~\ref{thm17}, we have
\begin{align*}
\scrO_{\widetilde{X}}( f+F' +e_1+ e_2+\Gamma) &\subseteq \scrO_{\widetilde{X}}( f+F' +2e_1+ 2e_2+\Gamma);\\
  \scrO_{\widetilde{X}}( f+e +2e_1+ 2e_2+\Gamma) &= \nu^*\scrO_{\widetilde{\mathbb{F}}_1}(f+ e + d_1+ d_2 + \sigma_0);\\
\nu_*\nu^*\scrO_{\widetilde{\mathbb{F}}_1}(f+ e + d_1+ d_2 + \sigma_0) &=\scrO_{\widetilde{\mathbb{F}}_1}(f+ e + d_1+ d_2 + \sigma_0) \oplus \scrO_{\widetilde{\mathbb{F}}_1}(-\sigma_0 -f +2e).
\end{align*}
We have     $H^0(\widetilde{\mathbb{F}}_1; \scrO_{\widetilde{\mathbb{F}}_1}(-\sigma_0 -f +2e)) =0$ and $f+ e + d_1+ d_2 + \sigma_0 = \sigma_0 + 2f - e$. Thus it suffices to show that any section of $\scrO_{\widetilde{\mathbb{F}}_1}(\sigma_0 + 2f - e)$ vanishing at $12$ points $\nu(Z) \subseteq B_0$  must be identically $0$. 
Since  $B_0 \equiv 4\sigma_0 +3f+2d_2$, the divisor $\sigma_0 + 2f - e - B_0$ is not effective. Thus  a nonzero section of    $\scrO_{\widetilde{\mathbb{F}}_1}(\sigma_0 + 2f - e)$   must meet $B_0$ in a finite number number of points. If there were a nonzero section of $\scrO_{\widetilde{\mathbb{F}}_1}(\sigma_0 + 2f - e)$ vanishing at  $\nu(Z)$, then 
$$12 \le (\sigma_0 + 2f - e)\cdot B_0 = (\sigma_0 + 2f - e) \cdot (4\sigma_0 + 3f +2d_2) = 7.$$
This is a contradiction, hence $H^2(\hY; T_{\hY}(-D_1 - D_2)) = 0$ as claimed. 
\end{proof}

In the case where one of $D_1$, $D_2$ is reducible, we have the following:

\begin{proposition}\label{21red} Suppose that, in the situation of Proposition~\ref{21case},  $\hY$ is the blowup of $X$ at two distinct points. Then $H^2(\hY; T_{\hY}(-\log(D_1+ D_2)) = 0$. Hence, under this assumption, 
\begin{enumerate}
\item[\rm(i)] If both singularities are cusp singularities, then the map $\mathbb{T}^1_Y \to H^0(Y; T^1_Y)$ is surjective. 
\item[\rm(ii)]  If one singularity is a cusp singularity and the other is simple elliptic, then there is a deformation of $Y$ to an I-surface with two simple elliptic singularities, one  of multiplicity $2$ and the other of multiplicity one. 
\end{enumerate}
\end{proposition} 
\begin{proof} In this case, let $\rho\colon \hY \to X$ be the blowdown. By assumption, the image of $D_1+ D_2$ is the nodal divisor $G+ \Gamma$. By Serre duality, since $K_{\hY} = -F + C_1 + C_2$, it suffices to show that 
  $$H^0(\hY; \Omega^1_{\hY}(\log(D_1+ D_2))(-F + C_2 + C_2))=0.$$
   By the usual Hartogs argument, it suffices to show that $H^0(X; \Omega^1_X(\log(G+\Gamma))(-F)) = 0$, and hence that $H^0(X; \Omega^1_X(\log(G+\Gamma)) = 0$. Since $G$ is a (possibly reducible) fiber and $\Gamma$ is a (possibly reducible) bisection meeting $G$ transversally at two points, this is   straightforward.
\end{proof}

In the case of an exotic deformation of a cusp singularity, $Y$ can deform to an I-surface $Y_t$ with two singularities, both of multiplicity one. \emph{A priori}, the minimal resolution of $Y_t$ could be either an Enriques surface or a rational surface. The following says that the Enriques case does not arise: 

\begin{proposition}\label{21toR} Let $Y$ be an I-surface with exactly two elliptic singular points $x_1$ and $x_2$, such that $x_1$ has multiplicity $2$, $x_2$ has multiplicity $1$, and $x_1$ is a cusp singularity. If $Y_t$ is the I-surface obtained by taking a deformation of $x_1$ to an elliptic singularity of multiplicity $1$, then the minimal resolution $\hY_t$ is a rational surface.
\end{proposition} 
\begin{proof} We shall just outline the proof   in case the cusp singularity has self-intersection sequence $(-4, -2, \dots, -2)$. Let $E$ be the component of $D_1$ of square $-4$. It follows from Proposition~\ref{21case} that $\hY$ is the blowup of a rational elliptic surface $X$ and that $E$ is the proper transform of a fiber component $\overline{E}$. By Wahl's theory and the discussion of \S\ref{sssmooth}, especially Remark~\ref{remark14}(i), there is a one parameter   semistable family $\widetilde{\mathcal{Y}}\to \Delta$ whose general fiber is $\hY_t$ and whose fiber $\widetilde{\mathcal{Y}}_0$ over $0$ is the $d$-semistable normal crossing surface $\hY\amalg_E\Pee^2$, where $E\subseteq \hY$ is identified with a conic (also denoted $E$) in $\Pee^2$. 

If the general fiber $\hY_t$ were an Enriques surface, then by semicontinuity $\dim H^0(\widetilde{\mathcal{Y}}_0; \omega^{\otimes 2}_{\widetilde{\mathcal{Y}}_0}) \ge 1$. On the other hand, $\omega_{\widetilde{\mathcal{Y}}_0}|\hY = K_{\hY}\otimes \scrO_{\hY}(E)$ and $\omega_{\widetilde{\mathcal{Y}}_0}|\Pee^2 = K_{\Pee^2}\otimes \scrO_{\Pee^2}(E)=  \scrO_{\Pee^2}(-1)$.  In particular, $H^0(\Pee^2; \omega^{\otimes 2}_{\widetilde{\mathcal{Y}}_0}|\Pee^2) = 0$.  As for $H^0(\hY; \omega^{\otimes 2}_{\widetilde{\mathcal{Y}}_0}|\hY)$, first note that $K_{\hY} = \scrO_{\hY} (-F+ C_1 + C_2)$ in the above notation.  Thus
$$\omega^{\otimes 2}_{\widetilde{\mathcal{Y}}_0}|\hY = \scrO_{\hY}(-2F+ 2C_1 + 2C_2 + 2E).$$
If as above  $\overline{E}$ is the image of $E$ in the rational elliptic surface $X$,  $\overline{E}$ is a fiber component and hence has negative square. 
There exists a nef and big divisor $H$ on $X$ such that $H\cdot F > 0$ and $H\cdot \overline{E} = 0$. The pullback of $H$ to $\hY$ is then a nef and big divisor $H'$ such that
$$H'\cdot (-2F+ 2C_1 + 2C_2 + 2E) = -2H \cdot F < 0.$$
Thus $H^0(\hY; \scrO_{\hY}(-2F+ 2C_1 + 2C_2 + 2E)) = 0$. Hence $H^0(\widetilde{\mathcal{Y}}_0; \omega^{\otimes 2}_{\widetilde{\mathcal{Y}}_0}) = 0$, so that $\hY_t$ cannot be birational to an Enriques surface. Hence it is rational.
\end{proof} 

\subsection{The case $(m_1, m_2) = (1,1)$}

\begin{example}\label{ex11}  Suppose that $Y_0$ is a rational elliptic surface with a multiple fiber $F$ of multiplicity $2$ and $G$ is a general fiber. Thus $K_{Y_0} = -F$ and $G =2F$. Let $\Gamma$ be an irreducible curve such that $\Gamma^2 =3$, $\Gamma \cdot G =2$ (i.e.\ $\Gamma$ is a bisection of the elliptic fibration) and hence $\Gamma\cdot F =1$. Thus $p_a(\Gamma) = 2$.  We assume that $\Gamma$ has a node or cusp at a point $p$, necessarily not on $F$, and let $G$ be the fiber passing through $p$.  Thus the normalization of $\Gamma$ is an elliptic curve. Blowing up $p$ and taking the proper transforms $D_1 = G -C$ of $G$ and $D_2 = \Gamma - 2C$ of $\Gamma$ gives two curves $D_1$, $D_2$ with $D_i^2 =-1$. Here $D_1\cdot C =1$ and $D_2 \cdot C =2$. Then 
$$L = K_{\hY} + D_1 + D_2 = -F + 2F -C + \Gamma -2C + C = F+D_2,$$
so that $L^2 = 1$, $L\cdot D_1 = L\cdot D_2 = 0$, and $L$ is nef and big. As stated, $\hY$ has two simple elliptic singularities. Further degenerating this construction (for example, letting $G$ be a singular fiber) leads to cusp singularities. 

Note that, if $A$ is an exceptional curve on $Y_0$, then $A\cdot F =1$ and hence $A\cdot G = 2$.  Thus the exceptional curves give bisections of the elliptic fibration. Since $\Gamma\cdot F =1$, it is natural to try $\Gamma = E + 2F$ for some exceptional curve $E$. Taking $E= E_1$, we can assume that $F$ is the blowup of a smooth plane cubic $\overline{F}$ in $\Pee^2$ at $9$ points $p_1, \dots, p_9$, with exceptional divisors $E_1, \cdots E_9$ and $\Gamma\cdot E_1 = 1$, $\Gamma \cdot E_i = 2$ for $i > 1$. Then $\Gamma = dH - E_1 - \sum_{i=2}^92E_i$ where $H$ is the pullback of a line in $\Pee^2$, so 
$$d^2 - 1 - 32 = 3, \quad  \text{ i.e.\ }  \quad d=6.$$
Thus $\Gamma$ is the proper transform of a sextic passing through $p_1$ simply and with double points at $p_2, \dots, p_9$ and at a remaining point not on $\overline{F}$. Note that this agrees with 
$$\Gamma \cdot F = 3d -1 -16 =18-17 = 1.$$
\end{example} 

\begin{remark}
Starting with  a $\Gamma\in |E+2F|$, the reconstruction of $\hY$ requires a   singular point of $\Gamma$ not on $F$. 
By Riemann-Roch, $\dim |\Gamma| = 2$, so the expectation is that there is  a $1$-dimensional family of singular $\Gamma$. However, write $Y_0$ as the blowup of $\Pee^2$ at $9$ points $p_1, \dots, p_9$, with corresponding exceptional divisors $E_1, \dots, E_9$,  and take $E=E_1$. Then $E+F$ is effective since it is the proper transform of the linear system of cubics passing through $p_2, \dots, p_9$. Thus there is a one-dimensional family of singular elements of $|\Gamma|$ of the form $N + F$, where $N = 3H - \sum_{i=2}^9E_i$ is the proper transform of a smooth cubic in $\Pee^2$ passing through $p_2, \dots, p_9$.  A general such   $N + F$, where $N$ is smooth,  is a union of the two elliptic curves $N$ and $F$ meeting transversally at a point of $F$ (the base point $q$ of $|N|$, which is some point of $F$ not equal to $p_1$). Of course, the points $p_2, \dots, p_9, q$ are not general. This gives one component of the locus of singular elements of $|\Gamma|$ but not the locus coming from $\hY$.
\end{remark}

Our first goal is to show that all such $\hY$ arise as in Example~\ref{ex11}. We begin with the following more general result:

\begin{lemma}\label{11lemma} Suppose that the minimal resolution $\hY$  of the I-surface $Y$ is a rational surface  and the  divisors $D_1$ and $D_2$  both have multiplicity one. 
Then: 
\begin{enumerate}
\item[\rm(i)] There exist exceptional curves $C_i$ on $\hY$, $i=1,2$ such that $C_i \cdot D_i = 2$ and  $C_i\cdot D_j =1$ for $i\neq j$;
 \item[\rm(ii)] If $\rho_i\colon \hY \to X_i$ is  the contraction of $C_i$, then $X_i$ is a rational  elliptic surface with a multiple fiber $F_i$ of multiplicity $2$, not the image of $D_1$ or $D_2$;
  \item[\rm(iii)] For $j\neq i$, the image $G_i$ of $D_j$ is a smooth elliptic curve or cycle of rational curves, disjoint from $F_i$, and $D_j = G_i-C_i=2F_i-C_i$. 
   \item[\rm(iv)] The image $\Gamma_i$ of $D_i$ is a curve of arithmetic genus $2$ with a node or cusp at the unique intersection point $p_i$ of $\Gamma_i$ and $G_i$. Moreover $\Gamma_i^2=3$ and $\Gamma_i \cdot F_i = D_i\cdot M_i = 1$. 
   \end{enumerate}
\end{lemma}
\begin{proof} Let $M_1 = K_{\hY} +D_2$ and $M_2 = K_{\hY} +D_1$.  Thus $M_i\cdot D_i = 1$ and $M_i\cdot D_j =0$, $j\neq i$. Moreover, $|2M_i|$ defines an elliptic fibration $\tilde\mu_i\colon \hY \to \Pee^1$ with exactly one multiple fiber  $F_i$, of multiplicity $2$.  Let $\mu_i\colon X_i \to \Pee^1$ be the corresponding relatively minimal fibration. Then $X_i$ is a rational elliptic surface with exactly one multiple fiber  $F_i$, of multiplicity $2$. Since $K_{X_i} = -F_i$ and $K_{X_i}^2 =0$,   $\rho_i\colon \hY \to X_i$ is the blowup of $X_i$ at a single point $p_i$, with corresponding exceptional divisor $C_i$. Since $M_i\cdot D_j =0$ for $j\neq i$ and $D_j$ has arithmetic genus one, $D_j$ is contained in a fiber of $\tilde\mu_i\colon \hY \to \Pee^1$ and its image $G_i$ in $X_i$ is a complete fiber of $\mu_i$. Hence, as $D_j^2=-1$,  $C_i\cdot D_j =1$ for $j\neq i$.  
 Also, $M_i \cdot C_i =0$. Then $K_{\hY} = -F_i + C_i = -M_i + C_i$. Thus 
\begin{align*}
1= K_{\hY} \cdot D_i = ( -M_i + C_i)\cdot D_i &= -1 + (C_i \cdot D_i);\\
1 = K_{\hY}\cdot D_j = (-M_i+ C_i) \cdot D_j &= C_i\cdot D_j.
\end{align*}
Thus $C_i \cdot D_i = 2$ and we recover the fact  that $C_i\cdot D_j =1$ for $j\neq i$.

 As noted above, if  $G_i$ is the image of $D_j$, then $G_i$ is a complete fiber, hence is a smooth elliptic curve or cycle of rational curves and $D_j$ is the blowup of $G_i$ at a smooth point.  Since $C_i \cdot D_i = 2$, the image $\Gamma_i$ of $D_i$ is a curve of arithmetic genus $2$ with a node or cusp at the unique intersection point $p_i$ of $\Gamma_i$ and $G_i$. Thus    $\Gamma_i^2=3$ and $\Gamma_i \cdot F_i = D_i\cdot M_i = 1$.  Since $G_i\cdot \Gamma_i =2$ but $F_i\cdot \Gamma_i=1$, $G_i\neq F_i$ and hence $G_i$ is disjoint from $F_i$. Thus $G_i = 2F_i$ and  $D_j = G_i-C_i=2F_i-C_i$. This proves all of the statements (i)--(iv) of the lemma. 
\end{proof}

\begin{proposition}\label{11case}  With assumptions as in Lemma~\ref{11lemma}, suppose that $D_1$ and $D_2$ are irreducible. Then  there exist:
\begin{enumerate}
\item[\rm(i)] A rational elliptic surface $Y_0$ with a multiple fiber $F$ of multiplicity $2$;
\item[\rm(ii)] An exceptional curve $E$ on $Y_0$;  
\item[\rm(iii)] An irreducible  curve $\Gamma\in |E+2F|$  with $p_a(\Gamma) =2$ which has a node or cusp singularity at a point $p\in \Gamma$ which does not lie on $E$,
\end{enumerate} 
such that $\hY$ is the blowup of $Y_0$ at $p$, with $D_2$ the proper transform of $\Gamma$ and $D_1$ the proper transform of an irreducible fiber of $Y_0$ passing through $p$.
\end{proposition}
\begin{proof}
Referring to Lemma~\ref{11lemma}, fix for example $i=2$ and let $\mu_2\colon X_2 \to \Pee^1$ be the corresponding rational elliptic surface. We omit the subscript $2$ for $F$, $\Gamma$, $C$.  Then $\alpha = \Gamma -2F$ is a numerical exceptional curve on $X_2$, because 
 $\alpha\cdot F  =1$ and hence $\alpha^2 = K_{X_2}\cdot \alpha =-1$. By a standard Riemann-Roch argument, $\alpha$ is effective, say $\alpha = E$. Thus  $\Gamma = E + 2F $, $D_2 = E + 2F - 2C$, $D_1 = 2F -C$, and $K_{\hY} = -F + C$, and so 
\begin{align*}
M_1 &= K_{\hY} + D_2 =  -F +C  +  E + 2F - 2C = E+F -C;\\
M_2 & = K_{\hY} + D_1 = -F +C + 2F -C = F;\\
L &= -F +C + 2F -C + E + 2F - 2C = 3F + E - 2C.
\end{align*} 

\begin{lemma} Suppose that  $D_1$ and  $D_2$ are irreducible.
\begin{enumerate}
\item[\rm(i)]  The effective curve  $E$ is  irreducible and hence is an exceptional curve   on $X_2$. Moreover, the singular point of $\Gamma$ does not lie on $E$.
\item[\rm(ii)] The multiple fiber $F$ is irreducible, and every non-multiple fiber of $\mu_2$ has at most two components, and in particular is reduced.
\end{enumerate} 
\end{lemma} 
\begin{proof} (i) Since $E\cdot F =1$ and $F$ is nef, by the Hodge index theorem there are irreducible curves $E_0, A_i$, a nonnegative integer $m$,  and positive integers $n_i$ such that  
$$E = E_0 + \sum_in_iA_i + mF,$$
where $E_0\cdot F = 1$,  $A_i\cdot F =0$ for all $i$, and $A_i^2 < 0$. In particular, since $X_2$ is an anticanonical surface, $E_0 ^2 \ge -1$, $A_i$ is smooth rational and $A_i^2 =-2$. Hence $L\cdot \rho_2^*A_i > 0$ for every $i$, where $\rho_2\colon \hY \to X_2$ is the blowup morphism, by Assumption~\ref{assumption27} on $Y$ and $\omega_Y$.  Then 
$$\rho_2{}_*L \cdot A_i = (3F+ E) \cdot A_i = E\cdot A_i >0.$$
Since $E\cdot F = 1$ and $E\cdot E_0 \ge E_0^2 + m \ge m-1$,
$$-1 = E^2 \ge E\cdot E_0 + \sum_in_i + m \ge 2m-1 + \sum_in_i.$$
This is only possible if $m = 0$ and there are no curves $A_i$, i.e.\ $E= E_0$ for an irreducible curve $E_0$ with $E_0^2 =-1$ and $K_{X_2}\cdot E_0 =-1$. Thus $E=E_0$ is an exceptional curve. Since $E\cdot \Gamma = 1$, $E\cap \Gamma$ is a smooth point of $\Gamma$.

\smallskip
\noindent (ii) Arguing as in (i),  if $A$ is a fiber component of $\mu_2$, $(3F+ E) \cdot A > 0$   hence $E\cdot A >0$. In particular, $E$ meets every fiber component. But $E\cdot F =1$ and thus $E$ is a bisection of the ruling. It follows that $F$ can have at most one component and the remaining fibers can have at most two. This concludes the proof of the lemma. 
\end{proof} 

To complete the proof  of Proposition~\ref{11case}, the surface  $\hY$ is a blowup of the elliptic surface $Y_0=X_2$ at the singular point $p$ of $\Gamma \in |E+2F|$, $D_1$ is the proper transform of an irreducible fiber containing $p$, and $p\notin E$ as required. 
\end{proof}

\begin{proposition}\label{11irred}  In the notation of Proposition~\ref{11case}, suppose that $D_1$ and $D_2$ are irreducible and that all fibers of $Y_0$ are nodal. Then $$H^2(\hY; T_{\hY}(-D_1 - D_2)) = 0.$$ 
Hence the map $\mathbb{T}^1_Y \to H^0(Y; T^1_Y)$ is surjective.
\end{proposition}
\begin{proof} In the notation of  Proposition~\ref{11case}, $K_{\hY} = -F + C$, where $C$ is the exceptional divisor of the blowup $\hY \to Y_0$ at the point $p\in \Gamma$,  and 
$$D_1 + D_2 = 2F - C + E +  2F -2C = 4F + E - 3C.$$
Then $H^2(\hY; T_{\hY}(-D_1 - D_2))$ is Serre dual to
$$H^0(\hY; \Omega^1_{\hY}\otimes L) = H^0(\hY; \Omega^1_{\hY}(3F + E - 2C)) = H^0(\hY; \Omega^1_{\hY}(f+F + E - 2C)).$$
A local calculation shows that  $\rho_2{}_*  \Omega^1_{\hY}(-2C_2) \subseteq   \mathfrak{m}_p \otimes  \Omega^1_{Y_0}$.  Thus
$$H^0(\hY; \Omega^1_{\hY}\otimes L) \subseteq  H^0(Y_0; \mathfrak{m}_p \otimes \Omega^1_{Y_0}(f+ F + E) ).$$
By  Theorem~\ref{tangbundell},  there is an exact sequence
$$0 \to \mu ^*\Omega^1_{\Pee^1}(F ) \to \Omega^1_{Y_0} \to I_Z \otimes \scrO_{Y_0}(f) \to 0,$$
where $Z$ is a reduced $0$-dimensional subscheme whose support is exactly the singular points of the reductions of the fibers and  $\ell(Z) = 12$ since $Y_0$ is a rational elliptic surface and by the assumption that all fibers are nodal. Note that $p$ is not contained in the support of $Z$ since $D_1$ is the proper transform of a fiber at a smooth point.  Hence there is an exact sequence
$$0 \to \mathfrak{m}_p \otimes \mu ^*\Omega^1_{\Pee^1}(F ) \to \mathfrak{m}_p \otimes \Omega^1_{Y_0} \to\mathfrak{m}_p \otimes I_Z \otimes   \scrO_{Y_0}(f) \to 0,$$
As  $\mu^*\Omega^1_{\Pee^1} = \scrO_{Y_0}(-2f)$, taking the tensor product with $\scrO_{Y_0}(f+ F + E)$ gives the following exact sequence:
$$0 \to \mathfrak{m}_p \otimes  \scrO_{Y_0}(E) \to  \mathfrak{m}_p \otimes  \Omega^1_{Y_0}(f+ F  + E) \to \mathfrak{m}_p \otimes   I_Z\otimes \scrO_{Y_0}(2f + F +E)   \to 0.$$
Since $E$  is an exceptional curve, $H^0(Y_0; \scrO_{Y_0}(E))$ has dimension one. Since $p\notin E$,    $$H^0(Y_0;  \mathfrak{m}_p \otimes \scrO_{Y_0}(E) ) =0.$$ As for the remaining term, we have
$$H^0(Y_0;  \mathfrak{m}_p \otimes   I_Z\otimes \scrO_{Y_0}(2f + F  +E) )\subseteq H^0(Y_0; I_Z\otimes \scrO_{Y_0}(2f + F  +E) ).$$   As in the proof of Proposition~\ref{21irred}, the blowup $\widetilde{Y}_0$ of $Y_0$ at two points is a double cover  $\nu\colon \widetilde{Y}_0 \to \widetilde{\mathbb{F}}_1$, where $\widetilde{\mathbb{F}}_1$ is the blowup of $\mathbb{F}_1$ at two points. The line bundle  $\scrO_{Y_0}(2f + F  +E)$ pulls back to a line bundle on $\widetilde{Y}_0$ contained in $\nu^*\scrO_{\widetilde{\mathbb{F}}_1}(2f+ e + d_1+d_2 + \sigma_0)= \nu^*\scrO_{\widetilde{\mathbb{F}}_1}(\sigma_0 +3f-e)$.  As $B_0 \equiv 4\sigma_0 + 3f + 2d_2$, we have
$$(\sigma_0 +3f-e) \cdot B_0 = (\sigma_0 +3f-e) \cdot (4\sigma_0 + 3f + 2d_2) = 11 < 12.$$
Then arguments along the lines of the proof of Proposition~\ref{21irred} or Theorem~\ref{thm17}  show  that $$H^0(Y_0; \mathfrak{m}_p \otimes I_Z\otimes \scrO_{Y_0}(2f + F  +E)) =0$$ and hence that $H^2(\hY; T_{\hY}(-D_1 - D_2)) = 0$. 
\end{proof} 

 \begin{remark} It is possible to give partial results in the case where $D_1$ or  $D_2$ is reducible, by using a slight generalization of Theorem~\ref{tangbundelllog} where the component $E_1$ is allowed to have a node. We omit the details.
 \end{remark}

\section{Geometry of certain elliptic ruled surfaces}\label{section7}

It is well-known that an algebraic elliptic surface $X$ i.e.\ a relatively minimal fibration over a base curve whose general fiber is an elliptic curve, is also an elliptic ruled surface  if and only if the base curve is $\Pee^1$, all fibers have smooth reduction, and $X$ is either a product surface $E\times \Pee^1$ or the log transform of $E\times \Pee^1$ at  two fibers with the same multiplicities $m$, or the log transform of $E\times \Pee^1$ at three multiple fibers, all of multiplicity $2$ (with a condition on the torsion points used to define the log transforms, cf.\ \cite[I.3.23 and I.6.12]{FM}). We describe these surfaces and their geometry in some detail.

\subsection{The case of two multiple fibers}\label{subsection11} Let $E$ be an elliptic curve and $e\in E$ a point of order $2$. Define $X_0 = (E\times \Pee^1)/\iota$, where $\iota(x,t) = (x+e, -t)$. Thus there is an induced map $\psi \colon X_0 \to \Pee^1/\{t\mapsto -t\}$. The general fiber of $\psi$ is $E$, but over $0$ and $\infty$ there are two multiple fibers $\sigma_1, \sigma_2$ isomorphic to $E/(e)=B$. There is also the map $\rho\colon X_0 \to E/(e) = B$. Since $E\to B$ is \'etale, $\rho$ is a fiber bundle with all fibers isomorphic to $\Pee^1$. Note in fact that there is a $(\Zee /2\Zee)\times (\Zee /2\Zee)$-action on $E\times \Pee^1$ generated by   $\iota_1$ and $\iota_2$ where $\iota_1(x,t) = (x+e, t)$ and $(\iota_2(x,t) = (x, -t)$; thus $\iota = \iota_1\circ \iota_2$. Hence $X_0$ is also a double cover of $B\times \Pee^1$ branched along two fibers $(B \times \{0\} )+ (B \times \{\infty\})=\pi_2^*\scrO_{\Pee^1}(2)$, but with the square root of this divisor of the form $\pi_1^*\eta \otimes  \pi_2^*\scrO_{\Pee^1}(1)$ for a nontrivial $2$-torsion line bundle on $B$. The map $X_0 \to B\times \Pee^1$ is given by the maps $(\rho, \psi)$.  Of course, a similar construction works if $e$ is a point of order $m \ge  1$ and we consider the $\Zee/m\Zee$-action on $E\times \Pee^1$ generated by $(x,t) \mapsto  (x+e, \zeta t)$, where $\zeta$ is a primitive $m^{\text{\rm{th}}}$ root of unity. 

As a ruled surface, $X_0$ has the two disjoint sections $\sigma_1, \sigma_2$. The normal bundle of $\sigma_i$ in $X_0$ is $ \eta$, where $\eta$ is a $2$-torsion line bundle, since $\sigma_i$ is also a multiple fiber of multiplicity $2$. Thus $X_0 =\Pee(\scrO_B \oplus  \eta)$, where the two sections of $X_0$ correspond to the sections of $\scrO_B  \oplus  \eta$ and $(\scrO_B \oplus  \eta)\otimes \eta =\eta \oplus \scrO_B$. (From this point of view, it is slightly more complicated  to show that $|2\sigma_i|$ is a pencil.)

\subsection{The case of three multiple fibers}\label{subsection12}  Let $B$ be an elliptic curve, fix a point $p_0\in B$,  and let $W$ be the unique nonsplit extension of $\scrO_B(p_0)$ by $\scrO_B$:
$$0 \to \scrO_B \to W \to \scrO_B(p_0) \to 0.$$
Let $\rho\colon X_1=\Pee(W)\to B$ be the corresponding ruled surface over $B$. The section of $W$ defines a section $\sigma_{p_0}$ with $\rho_*\scrO_{X_1}(\sigma_{p_0}) = W$.  Applying $\rho_*$ to the exact sequence
$$0 \to \scrO_{X_1} \to \scrO_{X_1}(\sigma_{p_0}) \to \scrO_{\sigma_{p_0}}(\sigma_{p_0}) \to 0,$$
and identifying $\sigma_{p_0}$ with $B$, there is an  exact sequence 
$$0 \to \scrO_B \to W \to  \scrO_{\sigma_{p_0}}(\sigma_{p_0}) \to 0.$$
Hence $\scrO_{\sigma_{p_0}}(\sigma_{p_0})\cong  \scrO_B(p_0)$. In particular $\sigma_{p_0}^2 = 1$. Using the shorthand notation $p\cdot f$ for the divisor $\rho^*p$ on $X_1$, it follows that 
$$K_{X_1} = -2\sigma_{p_0} + p_0\cdot f,$$
as one checks by adjunction. 

\begin{lemma}\label{lemmam1} {\rm(i)} If $\sigma$ is a section of $X_1$ with $\sigma^2 = 1$, then $\sigma \equiv \sigma_{p_0}$.
 
 \smallskip
 \noindent {\rm(ii)} For every $p\in B$, there are exactly $4$ sections $\sigma$ with $\scrO_{X_1}(\sigma )|\sigma  \cong \scrO_B(p)$ via the natural isomorphism $\sigma\cong B$. We will denote any such section by $\sigma_p$.
\end{lemma}
\begin{proof} This is an easy consequence of the description of sections of $X_1$. For (ii), compare also Lemma~\ref{lemmam4}. 
\end{proof}

\begin{lemma}\label{lemmam2} With notation as above,
\begin{enumerate}
\item[\rm(i)] The divisor $\Gamma_q = 2\sigma_{p_0} - q\cdot f$ is effective $\iff$ $q$ is a $2$-torsion point on $B$ for the group law with origin $p_0$. 
\item[\rm(ii)] With $q$ as in {\rm(i)}, $2\Gamma_q$ is linearly equivalent to $-2K_{X_1}$. 
\item[\rm(iii)] With $q$ as in {\rm(i)}, $|2\Gamma_q|$ is a pencil on $X_1$ defining a morphism $\psi\colon X_1 \to \Pee^1$. The general fiber is a smooth elliptic curve $\phi$, linearly equivalent to $-2K_{X_1}$ and every fiber which is not smooth is one of the three multiple fibers $2\Gamma_{q_i}$ where the $q_i$ are the $2$-torsion points on $B$.
\item[\rm(iv)] With $q$ as in {\rm(i)}, the normal bundle $N_{\Gamma_q/X_1}$ is a nontrivial $2$-torsion line bundle on $\Gamma_q$.
\end{enumerate}
\end{lemma} 
\begin{proof} (i) First note that 
$$H^0(X_1; \scrO_{X_1}(\Gamma_q)) = H^0(B; \rho_*\scrO_{X_1}(2\sigma_{p_0} - q\cdot f)) = H^0(B; \rho_*\scrO_{X_1}(2\sigma_{p_0})\otimes \scrO_B(-q)).$$
Moreover, $ \rho_*\scrO_{X_1}(2\sigma_{p_0})\cong \Sym^2W$. Also, $W\otimes W = \det W \oplus \Sym^2W$.  Since $W$ has rank $2$, 
$W \cong W\spcheck \otimes \det W $, and hence
$$W\otimes W \cong W \otimes  W\spcheck \otimes \det W \cong \mathit{Hom}(W, W) \otimes \scrO_B(p).$$
By a well-known argument, $\mathit{Hom}(W, W)\cong \scrO_B \oplus \eta_1 \oplus \eta_2 \oplus \eta_3$, 
where the $\eta_i$ are the  three nontrivial $2$-torsion line bundles on $B$. Thus 
$$W\otimes W \cong \scrO_B(p_0)\oplus (\scrO_B(p_0)\otimes \eta_1) \oplus (\scrO_B(p_0)\otimes \eta_2) \oplus (\scrO_B(p_0)\otimes \eta_3).$$
It follows that 
$$\Sym^2W\otimes \scrO_B(-q) \cong (\scrO_B(p_0-q)\otimes \eta_1) \oplus (\scrO_B(p_0-q)\otimes \eta_2) \oplus (\scrO_B(p_0-q)\otimes \eta_3).$$
Thus $H^0(X_1; \scrO_{X_1}(\Gamma_q)) \neq 0$ $\iff$ $\scrO_B(q-p_0)\cong \eta_i$ for some $i$ $\iff$ $q-p_0$ is a $2$-torsion point in the corresponding group law on $B$.

\smallskip
\noindent  (ii) If $q$ is a $2$-torsion point, $2q = 2p_0$, and thus  $2\Gamma_q = 4\sigma_{p_0} - 2q\cdot f= 4\sigma_{p_0} - 2p_0\cdot f= -2K_{X_1}$. 

\smallskip
\noindent  (iii) The proof that $|2\Gamma_q|$ is a pencil on $X_1$ is deferred until the next subsection. Since $(\Gamma_q)^2 =0$ and $|2\Gamma_q|$ contains three disjoint effective divisors, $|2\Gamma_q|$ has no base locus or fixed curves and so defines a morphism $\psi$ to $\Pee^1$. If $\phi$ is a general element, its genus $g(\phi) =1$ since $\phi^2 + \phi\cdot K_{X_1} = 0$. The other statements follow easily from (ii).

\smallskip
\noindent  (iv) This is clear since $\Gamma_q$ is the reduction of a multiple fiber with multiplicity $2$.
\end{proof} 

In terms of numerical equivalence $\equiv$, $\operatorname{Num}X_1 = \Zee \cdot \sigma_{p_0} \oplus \Zee\cdot f$. Thus $\sigma_p \equiv  \sigma_{p_0}$ for all $p\in B$ and $\Gamma_q \equiv \Gamma_{q'}$ for all $q, q' \in B$. Also, $\sigma_p^2 = \sigma_p \cdot \sigma_{p'} = 1$ for all $p, p'\in B$, $\Gamma_q^2 = \Gamma_q\cdot \Gamma_{q'} = 0$ for all points $q, q'$ on $B$ and $\Gamma_q\cdot \sigma_p = 1$. If $\phi$ is the fiber, then $\phi\cdot  \sigma_p  = 2 (\Gamma_q\cdot\sigma_p) =2$ and hence the $\sigma_p$ are bisections of the fibration $\psi$. 

\begin{lemma}\label{lemmam3}   If $X=\Pee(V)$ is a geometrically ruled surface over an elliptic curve $B$ and there exist two irreducible curves $\sigma$, $\Gamma$ on $X$ with $\sigma^2 =1$, $\Gamma^2 = 0$ and $\Gamma$ is not a fiber of the ruling, then $X \cong X_1=\Pee(W)$ as defined at the beginning of \S\ref{subsection12}. In this case, $\sigma = \sigma_p$ for some $p\in B$ and $\Gamma \equiv c\Gamma_q$ for some positive integer $c$ and some point $q \in B$.
\end{lemma}
\begin{proof} First, the intersection form on $H^2(X; \Zee)$ is odd. By the classification of rank $2$ bundles over an elliptic curve, there are two possibilities: Either $V$ is either unstable of odd degree, hence after twisting by a line bundle is of the form $\scrO_B \oplus \lambda$, where $\deg \lambda = e$ is odd and positive, or $V\cong W\otimes \lambda$ for some line bundle $\lambda$ on $B$. If $V$ is unstable, by a standard calculation \cite[V  2.20]{Hartshorne}, every irreducible curve of nonnegative self-intersection  which is not the class of a fiber   has self-intersection  $\ge e \ge 1$, so this case does not arise. Thus, after a twist, $V\cong W$, so $X \cong X_1$ and has invariant $e= -1$. By \cite[V  2.21]{Hartshorne}, if $C$ is the class of an irreducible curve and not the class of a fiber, then either (i) $C\equiv \sigma_{p_0}+bf$ with $b \ge 0$ and $C^2=2b+1\ge 1$, with equality $\iff$ $b=0$,  or (ii) $C \equiv a\sigma_{p_0} + bf$ with $a > 1$ and $b\ge -\frac12a$. In  Case (ii), $C^2 = a^2 + 2ab = a(a+2b)\ge 0$, with equality holding $\iff$ $b=-\frac12a $. Hence, in Case (ii),  if $C^2 >0$, then $C^2\ge a > 1$. Thus, if $\sigma^2 =1$, then   $\sigma \equiv \sigma_{p_0}$ and so $\sigma = \sigma_p$ for some $p\in B$. Moreover,   $\Gamma \equiv a(\sigma_{p_0} - \frac12f)$. Since $\Gamma$ is an integral class, $a$ is even and $\Gamma \equiv \frac{a}2(2\sigma_{p_0} -f)  \equiv   c\Gamma_q$ for a positive integer $c$. 
\end{proof} 

\subsection{$X_1$ as a symmetric product}\label{subsection13}  Fixing the  point $p_0$ on $B$ defines a group law. The addition map $B\times B \to B$ then defines a morphism $\rho\colon \Sym^2B \to B$. The fiber $f_x= \rho^{-1}(x)$ is the set of all unordered pairs $(e_1, e_2)$ with $e_1+ e_2 \equiv x+ p_0$, and hence is $\cong \Pee^1$. Note that $\nu \colon B\times B \to \Sym^2B $ is a double cover of $\Sym^2B$ branched along the diagonal $\Delta$.  We identify $\Delta$ with its image $\overline{\Delta}=\nu(\Delta)$ in $\Sym^2B$.  

For each $a\in B$, we have the image $\sigma^{(a)}$ of $\{a\} \times B$ or equivalently of $B\times \{a\}$. Then $\sigma^{(a)}\cap f_x=\{a,b\}$, where $b\in B$ is the unique point such that $a+b \equiv x+p_0$. In particular, $\sigma^{(a)}$ is a section of $\rho$. For $b\neq a$, $ \sigma^{(a)} \cap \sigma^{(b)} =\{a,b\}$, and the intersection is clearly transverse since $\{a\} \times B$ and $B\times  \{b\} $ meet transversally. Moreover, $\scrO_{\Sym^2B}(\sigma^{(b)})| \sigma^{(a)} \cong \scrO_B(a+b)$ (where $a+b$ is viewed as a point on $B$ via addition). In particular,  $\sigma^{(a)}\cdot  \sigma^{(b)} =1$ for all $a\neq b$ and hence for all $a$. By continuity or by considering $\nu^*(\sigma^{(a)}) = (\{a\} \times B)+(B\times \{a\})$, $\scrO_{\Sym^2B}(\sigma^{(a)})| \sigma^{(a)} \cong \scrO_B(2a)$. Note that, for every point $s$ of $\Sym^2B$ of the form $\{a,b\}$ with $a\neq b$, there are two different sections $ \sigma^{(a)}$, $\sigma^{(b)}$ passing through $s$, but for $s\in \overline{\Delta}$ there is just one section. 

By construction, the image $\nu(\Delta) = \overline{\Delta}$ is the branch locus of $\nu$. For $(e_1, e_2) \in B\times B$, we can consider 
$$\Delta+(e_1, e_2) = \Delta + (0,q) = \{(b, b+q): b\in B\} =\{(b_1, b_2): b_2-b_1 = q\},$$
where $q = e_2-e_1$. Note that $\nu(\Delta + (0,q)) =\nu(\Delta + (q,0))$. In general, $\nu^{-1}(\nu(\Delta + (0,q)))$ has two disjoint components $\Delta + (0,q)$ and $\Delta + (q,0)$. However, if $2q=0$ so that $-q =q$, then $\Delta + (0,q)=\Delta + (q,0)$ and $\nu$ induces an \'etale double cover from $\Delta + (0,q)\cong B$ to its image $\Gamma_q \cong B/(q)$. Clearly $(\nu^*\Gamma_q)^2 =0$, hence $\Gamma_q^2 =0$. Also, 
$$\nu^*\Gamma_q \cap \nu^*\sigma^{(a)} = \{(a, a+q), (a+q, a)\},$$
with the intersections transverse. Thus $\Gamma_q \cdot \sigma^{(a)} =1$.  Thus, by Lemma~\ref{lemmam3}, we have:

\begin{lemma}\label{lemmam4} As a geometrically ruled surface, $\Sym^2B \cong X_1$. Under this isomorphism,  the sections $\sigma^{(a)}$ correspond to sections $\sigma_p$ with $p = 2a$ and $\Gamma_q$ as defined above corresponds to   one of the $\Gamma_{q_i}$ as previously defined. \qed
\end{lemma}

We will thus denote $\Sym^2B$ by $X_1$ in what follows. Since $\nu^*\Gamma_q\cap \nu^*f_x = \{(a,a+q): 2a + q =x\}$, with transverse intersections,  $\nu^*\Gamma_q\cdot \nu^*f_x   =4$. Hence we can see directly that $\Gamma_q$  is a bisection of $\rho$. Likewise, $\Delta\cdot \nu^*f_x   =4$. Since $\Delta = 2\nu^*\overline{\Delta}$, $\overline{\Delta}$ is a $4$-section of $\rho$. In fact, if $\lambda$ is the line bundle defining the double cover $B\times B \to X_1$, since $\nu^*(K_{X_1}\otimes \lambda) = K_{B\times B} =\scrO_{B\times B}$, $\lambda = -K_{X_1}$ up to $2$-torsion, hence $\overline{\Delta} =-2K_{X_1}$. Let $\delta \colon B \times B \to B$ be the difference map: $\delta(a,b) = b-a$. From the diagram
$$\begin{CD}
B\times B @>{\delta}>> B \\
@V{\nu}VV @VVV \\
X_1 = \Sym^2B @>{\psi}>> B/\pm 1 \cong \Pee^1,
\end{CD}$$
it follows that there is a fibration $X_1\to \Pee^1$ whose general fiber is an elliptic curve $\cong B$. Here, since $\Delta$ is the branch locus of $\nu$, $\overline{\Delta}$ is a smooth fiber $\phi_0$ of $\psi$, and the $\Gamma_q$ are three multiple fibers of multiplicity $2$. In particular, $B\times B \to X_1$ is the double cover of $X_1$ branched along $\phi_0$.  This completes the proof of Lemma~\ref{lemmam2}(iii). 

\begin{remark}
The three points $p_1, p_2, p_3\in \Pee^1$ which are the images of $\psi(\Gamma_{q_i})$ plus the image of $\phi_0$ define a double cover of $\Pee^1$ branched along those points which is isomorphic to $B$, so this specifies the fiber $\phi_0$. The corresponding double cover of $X_1$ is $B\times B$. Deforming $\phi_0$ to a general fiber gives a new double cover of $X_1$ which is  $B \times B'$ for a general elliptic curve $B'$.
\end{remark}

\subsection{$X_0$ as an elementary transformation of $X_1$}\label{subsection14} Let $x$ be a general point of $X_1$, so that there are exactly two sections $\sigma_{p_1}$ and $\sigma_{p_2}$ containing $x$. Let $f$ be the fiber through $x$. Then make the elementary modification $X_1'$ of $X_1$ at $p$: blow up $p$ and blow down the proper transform $f'$ of $f$. On $X_1'$, the proper transforms $\sigma_{p_1}'$ and $\sigma_{p_2}'$ of $\sigma_{p_1}$ and $\sigma_{p_2}$ are two disjoint sections with $(\sigma_{p_1}')^2 = (\sigma_{p_2}')^2 =0$. Identifying $\sigma_{p_1}$ and $\sigma_{p_2}$ with $B$, the corresponding normal bundles of $\sigma_{p_1}'$ and $\sigma_{p_2}'$ are $\lambda$ and $\lambda^{-1}$ for some line bundle $\lambda$ of degree $0$ on $B$. Explicitly, if $\sigma_{p_1}= \sigma^{(a)}$ and $\sigma_{p_2}= \sigma^{(b)}$ in the previous notation, then $x =\{a,b\} \in \Sym^2B$.  It is easy to check that $\lambda =\scrO_B(a-b)$ and hence $\lambda^{-1} =\scrO_B(b-a)$. Thus the elementary modification $X_1'\cong \Pee(\scrO_B \oplus \lambda)$. Note that, if $p$ lies on the branch divisor, then there is only one section of $X_1$ passing through $p$ and hence there is a unique section of square $0$. In this case,  $X_1'\cong \Pee(\mathcal{E})$, where $\mathcal{E}$ is the unique nontrivial extension of $\scrO_B$ by $\scrO_B$. 

Now suppose that $x\in \Gamma_q$.  Then in particular $x\notin \overline{\Delta}$, so there are indeed two sections passing through $x$. Moreover, since $\Gamma_q$ is the image of $\Delta + (0,q)$, $x =\{a,a+q\}$ and $\lambda$ is $2$-torsion. Thus $X_1' \cong X_0$. 

\section{The case $\kappa(\hY) =-\infty$: the elliptic ruled case}\label{section8}

By Proposition~\ref{211case}, the possibilities for $(m_1, m_2, m_3)$ are either $(2,1,1)$ or $(2,1,1)$. We begin by analyzing these cases in more detail and then describe the relevant deformation theory.

\subsection{The case $(m_1, m_2, m_3) = (2,1,1)$:} To construct examples, begin with the surface $X_1$ of \S\ref{subsection12}. Choose the following configuration of curves $\Gamma, \sigma_2, \sigma_3$ on $Y_0$: $\Gamma$ is a smooth bisection with $\Gamma^2 =0$,  and $\sigma_2$ and $\sigma_3$ are two sections with $\sigma_i\equiv \sigma_{p_0}$. Thus    $\sigma_i^2 = \sigma_i\cdot \sigma_j = \sigma_i\cdot \Gamma = 1$. Moreover, we can assume that   $\Gamma \cap \sigma_i = p_i$, $i=2,3$, and $\sigma_2\cap \sigma_3 = p_1$, where the points $p_1, p_2, p_3$ are all distinct. Thus $\Gamma + \sigma_2 + \sigma_3$ is combinatorially a triangle. Let $\hY$ be the blowup of $Y_0$ at $p_1, p_2, p_3$, with corresponding exceptional divisors $C_1, C_2, C_3$. Let $D_1 = \Gamma- C_2 -C_3$, $D_2 = \sigma_2 - C_1 -C_3$, and $D_3 =   \sigma_3 - C_1 -C_2$. Then $D_i^2 = -1$, $i=2,3$ and $D_1^2 =-2$. One calculates that, if $L = K_{\hY} + D_1 + D_2 + D_3$, then
\begin{align*}
L &\equiv 2\sigma_0 - C_1-C_2 - C_3; \\
L^2 &= 4 -3 = 1; \\
L \cdot D_i &= 0 \text{ for all $i$};\\
L&\equiv D_1 + D_2 + C_1.
\end{align*}
Since $L\cdot D_2 = L\cdot D_3 =0$ and $L\cdot C_1 = 1$, $L$ is nef and big.
Also note that 
\begin{align*}
M_1 &= K_{\hY} + D_2 + D_3  \equiv f - C_1 = C_1';\\
M_i & = K_{\hY} + D_1 + D_i  \equiv \sigma_0 - C_i, i\neq 1,
\end{align*}
where $C_1' = f-C_1$ is the proper transform of the fiber through $p_1$, hence is an exceptional curve with $D_1 \cdot C_1' = \Gamma \cdot f = 2$,   
  $(\sigma_0 - C_i)^2 =0$,  and $\sigma_0 - C_i$ is nef. 
  
  \begin{proposition} Suppose that $\hY$ has three elliptic singularities with $(m_2, m_2, m_3) = (2,1,1)$. Then $\hY$ is as constructed above.
\end{proposition}
\begin{proof} Note that, in this case, $K_{\hY}^2 =1-4= -3$. Thus any contraction of $\hY$ at three disjoint exceptional curves is geometrically ruled.

Let $f$ be a general  fiber of the Albanese map $\rho\colon\hY \to B$. Then $f\cdot A = 0$ for every exceptional curve $A$ on $\hY$ and hence $A$ is a component of a (reducible) fiber. By Proposition~\ref{211case},  $f\cdot D_2 = f\cdot D_3 =1$.  Consider the morphism $\varphi_2\colon \hY \to \Pee^1$ defined by $|2M_2|$, say. Since $M_2\cdot D_1 = M_2 \cdot D_3 =0$, $D_1$ and $D_3$ are contained in two distinct fibers of $\varphi_2$. Then necessarily $D_1 = F_1 -$ (a sum of exceptional curves) and $D_3 = F_3 -$ (a sum of exceptional curves), where $F_1$ and $F_3$ are numerically equivalent either to $M_2$ or $2M_2$. Since $1= f \cdot D_3 = f\cdot F_3$, $F_3\equiv M_2$ and $f\cdot M_2 =1$. By Theorem~\ref{geomanal}(i),  $ M_1=C$ is an exceptional curve with $C \cdot D_1 = 2$ and $C\cdot D_2 = C\cdot D_3 =0$. Since $C$ is a component of a fiber,    $f\cdot D_1 \ge C\cdot D_1 = 2$.  On the other hand,   
$$f\cdot D_1 = f\cdot F_1\le 2(f\cdot M_2) = 2(f\cdot F_3) = 2.$$
  Thus $f\cdot D_1  = C \cdot D_1 = 2$, $F_1\equiv 2M_2$,  and $C$ is a reduced component of $f$. If  $C'$ is  any   component of the effective  divisor $f -C$, then $C'\cdot D_1=0$,  $C'\cdot D_i \le f\cdot D_i = 1$, $i=2,3$ and if $C'\cdot D_i =1$, then $C'$ is a reduced component of the fiber.   In addition,
$$L\cdot f =  K_{\hY}\cdot f + D_1\cdot f + D_2\cdot f + D_3 \cdot f = -2 + 2 + 1 + 1 =2.$$

Since $D_1^2 = -2$ and $D_3 ^2 = -1$, there exist exceptional curves $A_1, A_1', A_3$ such that $D_1 = F_1 - A_1 - A_1'$ and $D_3 = F_3 - A_3$. Here if $D_1$ is obtained from $F_1$ by taking an infinitely near blowup, we can take $A_1$ to be an exceptional curve and $A_1' = A_1 + A_1''$ where $A_1''$ is a curve disjoint from $D_1$. In any case, neither $A_1$ nor $A_1'$ is equal to $C$ since $A_1\cdot D_1 = A_1'\cdot D_1 = 1$ but $C\cdot D_1 = 2$. Likewise, $A_1$ can't be a component  $C'$ of the fiber containing $C$, since then we would have $f\cdot D_1 = 2 \ge (A_1+ C)\cdot D_1 \geq 3$, and symmetrically for  $A_1'$. Thus $A_1$ and $C$ are components of two different fibers of $\rho$. In particular, $C'$ and $A_1$ are disjoint for every   component $C'$ of the fiber $f$ containing $C$ (including $C$). A similar statement holds for $A_1'$.   Also, $A_1$ and $ A_3$ are disjoint since they are components of two different fibers of $\varphi_2$.  Similarly  $A_1'$ and $A_3$ are disjoint and $D_3$ and $A_1$ are disjoint.

Since $L\cdot f =2$ and $L\cdot A_1 \neq 0$, $L\cdot A_1 $ is either $1$ or $2$. But $D_1\cdot A_1 = 1$,  $D_2\cdot A_1 \le D_2\cdot f = 1$,  and $D_3\cdot A_1 =0$. 
Thus
$$L\cdot A_1 = K_{\hY}\cdot A_1 + D_1\cdot A_1 + D_2\cdot A_1 + D_3 \cdot A_1  = -1+ 1 + D_2\cdot A_1 +0.$$
Hence $L\cdot A_1 = D_2\cdot A_1 = 1$.  Thus $D_1\cdot A_1 = D_2\cdot A_1 = 1$ and $D_3\cdot A_1 =0$.  By the symmetry between $D_2$ and $D_3$, there exists an exceptional  curve $B_1$ such that $D_1\cdot B_1 = D_3\cdot B_1 = 1$,  $D_2\cdot B_1 =0$, and $B_1 \cdot C' =0$ for every   component $C'$ of the fiber $f$ containing $C$. Next, we claim that, possibly after replacing $A_1$ by $A_1'$, we can assume that $A_1$ and $B_1$ are disjoint: If $A_1$ and $B_1$ are not disjoint, then they are components of the same fiber of $\rho$.  Then $A_1+B_1 \equiv f$ is a complete fiber since $A_1$ and $B_1$ are both exceptional curves. In this case, $A_1'$ must be an exceptional curve disjoint from $A_1$ and $B_1$, and we can replace $A_1$ by $A_1'$. The upshot is that  there exist three disjoint exceptional curves $A_1, B_1, C$ with $A_1 \cdot C' =B_1 \cdot C' =0$ for every   component $C'$ of the fiber $f$ containing $C$.  Contracting the  exceptional curves $A_1, B_1, C$ produces a relatively minimal model of $\hY$, which is therefore geometrically ruled. Thus  the fiber of $\rho$ containing $C$ must be of the form   $C+ C'$, where $C'$ is an exceptional curve disjoint from $A_1, B_1$. Since $f\cdot D_i = 1$, $i=2,3$, $C'\cdot D_i = 1$, $i=2,3$ and $C' \cdot D_1 =0$. Let $X$ be the geometrically ruled surface obtained by blowing down   $A_1, B_1, C'$ and let $\Gamma$, $\sigma_2$, $\sigma_3$ be the images of $D_1$, $D_2$, $D_3$ respectively. Then $\Gamma$ is a bisection of the ruling and $\sigma_2$, $\sigma_3$ are sections, with $\Gamma^2=0$ and $\sigma_2^2 = \sigma_3^2=1$. By Lemma~\ref{lemmam3}, $X = X_1$ and we obtain $\hY$ as described at the beginning of this section.
\end{proof}
  
 \subsection{The case $(m_1, m_2, m_3) = (1,1,1)$:} To construct examples, begin as before with the surface $X_1$. Let $\Gamma_{q_1}$ and $\Gamma_{q_2}$ be two disjoint bisections as in Section~\ref{section1}, and let $\sigma_3$ be a section meeting $\Gamma_{q_i}$ at the point $x_i$. Blow up $x_i$, $i=1,2$,  and let $D_i$ be the proper transform of $\Gamma_{q_i}$, $i=1,2$, and $D_3$ be the proper transform of $\sigma_3$. If $C_i$ is the exceptional curve corresponding to $x_i$, then in the notation of Section~\ref{section1},   $D_i \equiv 2\sigma_{p_0} -f - C_i$, $i=1,2$, and $D_3 \equiv \sigma_{p_0} - C_1 -C_2$. Then $D_i^2 = -1$, $1\le i \le 3$ and the $D_i$ are disjoint elliptic curves. If $L = K_{\hY} + D_1 + D_2 + D_3$, then
\begin{align*}
L &= K_{\hY} + D_1 + D_2 + D_3 \equiv 3\sigma_{p_0} - f - C_1 -C_2;\\
&\equiv (\sigma_{p_0} -C_1-C_2) + (2\sigma_{p_0} - f) \equiv D_3 + \Gamma_i \equiv D_3 + D_i + C_i.
\end{align*} 
Thus, $L$ is numerically equivalent to an effective divisor, $L\cdot D_i = 0$, $1\le i \le 3$, $L\cdot C_i = 1$, and $L^2 = 9-6 -2 = 1$. So $L$ is nef and big. Also note:
\begin{align*}
  M_i &= K_{\hY} + D_j + D_3 \equiv  \sigma_{p_0} -  C_j \equiv  D_3 + C_i, j\neq i \text{ and }  i,j\neq 3; \\
M_3 &=  K_{\hY} + D_1 + D_2 \equiv  2\sigma_{p_0} -  f.
\end{align*} 
Thus for all $i$, $M_i$ is nef and $M_i^2 =0$. Note that, for $i=1,2$,  $2M_i$ defines the elliptic fibration on $X_0$ as described in \S\ref{subsection11}, using the description of $X_0$ as an elementary modification of $X_1$ given in \S\ref{subsection14}. On the other hand, $2M_3$ defines the elliptic fibration on $X_1$  described in \S\ref{subsection12}. 

\begin{proposition} Suppose that $\hY$ has three elliptic singularities all of multiplicity $1$. Then $\hY$ is as constructed above.
\end{proposition}
\begin{proof}  Let $\rho\colon \hY \to B$ be the Albanese morphism and let $f\cong \Pee^1$ be a general fiber. Since $K_{\hY} ^2 = 1 - 3 =-2$, $\hY$ is the blowup of a  geometrically ruled elliptic surface at $2$ points (\emph{a priori} possibly infinitely near). Let $D_1$, $D_2$, $D_3$ be the three exceptional divisors of the resolution $\pi\colon \hY \to Y$. The main point will be to locate two disjoint exceptional curves $C_1$ and $C_2$ such that, possibly after relabeling the $D_i$, 
\begin{align*}
C_1\cdot D_1 &= C_1 \cdot D_3 = 1,  \quad   C_1\cdot D_2  = 0; \\
C_2\cdot D_2 &= C_2\cdot D_3 = 1, \quad  C_2\cdot D_1  = 0.
\end{align*}
Given the existence of $C_1$, $C_2$,  let $X$ be the surface obtained by contracting $C_1$ and $C_2$ and let $\Gamma_1, \Gamma_2, \sigma_3$ be the images of $D_1, D_2, D_3$ respectively. Then  $\Gamma_1^2 = \Gamma_2^2 =   \Gamma_1\cdot  \Gamma_2 = 0$, $\sigma_3^2 = 1$, and $\Gamma_i\cdot \sigma_3 = 1$. In particular, the numerical equivalence class of $\Gamma_i$ is primitive. By Lemma~\ref{lemmam3},   $X=X_1$, $\Gamma_i = \Gamma_{q_i}$ for appropriate  $2$-torsion points, and $\sigma_3 =\sigma_p$ for some point $p$.

To find the curves $C_i$, consider the divisor $M_3$. Since $M_3\cdot D_1 = M_3 \cdot D_2 =0$, each $D_i$ is contained in a fiber of the morphism $\hY \to \Pee^1$ defined by $2M_3$. Since $D_i$ is connected and $D_i^2 =-1$, the only possibility is that $D_1 = F_1 -A_1$ and $D_2 = F_2 -A_2$ where the $F_i$ are (reductions of) fibers of the relatively minimal fibration  $Y_0^{(3)} \to \Pee^1$ defined by $2M_3$, not necessarily multiple or of the same multiplicity, and $A_1, A_2$ are disjoint exceptional curves. In particular, $A_i\cdot D_j =\delta_{ij}$, $1\le i,j\le 2$.  The fiber $f$ of the morphism $\rho \colon \hY \to B$ satisfies $f =A_1 + A_1' = A_2 + A_2'$, where $A_1'$ and $A_2'$ are also exceptional curves on $\hY$, and that the set of exceptional curves on $\hY$ is exactly $\{A_1 , A_1' , A_2, A_2\}$.  There are $4$ subsets consisting of two disjoint exceptional curves: $\{A_1, A_2\}, \{A_1', A_2\}, \{A_1, A_2'\}, \{A_1', A_2'\}$.
We also have the following: let $C$ be any exceptional curve on $\hY$.
\begin{enumerate} 
\item It is not possible that $D_{j_1}\cdot C = D_{j_2}\cdot C =0$ for two different $j_1, j_2$: if so, let $i\neq j_1, j_2$. Then $M_i \cdot C = K_{\hY}\cdot C =-1$, contradicting the fact that $M_i$ is nef. 
\item Since $D_i$ is not contained in a fiber of $\rho$ (because $p_a(D_i) = 1$), $D_i\cdot f >0$. Hence, if $f = C+C'$ for a pair of exceptional curves and $D_i\cdot C = 0$, then $D_i\cdot C' > 0$.
\item From what we have seen for $i=3$, for every $i$, if $\{j,k\} = \{1,2,3\} -\{i\}$, there exist a pair of disjoint exceptional curves $E_1, E_2$ such that $E_1\cdot D_j = 1, E_1\cdot D_k =0$ and $E_2\cdot D_j = 0, E_2\cdot D_k =1$.
\end{enumerate} 
Consider the following table, where the $(i,j)^{\text{th}}$ entry is the corresponding self-intersection and we have used the above information for $M_3$:

\medskip
\centerline{\begin{tabular}{c||c|c|c|c}
  & $A_1$ & $A_1'$ &$A_2$  & $A_2'$ \\
\hline\hline
$D_1$  & $1$ &   &0 & $>0$ \\
\hline
$D_2$& $0$& $>0$& $1$ &   \\
 \hline
$D_3$& $>0$&  &$>0$   &
\end{tabular}}

\medskip
By the above, no column has two zeroes, and if $D_i\cdot A_j = 0$, then   $D_i\cdot A_j'> 0$ and conversely. 
Applying the same arguments for $M_2$ and then $M_1$, one checks by a tedious argument  that, after possibly relabeling, the column for $A_1$ is $(1,0,1)$ and the column for $A_2$  is $(0,1,1)$ as desired. 
\end{proof}

\subsection{Applications to deformations} The goal now is to prove the following theorem:

\begin{theorem}\label{ellruled}  Let $\hY$ be the minimal resolution of an I-surface with $k=3$.
\begin{enumerate}
\item[\rm(i)] $H^2(\hY; T_{\hY}(-\log D)) = 0$. 
\item[\rm(ii)] The image of $\mathbb{T}^1_Y$ in $H^0(Y; T^1_Y)$ has codimension $2$. 
\item[\rm(iii)] For all $i$, $H^2(\hY; T_{\hY}(-\log D_i')(-D_i)) = 0$, and hence there is a deformation of $Y$ which smooths $p_i$ and is equisingular at the remaining points $p_j$. 
\item[\rm(iv)] If $(m_1, m_2, m_3) = (2,1,1)$, then smoothing $Y$ at the point of multiplicity $2$ yields an I-surface whose minimal resolution is a  blown up Enriques surface with two simple elliptic singularities of multiplicity one and   smoothing $Y$ at a point of multiplicity $1$ yields an I-surface whose minimal resolution is a a rational surface with one    simple elliptic singularity of multiplicity $2$ and another of multiplicity $1$. 
\item[\rm(v)] If $(m_1, m_2, m_3) = (1,1,1)$, then smoothing $Y$ at a point $p_i$ such that $\pi^{-1}(p_i) = \sigma$  is a section yields an I-surface whose minimal resolution is a rational surface  with two simple elliptic singularities of multiplicity one and smoothing $Y$ at a point $p_i$ such that $\pi^{-1}(p_i) = \Gamma$  is a bisection yields an I-surface whose minimal resolution is a  blown up Enriques surface  with two simple elliptic singularities of multiplicity one. 
\end{enumerate}
\end{theorem}
\begin{proof} Since the cases $(m_1, m_2, m_3) = (2,1,1)$ and $(m_1, m_2, m_3) = (1,1,1)$ are formally identical as far as (i)---(iii) are concerned, we shall just  write down  the details for the case $(m_1, m_2, m_3) = (1,1,1)$. In this case $\hY$ is the blowup of a surface $X_1 =\Pee(W)$  where the $D_i$ are the proper transforms of $\Gamma_1$, $\Gamma_2$, $\sigma_3$ respectively, where the $\Gamma_i$ are bisections and $\sigma_3$ is a section. We turn now to the proofs of the various statements.

\smallskip
\noindent (i): By Serre duality, $H^2(\hY; T_{\hY}(-\log D))$ is Serre dual to $H^0(\hY; \Omega^1_{\hY}(\log D)\otimes K_{\hY})$. By the usual Hartogs argument, it is enough to prove that 
$$H^0(X_1; \Omega^1_{X_1}(\log(\Gamma_1 + \Gamma_2 + \sigma_3))\otimes K_{X_1})=0.$$
Applying Serre duality again, we have to show that 
$$H^2(X_1; T_{X_1}(-\log (\Gamma_1 + \Gamma_2 + \sigma_3)))=0.$$
Using the  normal bundle sequence, we get an exact sequence
$$0 \to T_{X_1}(-\log (\Gamma_1 + \Gamma_2 + \sigma_3)) \to T_{X_1} \to N_{\Gamma_1/X_1} \oplus N_{\Gamma_2/X_1} \oplus N_{\sigma_3/X_1} \to 0.$$
Since  $N_{\Gamma_i/X_1}$ is a nontrivial $2$-torsion line bundle on $\Gamma_i$  and     $N_{\sigma_3/X_1}$ has degree $1$, 
$$H^1(\Gamma_i; N_{\Gamma_i/X_1})  =H^1(\sigma_3; N_{\sigma_3/X_1}) =0.$$    Moreover,  $H^2(X_1; T_{X_1})=0$ since $H^2(X_1; T_{X_1})$ is a birational invariant (or directly). Thus 
$$H^2(X_1; T_{X_1}(-\log (\Gamma_1 + \Gamma_2 + \sigma_3)))= H^2(\hY; T_{\hY}(-\log D)) =0.$$ 

\smallskip
\noindent (ii): Using the fact that $\mathbb{T}^2_Y =0$, which holds generally for I-surfaces, it suffices to prove that $H^2(Y; T^0_Y)$ has dimension $2$. By the proof of
 Lemma~\ref{lemma1101} since $H^2(\hY; T_{\hY}) =0$, $H^2(Y; T^0_Y) \cong \operatorname{Coker}\{H^1(\hY; T_{\hY}) \to H^0(Y; R^1\pi_*T_{\hY})\}$. Using the commutative diagram of Proposition~\ref{prop1121} and the fact that 
 $$H^0(D;N_{D/Y}) = H^2(\hY; T_{\hY}(-\log D)) =0,$$
  there is a commutative diagram with exact rows
 $$\begin{CD}
0 @>>>  H^1(\hY; T_{\hY}(-\log D)) @>>> H^1(\hY; T_{\hY}) @>>>  H^1(D;N_{D/\hY}) @>>> 0 \\
@. @VVV @VVV @VV{=}V @. \\
0 @>>> H^1(D; T_D) @>>>  H^0(Y; R^1\pi_*T_{\hY}) @>>> H^1(D;N_{D/\hY}) @>>> 0.
\end{CD}$$
It follows that 
$$\operatorname{Coker}\{H^1(\hY; T_{\hY}) \to H^0(Y; R^1\pi_*T_{\hY}\} \cong \operatorname{Coker}\{H^1(\hY; T_{\hY}(-\log D)) \to H^1(D; T_D) \}.$$
In our case, $T_D \cong T_{\Gamma_1} \oplus T_{\Gamma_2}\oplus T_{\sigma_3}$, and it is easy to check that the image of $H^1(\hY; T_{\hY}(-\log D))$ in $H^1(D; T_D)$ has dimension $1$ (for example by the arguments below for (iii)). 

\smallskip
\noindent (iii): By Proposition~\ref{prop1171}(ii) and (i) above, it suffices to prove that the map $H^1(\hY; T_{\hY}(-\log D)) \to H^1(D_i; T_{D_i})$ is surjective.

First consider the deformation theory of $X_1$. There is the usual exact sequence
$$0 \to T_{X_1/B} \to T_{X_1} \to \rho^*T_B \to 0.$$
Then $R^1\rho_*T_{X_1/B} = R^1\rho_*T_{X_1} =0$ and there is an  exact sequence
$$ 0 \to \operatorname{ad} W \to R^0\rho_* T_{X_1} \to T_B \to 0.$$
The simple bundle $W$ satisfies $H^0(B; \operatorname{ad} W) = H^1(B; \operatorname{ad} W) = 0$. Thus 
$$H^1(X_1; T_{X_1}) \cong H^1(B;  R^0\rho_* T_{X_1}) \cong H^1(B; T_B).$$
This says that $\mathbf{Def}_{X_1} \to \mathbf{Def}_B$ is smooth,  in fact it is  an isomorphism on the analytic germs representing the functors. 

Given an element $\theta\in H^1(X_1; T_{X_1}) \cong  H^1(B; T_B)$, let $\mathcal{X} \to \mathcal{B}$ be the corresponding morphism over $\Spec\Cee[\varepsilon]$. Because $H^1(\Gamma_i; N_{\Gamma_i/X_1})  = H^1(\sigma_3; N_{\sigma_3/X_1}) =0$, the curves $\Gamma_i$ and $\sigma_3$ extend to Cartier divisors $\widetilde{\Gamma}_i$ and $\widetilde{\sigma}_3$ on $\mathcal{X}$ which are flat over $\Spec\Cee[\varepsilon]$ and \'etale over $\mathcal{B}$. Blowing up the intersections  $\widetilde{\Gamma}_i\cap \widetilde{\Gamma}_j$ and $\widetilde{\Gamma}_i\cap \widetilde{\sigma}_3$ leads to a scheme $\widetilde{\mathcal{Y}} \to \mathcal{B}$ with Cartier divisors $\mathcal{D}_i$, $1\le i \le 3$. Thus $\mathcal{D}_i\cong \widetilde{\Gamma}_i$ for $i=1,2$ and $\mathcal{D}_3\cong \widetilde{\sigma}_3$. Hence the element $\hat{\theta} \in H^1(\hY; T_{\hY})$ corresponding to $\widetilde{\mathcal{Y}}$ lies in the image of $H^1(\hY; T_{\hY}(-\log D))$. Its image in $H^1(D_i; T_{D_i}) $ maps onto the class corresponding to $\mathcal{B}$, which is a generator of the one-dimensional vector space $H^1(B; T_B)$. Moreover, $H^1(D_i; T_{D_i}) \cong H^1(B; T_B)$ since the morphisms $D_i \to B$ induced by projection are \'etale. Hence the map $H^1(\hY; T_{\hY}(-\log D)) \to H^1(D_i; T_{D_i})$ is surjective as required.

\begin{remark} In terms of moduli, the exact sequence for $T_{X_1}$ above shows that $H^0(X_1; T_{X_1}) \cong H^0(B; T_B)$. Then in the case $(m_1, m_2, m_3) = (1,1,1)$, there is an exact sequence
$$H^0(B; T_B) \to H^0(\sigma_3; N_{\sigma_3/X_1}) \to H^1(X_1; T_{X_1}(-\log \overline{D})) \to H^1(X_1; T_{X_1})\to 0,$$
where $\overline{D}$ is the divisor $\Gamma_1+ \Gamma_2 + \sigma_3$. It is easy to check that $H^0(B; T_B) \to H^0(\sigma_3; N_{\sigma_3/X_1}) $ is an isomorphism, hence $H^1(X_1; T_{X_1}(-\log \overline{D})) \to H^1(X_1; T_{X_1})$ is an isomorphism. Since the exceptional divisors of $\hY \to X_1$ are determined by $\overline{D}$, the number of moduli of the pair $(\hY, D_1+ D_2 + D_3)$, or equivalently the dimension of the space of equisingular deformations of  $Y$,  is one. The surface $\hY$ depends however on $4$ moduli, because for a small deformation of $\hY$ the blowup points don't have to lie on intersections of the  $\overline{D}$.

There is a similar picture for the case $(m_1, m_2, m_3) = (2,1,1)$, except that in this case we use the configuration $\Gamma + \sigma_2 + \sigma_3$, and the map $H^0(B; T_B) \to H^0(\sigma_2; N_{\sigma_2/X_1})\oplus H^0(\sigma_3; N_{\sigma_3/X_1})$ has a one-dimensional cokernel. Thus the number of moduli of the pair $(\hY, D_1+ D_2 + D_3)$ in this case is two. 
\end{remark} 

\smallskip
\noindent (iv): If $(m_1, m_2, m_3) = (2,1,1)$ and we smooth a point of multiplicity $1$ but keep the other two simple elliptic singularities, then the result $Y_t$ is an I-surface with two simple elliptic singularities, one  of multiplicity $2$ and another of multiplicity $1$. By Theorem~\ref{theorem5}  and Lemma~\ref{kappaminus}, the minimal resolution $\hY_t$ is a rational surface. To deal with the point of multiplicity $2$, recall that $\hY$ is obtained as follows: starting with the configuration $\Gamma, \sigma_2, \sigma_3$, where $\Gamma$ is a bisection and $\sigma_2, \sigma_3$ are sections, blow up the three intersection points. Then the proper transform $D_1$ of $\Gamma$ is the component of $D$ of square $-2$, and the other components $D_2$ and $D_3$ meet an exceptional curve disjoint from $D_1$. Thus, if $\hY_t$ is a deformation obtained by smoothing the point of multiplicity $2$ but keeping the other singularities, we can obtain $\hY_t$ as follows:   Start with the two point blowup $\widehat{Y}$ of $X_1$ obtained by blowing up $\Gamma \cap \sigma_2$ and $\Gamma\cap \sigma_3$. Then   smooth the contraction of the proper transform $D_1$ of $\Gamma$ while keeping the proper transforms $\hat{\sigma}_2$ and $\hat{\sigma}_3$. (Note that the $\hat{\sigma}_i$ are stable submanifolds since $N_{\hat{\sigma}_i/\widehat{Y}}$ is a nontrivial line bundle of degree $0$ on $\hat{\sigma}_i$, by the discussion in \S\ref{subsection14},  and hence $H^1(\hat{\sigma}_i; N_{\hat{\sigma}_i/\widehat{Y}}) = 0$.)  Finally,  blow up the intersection point  of $(\hat{\sigma}_2)_t$ and $(\hat{\sigma}_3)_t$.  The proper transform $D_1$ of $\Gamma$ on   $\widehat{Y}$   is numerically equivalent to $-K_{\widehat{Y}}$, and  $2D_1$ is linearly equivalent to $-2K_{\widehat{Y}}$.  If $\overline{Y}$ is the surface obtained by contracting $\widehat{Y}$ along $D_1$, then $\omega_{\overline{Y}}^{\otimes 2} \cong \scrO_{\overline{Y}}$ and it is a standard argument that a small smoothing of $\overline{Y}$ will have a numerically trivial canonical bundle.  Thus $\hY_t$ is the blowup of a numerically $K$-trivial surface at one point. Hence it is a blown up Enriques surface.

\smallskip
\noindent (v):  Let  $(m_1, m_2, m_3) = (1,1,1)$, so that we begin with the configuration $\Gamma_1, \Gamma_2, \sigma_3$ on $X_1$ and blow up the two intersection points $\Gamma_i\cap \sigma_3$. As before, let $D_i$ be the proper transform of $\Gamma_i$, $i=1,2$ and let $D_3$ be the proper transform of $\sigma_3$. If we smooth a singular point of $Y$ which is the image of $D_1$ or $D_2$, an argument similar to that in (iv) shows that the result is a blown up Enriques surface. To deal with the case where we smooth a singular point of $Y$ which is the image of  $D_3$,   note that the proper transform $f'$ of the fiber of the ruled surface passing through $\Gamma_1\cap \sigma_3$ satisfies: $f'\cdot D_1=1$ and  $f'\cdot D_2 =2$  and is an exceptional curve on $\hY$ not meeting $D_3$. Now smooth the surface obtained by blowing down $D_3$ and let $Y_t$ be the deformed surface,  with minimal resolution $\hY_t$. Since $f'$ is disjoint from $D_3$, it is a stable submanifold  and hence deforms to $\hY_t$. Then there is an exceptional curve on $\hY_t$ which meets $D_2$ in $2$ points. Blowing  down this exceptional curve, the image $\overline{D}_2$ is a curve  of square $3$. In particular, $\hY_t$ cannot be the blowup of an Enriques surface at a point, so it must be a rational surface.
\end{proof}

 \bibliography{Isurfaces}

\providecommand{\bysame}{\leavevmode\hbox to3em{\hrulefill}\thinspace}
\providecommand{\MR}{\relax\ifhmode\unskip\space\fi MR }
\providecommand{\MRhref}[2]{%
  \href{http://www.ams.org/mathscinet-getitem?mr=#1}{#2}
}
\providecommand{\href}[2]{#2}
\begin{thebibliography}{CGL22}

\bibitem[AE23]{AlexeevEngel}
Valery Alexeev and Philip Engel, \emph{Compact moduli of {K}3 surfaces}, Ann.
  of Math. (2) \textbf{198} (2023), no.~2, 727--789. \MR{4635303}

\bibitem[Bea78]{Beauville}
Arnaud Beauville, \emph{Surfaces alg\'{e}briques complexes}, Ast\'{e}risque,
  vol. No. 54, Soci\'{e}t\'{e} Math\'{e}matique de France, Paris, 1978, Avec
  une sommaire en anglais. \MR{485887}

\bibitem[Bri79]{Brieskorn}
Egbert Brieskorn, \emph{Die {H}ierarchie der {$1$}-modularen
  {S}ingularit\"{a}ten}, Manuscripta Math. \textbf{27} (1979), no.~2, 183--219.
  \MR{526827}

\bibitem[CGL22]{CGL}
Xi~Chen, Frank Gounelas, and Christian Liedtke, \emph{Rational curves on
  lattice-polarised {K}3 surfaces}, Algebr. Geom. \textbf{9} (2022), no.~4,
  443--475. \MR{4450621}

\bibitem[FL22]{FL22d}
Robert Friedman and Radu Laza, \emph{{T}he higher {D}u {B}ois and higher
  rational properties for isolated singularities}, arXiv:2207.07566, 2022.

\bibitem[FL23]{FL23}
\bysame, \emph{Deformations of {C}alabi-{Y}au varieties with isolated log
  canonical singularities}, arXiv:2306.03755, 2023.

\bibitem[FM94]{FM}
Robert Friedman and John~W. Morgan, \emph{Smooth four-manifolds and complex
  surfaces}, Ergebnisse der Mathematik und ihrer Grenzgebiete (3) [Results in
  Mathematics and Related Areas (3)], vol.~27, Springer-Verlag, Berlin, 1994.
  \MR{1288304}

\bibitem[FPR15]{FPR1}
Marco Franciosi, Rita Pardini, and S\"{o}nke Rollenske, \emph{Log-canonical
  pairs and {G}orenstein stable surfaces with {$K_X^2=1$}}, Compos. Math.
  \textbf{151} (2015), no.~8, 1529--1542. \MR{3383166}

\bibitem[FPR17]{FPR2}
\bysame, \emph{Gorenstein stable surfaces with {$K^2_X=1$} and {$p_g>0$}},
  Math. Nachr. \textbf{290} (2017), no.~5-6, 794--814. \MR{3636379}

\bibitem[Fri83a]{Friedman83}
Robert Friedman, \emph{A degenerating family of quintic surfaces with trivial
  monodromy}, Duke Math. J. \textbf{50} (1983), no.~1, 203--214. \MR{700137}

\bibitem[Fri83b]{FriedmanSmoothings}
\bysame, \emph{Global smoothings of varieties with normal crossings}, Ann. of
  Math. (2) \textbf{118} (1983), no.~1, 75--114. \MR{707162}

\bibitem[Fri95]{FriedSO3}
\bysame, \emph{Vector bundles and {${\rm SO}(3)$}-invariants for elliptic
  surfaces}, J. Amer. Math. Soc. \textbf{8} (1995), no.~1, 29--139.
  \MR{1273414}

\bibitem[Fri98]{Friedmanbook}
\bysame, \emph{Algebraic surfaces and holomorphic vector bundles},
  Universitext, Springer-Verlag, New York, 1998. \MR{1600388}

\bibitem[Fri15]{Friedanticanon}
\bysame, \emph{On the geometry of anticanonical pairs}, arXiv:1502.02560, 2015.

\bibitem[GSB21]{GSB}
Ian Grojnowski and Nicholas Shepherd-Barron, \emph{Del {P}ezzo surfaces as
  {S}pringer fibres for exceptional groups}, Proc. Lond. Math. Soc. (3)
  \textbf{122} (2021), no.~1, 1--41. \MR{4210255}

\bibitem[Har77]{Hartshorne}
Robin Hartshorne, \emph{Algebraic geometry}, Graduate Texts in Mathematics,
  vol. No. 52, Springer-Verlag, New York-Heidelberg, 1977. \MR{463157}

\bibitem[Hor76]{Horikawa0}
Eiji Horikawa, \emph{Algebraic surfaces of general type with small
  {$c\sp{2}\sb{1}$}. {II}}, Invent. Math. \textbf{37} (1976), no.~2, 121--155.
  \MR{460340}

\bibitem[Hor78]{Horikawa}
\bysame, \emph{On the periods of {E}nriques surfaces. {I}}, Math. Ann.
  \textbf{234} (1978), no.~1, 73--88. \MR{491725}

\bibitem[Kar77]{Karras}
Ulrich Karras, \emph{Deformations of cusp singularities}, Several complex
  variables ({P}roc. {S}ympos. {P}ure {M}ath., {V}ol. {XXX}, {P}art 1,
  {W}illiams {C}oll., {W}illiamstown, {M}ass., 1975), Proc. Sympos. Pure Math.,
  vol. Vol. XXX, Part 1, Amer. Math. Soc., Providence, RI, 1977, pp.~37--44.
  \MR{472811}

\bibitem[Loo81]{Looijenga}
Eduard Looijenga, \emph{Rational surfaces with an anticanonical cycle}, Ann. of
  Math. (2) \textbf{114} (1981), no.~2, 267--322. \MR{632841}

\bibitem[M{\'{e}}r82]{Mer}
J.-Y. M{\'{e}}rindol, \emph{Les singularit\'{e}s simples elliptiques, leurs
  d\'{e}formations, les surfaces de del {P}ezzo et les transformations
  quadratiques}, Ann. Sci. \'{E}cole Norm. Sup. (4) \textbf{15} (1982), no.~1,
  17--44. \MR{672474}

\bibitem[Ser06]{Sernesi}
Edoardo Sernesi, \emph{Deformations of algebraic schemes}, Grundlehren der
  mathematischen Wissenschaften [Fundamental Principles of Mathematical
  Sciences], vol. 334, Springer-Verlag, Berlin, 2006. \MR{2247603}

\bibitem[Sha80]{Shah1}
Jayant Shah, \emph{A complete moduli space for {$K3$} surfaces of degree
  {$2$}}, Ann. of Math. (2) \textbf{112} (1980), no.~3, 485--510. \MR{595204}

\bibitem[Sha81]{Shah2}
\bysame, \emph{Degenerations of {$K3$} surfaces of degree {$4$}}, Trans. Amer.
  Math. Soc. \textbf{263} (1981), no.~2, 271--308. \MR{594410}

\bibitem[Ste83]{Steenbrink81}
Joseph H.~M. Steenbrink, \emph{Mixed {H}odge structures associated with
  isolated singularities}, Singularities, {P}art 2 ({A}rcata, {C}alif., 1981),
  Proc. Sympos. Pure Math., vol.~40, Amer. Math. Soc., Providence, RI, 1983,
  pp.~513--536.

\bibitem[Ste97]{Steenbrink}
\bysame, \emph{Du {B}ois invariants of isolated complete intersection
  singularities}, Ann. Inst. Fourier (Grenoble) \textbf{47} (1997), no.~5,
  1367--1377. \MR{1600383}

\bibitem[Usu89]{Usui}
Sampei Usui, \emph{Type {$\rm I$} degeneration of {K}unev surfaces}, Actes du
  Colloque de Th\'{e}orie de Hodge (Luminy, 1987), Ast\'{e}risque, vol.
  179-180, 1989, pp.~185--243. \MR{1042807}

\bibitem[Wah76]{WahlI}
Jonathan~M. Wahl, \emph{Equisingular deformations of normal surface
  singularities. {I}}, Ann. of Math. (2) \textbf{104} (1976), no.~2, 325--356.
  \MR{422270}

\bibitem[Wah80]{Wahl}
\bysame, \emph{Elliptic deformations of minimally elliptic singularities},
  Math. Ann. \textbf{253} (1980), no.~3, 241--262. \MR{597833}

\end{thebibliography}

\end{document}